\documentclass[a4paper,12pt, twoside]{article}

\newtheorem{thm}{Theorem}[section]
\newtheorem{prop}{Proposition}[section]
\newtheorem{rmk}{Remark}[section]
\newtheorem{lem}{Lemma}[section]
\newtheorem{defi}{Definition}[section]
\newtheorem{algo}{Algorithm}

\usepackage{authblk}
\usepackage{amsmath, amssymb}
\usepackage{amstext}
\usepackage[english]{babel}
\usepackage{xcolor}
\usepackage{amsfonts}
\usepackage{stackengine}
\usepackage{graphicx}
\usepackage{mathtools}
\usepackage{newtxtext}
\usepackage[varvw]{newtxmath}
\usepackage{times}
\usepackage{urwchancal}
\usepackage{bbold}
\usepackage{hyperref}
\usepackage{cleveref}
\usepackage{centernot}
\usepackage{subcaption}
\usepackage[export]{adjustbox}
\usepackage{mathrsfs}
\usepackage[md]{titlesec}
\usepackage[style=numeric-comp, backend=bibtex]{biblatex}
\usepackage{fancyhdr}

\newenvironment{proof}{\paragraph{Proof:}}{\hfill$\square$ \\}

\numberwithin{equation}{section}

\allowdisplaybreaks

\providecommand{\keywords}[1]
{
	\small	
	\textbf{\textit{Keywords:}} #1
}

\providecommand{\pacs}[1]
{
	\small	
	\textbf{\textit{PACS:}} #1
}

\providecommand{\MSC}[1]
{
	\small	
	\textbf{\textit{2010 MSC:}} #1
}

\fancypagestyle{mypagestyle}{
	\fancyhf{}
	\fancyhead[OC]{\textit{J. Doghman} \& \textit{L. Gouden{\`e}ge}}
	\fancyhead[EC]{\textit{Numerical and convergence analysis of the stochastic LANS-$\alpha$ equations}}
	\fancyhead[LO, LE]{\thepage}
	\fancyfoot{}
	
}
\title{Numerical and convergence analysis of the stochastic Lagrangian averaged Navier-Stokes equations \footnote{Jad Doghman is supported by a public grant as part of the Investissement d'avenir project [ANR-11-LABX-0056-LMH, LabEx LMH], and both authors are part of the SIMALIN project [ANR-19-CE40-0016] of the French National Research Agency.}}
\author[ ]{Jad Doghman$^{1}$\thanks{Corresponding author\\ \mbox{\hspace{5pt} }Email addresses: \texttt{jad.doghman@centralesupelec.fr} (J. Doghman),\\ \texttt{goudenege@math.cnrs.fr} (L. Goudenège)}}
\author[ ]{Ludovic Gouden{\`e}ge$^{1}$}
\affil[1]{CNRS, Fédération de Mathématiques de CentraleSupélec FR 3487, Univ. Paris-Saclay, CentraleSupélec, F-91190 Gif-sur-Yvette, France}
\date{}
\begin{document}
	\maketitle
	\begin{abstract}
		The primary emphasis of this work is the development of a finite element based space-time discretization for solving the stochastic Lagrangian averaged Navier-Stokes (LANS-$\alpha$) equations of incompressible fluid turbulence with multiplicative random forcing, under nonperiodic boundary conditions within a bounded polygonal (or polyhedral) domain of $\mathbb{R}^{d}$, $d \in \{2,3\}$. The convergence analysis of a fully discretized numerical scheme is investigated and split into two cases according to the spacial scale $\alpha$, namely we first assume $\alpha$ to be controlled by the step size of the space discretization so that it vanishes when passing to the limit, then we provide an alternative study when $\alpha$ is fixed. A preparatory analysis of uniform estimates in both $\alpha$ and discretization parameters is carried out. Starting out from the stochastic LANS-$\alpha$ model, we achieve convergence toward the continuous strong solutions of the stochastic Navier-Stokes equations in $2$D when $\alpha$ vanishes at the limit. Additionally, convergence toward the continuous strong solutions of the stochastic LANS-$\alpha$ model is accomplished if $\alpha$ is fixed.
	\end{abstract}
	
	
	
	\MSC{76D05, 65M12, 35Q35, 35Q30, 60H15, 60H35}
	
	\newpage
	
	\section{Introduction}\label{sec;introduction}
	Over the last few decades, many regularization models of the Navier-Stokes equations (NSEs) have arisen, especially the $\alpha$-regularizations, for the sake of better understanding the closure problem of averaged quantities in turbulent flows. Such turbulent modeling schemes (e.g. Leray-$\alpha$, Navier-Stokes-$\alpha$, Clark-$\alpha$, Modified Leray-$\alpha$) were introduced as effective subgrid-scale models of the NSEs which require massive grid points or Fourier modes, allowing for approximation to capture all the spatial scales down to the Kolmogorov scale (see for instance~\cite{Cao2009Titi} and the references therein), as well as their suitability with the empirical and experimental data for a thorough range of Reynolds numbers.
	
	In the present paper, we consider the stochastic version of the LANS-$\alpha$ equations \cite{MarsdenShkoller} (also known as the viscous Camassa-Holm equations \cite{BJORLAND2008907}, or the Navier-Stokes-$\alpha$ model \cite{doi:10.1063/1.870096, HolmMarsdenRatiu:98b})
	\begin{equation}\label{main equation}
		\begin{cases}
			\begin{aligned}
				&\partial_{t}\left(\bar{u} - \alpha^{2}\Delta \bar{u}\right) - \nu \Delta\left(\bar{u} - \alpha^{2}\Delta \bar{u}\right) + [\bar{u}\cdot\nabla] \left(\bar{u} - \alpha^{2}\Delta \bar{u}\right) + \left(\nabla \bar{u}\right)^{T}\cdot (\bar{u} - \alpha^{2}\Delta \bar{u}) + \nabla p \\& \hspace{20pt} = f(\cdot, \bar{u}) + g(\cdot, \bar{u})\dot{W},
			\end{aligned}\\
			div \ \bar{u} = 0,\\
			\bar{u}(0,\cdot) = \bar{u}_{0},
		\end{cases}
	\end{equation}
	for internal flow i.e. for a bounded domain in $\mathbb{R}^{d}$, $d \in \{2,3\}$. The unknown vector field $\bar{u}$ is called the filtered fluid velocity, and it depends on time and space variables, $\nu$ is the fluid kinematic viscosity, and $\alpha$ is a small spatial scale at which fluid motion is filtered. Note that both $\nu$ and $\alpha$ are positive constants. $f = f(t, \bar{u})$ is an external force, the scalar quantity $p = p(t,x)$ represents the pressure and $\bar{u}_{0}$ is the corresponding initial datum. The last term of equations \eqref{main equation}$_{1}$ describes a state-dependent random noise, and it is defined by $\displaystyle g(\cdot, \bar{u})\dot{W} \coloneqq g(t, \bar{u})\partial_{t} W(t,x)$, where $g$ is a diffusion coefficient. One of the aims herein is to approach the two-dimensional solutions of the stochastic NSEs via the LANS-$\alpha$ model, numerically. Whence the need to evoke the former equations with similar configurations:
	\begin{equation}\label{Navier-Stokes}
		\begin{cases}
			\begin{aligned}
				&\partial_{t}v  - \nu \Delta v + [v \cdot\nabla]v + \nabla \lambda = f(\cdot, v) + g(\cdot, v)\dot{W},
			\end{aligned}\\
			div \ v = 0,\\
			v(0,\cdot) = v_{0},
		\end{cases}
	\end{equation}
	where $v$ (resp. $\lambda$) is the corresponding fluid velocity (resp. pressure), and $v_{0}$ embodies its initial datum.
	
	Equations~\eqref{main equation} and \eqref{Navier-Stokes} are usually employed as a complementary model to their deterministic versions to better understand the situation of tiny variations or perturbations present in fluid flows. The former represents a modification of the latter by performing Lagrangian means, asymptotic expansions, and an assumption of isotropy of fluctuations in the Hamilton principle, which grant further physical properties (e.g. conservation laws for energy and momentum). More specifically, the convective nonlinearity $\left[v\cdot\nabla\right]v$ in the NSEs is adjusted so that the cascading of turbulence at scales under specific length stops. The latter adjustment is called a nonlinearly dispersive modification.
	
	For stochastic hydrodynamical systems, a general theory was developed in \cite{ChueshovMillet2009}, involving  well-posedness and large deviations. It covers, for instance, $3$D Leray-$\alpha$ model, but not completely the LANS-$\alpha$ equations which, in parallel, have been a subject of several papers. The existence and uniqueness of a variational solution to the problem \eqref{main equation} were investigated in \cite{Caraballo2006Takeshi} under Lipschitz-continuous conditions in a three-dimensional bounded domain. A similar study is proposed in \cite{deugoue2009stochastic}, but this time with a genuine finite-dimensional Wiener process depending only on time. LANS-$\alpha$ model driven by an additive space-time noise of trace class was considered in \cite{goudenege:hal-02615640}, where the authors proved the existence and uniqueness of an invariant measure, and a probabilistic strong solution.
	
	Speaking of the numerical approach side, the convergence analysis of suitable numerical methods for the stochastic LANS-$\alpha$ equations is less well developed. In connection with the deterministic version, both convergence rate and convergence analysis of an algorithm consisting of a finite element method were investigated in \cite{Connors2010} where the spatial scale $\alpha$ is considered in terms of the space discretization's step. The author in \cite{ccaglar2010convergence} conducted a similar study, with $\alpha$ being independent of the discretization parameters. On the other hand, numerical schemes for stochastic nonlinear equations admitting local Lipschitz nonlinearities related to the Navier-Stokes systems had been already investigated. For instance, authors in \cite{brzezniak2013finite} studied a finite element-based space-time discretization of the incompressible NSEs driven by a multiplicative noise. An enhancement of \cite{brzezniak2013finite} in dimension $2$ was carried out in \cite{bessaih2020spacetime}.
	
	This paper aims to provide a fully discrete finite element-based discretization of equations~\eqref{main equation} in a bounded convex polygonal or polyhedral domain. Notice that the underlying model consists of a fourth order problem, nevertheless we avoid the use of $C^{1}$ piecewise polynomials-based finite element methods by introducing a notion of differential filters that transform equations~\eqref{main equation} into a coupled problem of second order. The employed time-discretization herein is an Euler scheme. One highly valued characteristic of the finite element method is the prospect of meticulous interpretation provided by the functional analysis framework. In contrast to the linear stochastic partial differential equations, since we are dealing here with a nonlinear model, one cannot make use of the semigroup method or Green's function. Those techniques are effectively replaced by monotone or Lipschitz-continuous drift functions. It is worth highlighting the importance of constructing practical numerical schemes provided with exact divergence-free finite element functions. However, due to their computational complexity, one may notice the usage of a weak divergence-free condition that compensates for the strong sense's absence.
	
	The associated spatial scale $\alpha$ will be considered hereafter either in terms of the space discretization's step (case~$1$) or independently of all discretization's parameters (case~$2$). Therefore, our main results consist of the convergence in both $2$D and $3$D of Algorithm~\ref{Algorithm} toward the continuous solution of the $2$D stochastic NSEs for the case~$1$, together with the convergence toward the continuous solution of the stochastic LANS-$\alpha$ model for the case~$2$. Speaking of the followed approach, we begin by performing a priori estimates characterized by their uniformity in $\alpha$ and the discretizations' parameters, allowing us to extract convergent subsequences of the approximate solution. We also demonstrate iterates' uniqueness, thanks to the scheme's linearity. The second step consists in applying two compactness lemmas (Lemmas~\ref{lemma convergence} and \ref{lemma convergence discrete}) to achieve strong convergence of the iterates in the same probability space, toward the NSEs and the LANS-$\alpha$ model. In other words, we do not use the Skorokhod theorem. Although the mentioned compactness lemmas do not include spaces with randomness, we accomplish a preliminary convergence through a suggested sample subset $\Omega_{k,h}^{\varepsilon} \subset \Omega$, whose measure is bigger than $1 - \varepsilon$, with $\varepsilon$ being independent of $\alpha$ and the discretizations' parameters. The latter allows us to demonstrate the strong convergence in the whole probability set $\Omega$. Afterwards, we identify the limit of each term of Algorithm~\ref{Algorithm} with its continuous (in time) counterpart.
	
	The paper is organized as follows. We introduce in section \ref{section notations preliminaries} a few notions and preliminaries, including the spatial framework, the needed assumptions, the time and space discretizations alongside their properties, definition of solutions to problems~\eqref{main equation} and \eqref{Navier-Stokes}, the definition of continuous and discrete differential filters, the investigated algorithm together with the needed compactness lemmas. Section~\ref{section3} is tailored for the main results of this paper. We exhibit part of the proofs in section~\ref{section4}, where we focus on the a priori estimates and iterates' uniqueness. In section~\ref{section5}, we study the convergence analysis of the corresponding numerical scheme. Accordingly, we identify both deterministic and stochastic integrals, as the discretization steps tend to $0$, with the corresponding exact solution. We terminate this paper (section~\ref{section conclusion}) with a conclusion concerning the obtained limiting functions and how one can relate them to the stochastic NSEs and LANS-$\alpha$ model. We equip this section with a computational experiment to visualize the outcomes and to evaluate the performance of the proposed numerical scheme.
	
	\section{Notations and preliminaries}\label{section notations preliminaries}
	We state, in this section, preliminary background material following the usual notation employed in the context of the mathematical theory of Navier-Stokes equations.\\
	Given $T > 0$, we denote by $D \subset \mathbb{R}^{d}$, $d \in \{2,3\}$ a bounded convex polygonal or polyhedral domain with boundary $\partial D$, in which we seek a solution, namely a stochastic process $\Big(\bar{u}(t), p(t)\Big), t \in [0,T]$ satisfying equations~\eqref{main equation} in a certain sense. Define almost everywhere on $\partial D$ the unit outward normal vector field $\vec{n} \colon \partial D \to \mathbb{R}^{d}$. The following function spaces are required hereafter:
	\begin{align*}
		&\mathcal{V} \coloneqq \bigg\{z \in [C_{c}^{\infty}(D)]^{d} \ \big| \ div \ z = 0\bigg\},
		\\&\mathbb{H} \coloneqq \bigg\{z \in \left(L^{2}(D)\right)^{d} \ \big| \ div \ z = 0 \mbox{ a.e. in } D, \ z.\vec{n} = 0 \mbox{ a.e. on } \partial D\bigg\},
		\\&\mathbb{V} \coloneqq \bigg\{z \in \left(H^{1}_{0}(D)\right)^{d} \ \big| \ div \ z = 0 \mbox{ a.e. in } D\bigg\}.
	\end{align*}
	From now on, the spaces of vector valued functions will be indicated with blackboard bold letters, for instance $\mathbb{L}^{2} \coloneqq \left(L^{2}(D)\right)^{d}$ denotes the Lebesgue space of vector valued functions defined on $D$. Denote by $\mathscr{P} \colon \mathbb{L}^{2} \to \mathbb{H}$ the Leray projector, and by $A \colon D(A) \to \mathbb{H}$ the Stokes operator defined by $A \coloneqq -\mathscr{P}\Delta$ with domain $D(A) = \mathbb{H}^{2} \cap \mathbb{V}$. $A$ is a self-adjoint positive operator, and has a compact inverse, see for instance \cite{RSMUP_1961__31__308_0}. Let $\left(\Omega, \mathcal{F}, \left(\mathcal{F}_{t}\right)_{t \in [0,T]}, \mathbb{P}\right)$ be a complete probability space, $Q$ a nuclear operator, and denoted by $K$ a separable Hilbert space on which we define the $Q$-Wiener process $W(t), \ t \in [0,T]$ such that 
	\begin{equation}\label{def Wiener process}
		W(t) = \sum_{k \in \mathbb{N}}\sqrt{q_{k}}\beta^{k}(t)w_{k}, \ \ \ \ \ \forall t \in [0,T],
	\end{equation} 
	where $\displaystyle \{\beta^{k}(\cdot), \ k \in \mathbb{N}\}$ is a sequence of independent and identically distributed $\mathbb{R}$-valued Brownian motions on the probability basis $\displaystyle \left(\Omega, \mathcal{F}, \left(\mathcal{F}_{t}\right)_{t \in [0,T]}, \mathbb{P}\right)$, $\{w_{k}, \ k \in \mathbb{N}\}$ is a complete orthonormal basis of the Hilbert space $K$ consisting of the eigenfunctions of $Q$, with eigenvalues $\{q_{k}\}_{k \in \mathbb{N}^{*}}$. The following estimate will play an essential role in the sequel, cf.~\cite{ichikawa1982stability}.
	\begin{equation}\label{eq 2.3.4}
		\mathbb{E}\left[\left|\left|W(t) - W(s)\right|\right|^{2r}_{K}\right] \leq \left(2r - 1\right)!!\left(t - s\right)^{r}\left(Tr(Q)\right)^{r}, \ \ \ \ \forall r \in \mathbb{N},
	\end{equation}
	where $\displaystyle \left(2r - 1\right)!! \coloneqq (2r - 1)(2r - 3)\dotsc \times 5 \times 3 \times 1$.
	
	For any arbitrary Hilbert spaces $X, Y$, the sets $\mathscr{L}_{1}(X, Y)$ and $\mathscr{L}_{2}(X, Y)$ denote the nuclear, and Hilbert-Schmidt operators from $X$ to $Y$, respectively. For brevity's sake, if $X = Y$, we set $\mathscr{L}_{i}(X,X) = \mathscr{L}_{i}(X), i\{1,2\}$. Hereafter, $\displaystyle M^{p}_{\mathcal{F}_{t}}(0,T; X)$ denotes the space of all $\mathcal{F}_{t}$-progressively measurable processes belonging to $L^{p}\left(\Omega \times (0,T), d\mathbb{P}\times dt; X\right)$, for any Banach space $X$.
	
	Throughout this paper, the nonnegative constant $C_{D}$ depends only on the domain $D$, the symbols $(\cdot,\cdot)$ and $\langle \cdot,\cdot \rangle$ stand for the inner product in $\mathbb{L}^{2}$ and the duality product between $\mathbb{H}^{-1}$ and $\mathbb{H}^{1}$, respectively.
	Recall that $\alpha$ is a small spatial scale, thereby we assume that $\displaystyle \alpha \leq 1$. The latter leads to the following norm equivalence
	\begin{equation}\label{eq norm equivalence}
		\alpha||\cdot||_{\mathbb{H}^{1}} \leq \left|\left|\cdot\right|\right|_{\alpha} \leq ||\cdot||_{\mathbb{H}^{1}},
	\end{equation}
	where $\left|\left|\cdot\right|\right|_{\alpha}$ is defined by $\displaystyle \left|\left|\cdot\right|\right|_{\alpha}^{2} \coloneqq \left|\left|\cdot\right|\right|^{2}_{\mathbb{L}^{2}} + \alpha^{2}\left|\left|\nabla\cdot\right|\right|^{2}_{\mathbb{L}^{2}}$. We point out that the whole study herein maintains all the stated properties if one chooses $\alpha \leq \alpha_{0}$, for some $\alpha_{0} \in \mathbb{R}^{*}_{+}$.
	For arbitrary real numbers $x, y$, the inequality $x \lesssim y$ is a shorthand for $x \leq cy$ for some universal constant $c > 0$.
	
	We list the assumptions needed below for the data $\bar{u}_{0}$, $g, Q$, and $f$.
	\paragraph{\textbf{Assumptions}}\label{General assumptions}
	\begin{enumerate}
		\item[$(S_{1})$] $Q \in \mathscr{L}_{1}(K)$ is a symmetric, positive definite operator,
		\item[$(S_{2})$] $f \in L^{2}(\Omega; C([0,T]; \mathbb{H}^{-1}))$ and $g \in L^{2}(\Omega; C([0,T]; \mathscr{L}_{2}(K, \mathbb{L}^{2})))$ are sublinear Lipschitz-continuous mappings, i.e. for all $z_{1}, z_{2} \in \mathbb{V}$, $g(\cdot, z_{1})$ and $f(\cdot, z_{2})$ are $\mathcal{F}_{t}$-progressively measurable, and $d\mathbb{P}\times dt$-a.e. in $\Omega\times (0,T)$, 
		\begin{align*}
			&\left|\left|g(\cdot, z_{1}) - g(\cdot, z_{2})\right|\right|_{\mathscr{L}_{2}(K, \mathbb{L}^{2})} \leq L_{g}\left|\left|z_{1} - z_{2}\right|\right|_{\alpha}, \ \ \ \ \ \forall z_{1}, z_{2} \in \mathbb{V},\\&
			\left|\left|g(\cdot, z)\right|\right|_{\mathscr{L}_{2}(K,\mathbb{L}^{2})} \leq K_{1} + K_{2}\left|\left|z\right|\right|_{\alpha}, \ \ \ \forall z \in \mathbb{V},\\&
			\left|\left|f(\cdot, z_{1}) - f(\cdot, z_{2})\right|\right|_{\mathbb{H}^{-1}} \leq L_{f}\left|\left|z_{1} - z_{2}\right|\right|_{\alpha}, \ \ \ \ \ \forall z_{1}, z_{2} \in \mathbb{V},\\&
			\left|\left|f(\cdot, z)\right|\right|_{\mathbb{H}^{-1}} \leq K_{3} + K_{4}\left|\left|z\right|\right|_{\alpha}, \ \ \ \ \forall z \in \mathbb{V},
		\end{align*}
		for some time-independent nonnegative constants $K_{1}, K_{2}, K_{3}, K_{4}, L_{f}, L_{g}$.
		\item[$(S_{3})$] $\bar{u}_{0} \in L^{4}(\Omega, \mathcal{F}_{0}, \mathbb{P}; \mathbb{V})$.
	\end{enumerate}
	
	To avoid repetitions later on, we state the following assertions
	\begin{align}
		&\Big(a - b, a\Big) = \frac{1}{2}\left(\left|\left|a\right|\right|_{\mathbb{L^{2}}}^{2} - \left|\left|b\right|\right|_{\mathbb{L}^{2}}^{2} + \left|\left|a-b\right|\right|_{\mathbb{L}^{2}}^{2}\right) \ \ \mbox{ for all } a, b \in \mathbb{L}^{2},\label{eq identity (a-b,a)}  \\&
		\left(\sum_{m=1}^{M}\left|a_{m}\right|\right)^{2} \leq 3\sum_{m=1}^{M}\left|a_{m}\right|^{2}, \ \ \forall M \in \mathbb{N}\backslash\{0\}. \label{eq sum estimate}
	\end{align}
	
	\paragraph{\textbf{The trilinear form}} We define the trilinear form $\tilde{b}$, associated with the LANS-$\alpha$ equations, by $$\tilde{b}(z_{1}, z_{2}, w) = \Big\langle [z_{1}\cdot \nabla]z_{2}, w \Big\rangle + \Big\langle \left(\nabla z_{1}\right)^{T}\cdot z_{2}, w \Big\rangle, \ \ \ \ \forall z_{1}, z_{2}, w \in \mathbb{H}_{0}^{1}.$$
	We list in the following proposition a few corresponding properties.
	\begin{prop}\label{prop trilinear form}\
		\begin{enumerate}
			\item[(i)] $\displaystyle \tilde{b}(z_{1}, z_{2}, w) = -\tilde{b}(w, z_{2}, z_{1})$, for all $z_{1}, w \in \mathbb{V}$, $z_{2} \in \mathbb{H}^{1}_{0}$,
			\item[(ii)] $\displaystyle \tilde{b}(z_{1}, z_{2}, z_{1}) = 0, \ \ $ for all $(z_{1}, z_{2}) \in \mathbb{V}\times \mathbb{H}^{1}_{0}$,
			\item[(iii)] $\left|\tilde{b}(z_{1}, z_{2}, w)\right| \leq C_{D}\left|\left|z_{1}\right|\right|_{\mathbb{H}^{1}}\left|\left|z_{2}\right|\right|_{\mathbb{H}^{1}}\left|\left|w\right|\right|_{\mathbb{H}^{1}}$, for all $z_{1}, z_{2}, w \in \mathbb{H}^{1}_{0}$,
			\item[(iv)] $\displaystyle \int_{D}(\nabla z_{1})^{T}\cdot z_{2} wdx = - \int_{D}[w\cdot\nabla]z_{2}z_{1}dx$, for all $z_{1}, z_{2} \in \mathbb{H}^{1}_{0}$ and $w \in \mathbb{V}$.
		\end{enumerate}
	\end{prop}
	\begin{proof}
		For $(i)$-$(iii)$, see for instance \cite[Lemma 1]{Foias2002Camassa-holm}, and the last assertion was covered in \cite[Proposition 2.1]{Caraballo2006Takeshi} with a slight modification here regarding the utilized spaces.
	\end{proof}
	
	It is well-known that finite element methods based on $C^{1}$ piecewise polynomials are not easily implementable. This means that our fourth-order partial differential equation~\eqref{main equation} must undergo a modification so that it turns into a second-order problem. To this end, we shall propose a differential filter that deals with a Stokes problem. Such an idea emerges from \cite{Germano1986} within a slight adjustment for the sake of fitting the current framework. The divergence-free condition in the definition below is not mandatory as one can always use the Helmholtz decomposition to subsume the resulting gradient term within $\nabla \tilde{p}$.
	\begin{defi}[Continuous differential filter]\label{def cont diff filter}\ \\
		Given a (divergence-free) vector field $v \in \mathbb{L}^{2}$ vanishing on $\partial D$, its continuous differential filter, denoted by $\bar{u}$, is part of the unique solution $\left(\bar{u}, \tilde{p}\right) \in \mathbb{V}\times L_{0}^{2}(D)$ to
		\begin{equation}\label{eq Stokes}
			\begin{cases}
				\begin{aligned}
					&-\alpha^{2}\Delta \bar{u}  + \bar{u} + \nabla \tilde{p}= v, &\mbox{ in } D,\\
					&div \ \bar{u} = 0, &\mbox{ in } D,\\
					&\bar{u} = 0,  &\mbox{ on } \partial D.
				\end{aligned}
			\end{cases}
		\end{equation} 
	\end{defi}
	Note that the differential filter of a function $v$ is usually denoted by $\bar{v}$. Nevertheless, the employed notation herein will be $\bar{u}$ to obtain a clear vision of the relationship between the differential filter and equations~\eqref{main equation}. For a given $v \in \mathbb{L}^{2}$, problem~\eqref{eq Stokes} yields a unique $\bar{u} \in \mathbb{H}^{2}\cap \mathbb{V}$ provided that $D \subset \mathbb{R}^{d},$ is a bounded convex two-dimensional polygonal (three-dimensional polyhedral) domain. Moreover, the solution $\bar{u}$ satisfies $\displaystyle \left|\left|\bar{u}\right|\right|_{\mathbb{H}^{2}} \leq C_{D}\alpha^{-2}\left|\left|v\right|\right|_{\mathbb{L}^{2}}$. The former and the latter property are provided in \cite[Subsection 8.2]{grisvard2011elliptic}. Observe that $v$ in equations~\eqref{eq Stokes} is assumed to be null on $\partial D$ due to the occurring equality $\bar{u} = v$ when one passes to the limit in $\alpha$ after projecting \eqref{eq Stokes}$_{1}$ using the Leray projector $\mathscr{P}$.
	
	\subsection{Definition of solutions}
	Relying on paper~\cite{Caraballo2006Takeshi}, a solution to equations~\eqref{main equation} can be defined as follows:
	\begin{defi}\label{definition variational solution}\ \\
		Let $T > 0$ and assume that $(S_{1})$-$(S_{3})$ are valid. A $\mathbb{V}$-valued stochastic process $\bar{u}(t), \ t \in [0,T]$ is said to be a variational solution to problem \eqref{main equation} if it fulfills the following conditions:
		\begin{enumerate}
			\item[(i)] $\bar{u} \in M^{2}_{\mathcal{F}_{t}}(0,T; 
			D(A)) \cap L^{2}\left(\Omega; L^{\infty}(0,T; \mathbb{V})\right)$,
			\item[(ii)] $\bar{u}$ is weakly continuous with values in $\mathbb{V}$, $\mathbb{P}$-almost surely, 
			\item[(iii)] for all $t \in [0,T]$, $\bar{u}$ satisfies the following equation $\mathbb{P}$-almost surely
			\begin{equation}\label{eq definition solution}
				\begin{aligned}
					&\left(\bar{u}(t), \phi\right) + \alpha^{2}\left(\nabla \bar{u}(t), \nabla \phi\right) + \nu \int_{0}^{t}\Big(\bar{u}(s) + \alpha^{2}A\bar{u}(s), A\phi\Big)ds \\&+ \int_{0}^{t}\tilde{b}\left(\bar{u}(s), \bar{u}(s) - \alpha^{2}\Delta \bar{u}(s), \phi\right)ds = \left(\bar{u}_{0}, \phi\right) + \alpha^{2}\left(\nabla \bar{u}_{0}, \nabla \phi\right) \\&+ \int_{0}^{t}\Big\langle f(s, \bar{u}(s)), \phi \Big\rangle ds + \Big(\int_{0}^{t}g\left(s, \bar{u}(s)\right)dW(s), \phi\Big), \ \ \forall \phi \in D(A).
				\end{aligned}
			\end{equation}
		\end{enumerate}
	\end{defi}
	If $\bar{u}$ is a solution to problem~\eqref{main equation} in the sense of Definition~\ref{definition variational solution}, then considering $v = v(t)$ as in problem~\eqref{eq Stokes} grants a new (equivalent) formula for equation~\eqref{eq definition solution}, namely for all $t \in [0,T]$, there holds $\mathbb{P}$-almost surely 
	\begin{equation}\label{eq definiton solution modified}
		\begin{aligned}
			&\left(v(t), \phi\right) + \nu\int_{0}^{t}\left(\nabla v(s), \nabla \phi\right)ds + \int_{0}^{t}\tilde{b}(\bar{u}(s), v(s), \phi)ds = \left(v_{0}, \phi\right) \\&+ \int_{0}^{t}\big\langle f(s,\bar{u}(s)), \phi\big\rangle ds + \left(\int_{0}^{t}g(s, \bar{u}(s))dW(s), \phi\right), \ \ \forall \phi \in D(A),
		\end{aligned}
	\end{equation}
	where $v_{0} \in \mathbb{L}^{2}$ is given by equation~\eqref{eq Stokes} when $\bar{u} = \bar{u}_{0}$. The trilinear term involving the pressure $\int_{0}^{t}\tilde{b}(\bar{u}(s), \nabla\tilde{p}(s),\phi)ds$ does not appear in equation~\eqref{eq definiton solution modified} because
	$$\displaystyle\tilde{b}(\bar{u},\nabla\tilde{p},\phi) = \sum_{i,j=1}^{d}\int_{D}\bar{u}^{i}\partial_{i}\partial_{j}\tilde{p}\phi^{j} dx + \sum_{i,j=1}^{d}\int_{D}\partial_{i}\bar{u}^{j}\partial_{j}\tilde{p}\phi^{i} dx.$$
	The first term on the right-hand side turns into $-\int_{D}[\phi\cdot\nabla]\bar{u}\nabla\tilde{p}dx$ after performing an integration by parts, and the second term can be rewritten as $\int_{D}[\phi\cdot\nabla]\bar{u}\nabla\tilde{p}dx$.
	It is worth mentioning that \eqref{eq definiton solution modified}, coupled with the weak formulation of \eqref{eq Stokes}, establishes a well-posed problem whose solution satisfies equations~\eqref{main equation} in the sense of Definition~\ref{definition variational solution}.
	
	Next, we give a definition of strong solutions to problem~\eqref{Navier-Stokes} in $2$D.
	\begin{defi}\label{def NS solution}\ \\
		Given $T > 0$, let assumptions $(S_{1})$ and $(S_{2})$ be fulfilled, $d =2$ and $v_{0} \in L^{2}(\Omega, \mathcal{F}_{0}, \mathbb{P}; \mathbb{H})$ be the initial datum. An $\mathbb{H}$-valued stochastic process $v(t), t \in [0,T]$ is said to be a strong solution to equations~\eqref{Navier-Stokes} if it satisfies:
		\begin{enumerate}
			\item[(i)] $v \in M_{\mathcal{F}_{t}}^{2}(0,T; \mathbb{V}) \cap L^{2}(\Omega; C([0,T]; \mathbb{H}))$,
			\item[(ii)] for all $t \in [0,T]$, there holds $\mathbb{P}$-a.s.
			\begin{equation}\label{eq def weak formulation NS}
				\begin{aligned}
					&\left(v(t), \varphi\right) + \nu\int_{0}^{t}\left(\nabla v(s), \nabla\varphi\right)ds + \int_{0}^{t}\big\langle [v(s)\cdot\nabla]v(s), \varphi \big\rangle ds = \left(v_{0}, \varphi\right) \\&+ \int_{0}^{t}\big\langle f(s, v(s)), \varphi \big\rangle ds + \left(\int_{0}^{t}g(s, v(s))dW(s), \varphi\right), \ \ \forall \varphi \in D(A).
				\end{aligned}
			\end{equation}
		\end{enumerate}
	\end{defi}
	
	\subsection{Discretizations and algorithm}
	\paragraph{\underline{Time Discretization}}
	Let $M \in \mathbb{N}^{*}$ be given, and set $I_{k} = \{t_{\ell}\}_{\ell = 0}^{M}$ an equidistant partition of the interval $[0,T]$, where $t_{0} \coloneqq 0$, $t_{M} \coloneqq T$ and $\displaystyle k \coloneqq T/M$ is the time-step size. The equidistance condition is not mandatory in the sequel, but it is imposed for simplicity. One can generalize the presented method by associating a time-step $k_{m}$ with each sub-interval $[t_{m-1}, t_{m}]$, for all $m \in \{1, \dotsc, M\}$.
	
	\paragraph{\underline{Space discretization}}
	For simplicity's sake, we let $\mathcal{T}_{h}$ be a quasi-uniform triangulation of the domain $D \subset \mathbb{R}^{d}$, $d \in \{2, 3\}$ into simplexes of maximal diameter $h > 0$, and $\displaystyle \overline{D} = \bigcup\limits_{K \in \mathcal{T}_{h}}\overline{K}$. The space of polynomial vector fields on an arbitrary set $O$ with degree less than or equal to $n \in \mathbb{N}$ is denoted by $\mathbb{P}_{n}(O) \coloneqq \left(P_{n}(O)\right)^{d}$. For $n_{1}, n_{2} \in \mathbb{N}\backslash\{0\}$, we let
	\begin{align*}
		&\mathbb{H}_{h} \coloneqq \Big\{z_{h} \in \mathbb{H}^{1}_{0} \cap [C^{0}(\overline{D})]^{d} \ \Big| \ z_{h}|_{_{K}} \in \mathbb{P}_{n_{1}}(K), \ \ \forall K \in \mathcal{T}_{h}  \Big\},\\& 
		L_{h} \coloneqq \Big\{q_{h} \in L^{2}_{0}(D) \ \Big| \ q_{h}|_{_{K}} \in P_{n_{2}}(K), \ \ \forall K \in \mathcal{T}_{h}\Big\},
		\\& \mathbb{V}_{h} \coloneqq \Big\{ z_{h} \in \mathbb{H}_{h} \ \Big| \ \left(div \ z_{h}, q_{h}\right) = 0, \ \ \ \forall \ q_{h} \in L_{h}\Big\},
	\end{align*}
	be the finite element function spaces. For fixed $n_{1}, n_{2} \in \mathbb{N}\backslash \{0\}$, we assume that $\left(\mathbb{H}_{h}, L_{h}\right)$ satisfies the discrete $\inf$-$\sup$ condition; namely there is a constant $\beta > 0$ independent of the mesh size $h$ such that
	\begin{equation}\label{eq LBB condition}
		\sup\limits_{z_{h} \in \mathbb{H}_{h}\backslash \{0\}}\frac{\left(div \ z_{h}, q_{h}\right)}{\left|\left|\nabla z_{h}\right|\right|_{\mathbb{L}^{2}}} \geq \beta \left|\left|q_{h}\right|\right|_{L^{2}}, \ \ \ \forall \  q_{h} \in L_{h}.
	\end{equation}
	Given $z \in \mathbb{L}^{2}$, we denote by $\Pi_{h} \colon \mathbb{L}^{2} \to \mathbb{V}_{h}$ the $\mathbb{L}^{2}$-orthogonal projections, defined as the unique solution of the identity
	\begin{equation}\label{def projection}
		\left(z - \Pi_{h}z, \varphi_{h}\right) = 0, \ \ \forall \varphi_{h} \in \mathbb{V}_{h}.
	\end{equation}
	For $z \in \mathbb{H}_{0}^{1}$, $\Delta^{h}\colon \mathbb{H}_{0}^{1} \to \mathbb{V}_{h}$ denotes the discrete Laplace operator, defined as the unique solution of 
	\begin{equation}\label{def discrete Laplace}
		\left(\Delta^{h}z, \varphi_{h}\right) = - \left(\nabla z, \nabla \varphi_{h}\right), \ \ \ \forall \varphi_{h} \in \mathbb{V}_{h}.
	\end{equation}
	Estimate~\eqref{eqprojection} and the inverse inequality~\eqref{eq inverse estimate} below need to be satisfied by the recently defined approximate function spaces. Let $\mathbb{S}_{h}$ be a finite dimensional subspace of $\mathbb{H}^{1}_{0}$ equipped with an $\mathbb{L}^{2}$-projector $\Pi_{\mathbb{S}_{h}} \colon \mathbb{L}^{2} \to \mathbb{S}_{h}$, satisfying the following property:
	
	For $\displaystyle z \in \mathbb{H}_{0}^{1} \cap \mathbb{W}^{s,2}$, there is a positive constant $C$ independent of $h$ such that
	\begin{equation}\label{eqprojection}
		\sum_{j=0}^{1}h^{j}\left|\left|D^{j}\left(z - \Pi_{\mathbb{S}_{h}}z\right)\right|\right|_{\mathbb{L}^{2}} \leq Ch^{s}\left|\left|z\right|\right|_{\mathbb{W}^{s,2}}, \ \ 2 \leq s \leq n + 1, 
	\end{equation}
	where $n$ is the polynomials' degree in $\mathbb{S}_{h}$.\\
	Furthermore, assume that $\mathbb{S}_{h}$ fulfills the following inverse inequality:
	
	For $\ell \in \mathbb{N}$, $1 \leq p,q \leq +\infty$ and $0 \leq m \leq \ell$, there exists a constant $C$ independent of $h$ such that 
	\begin{equation}\label{eq inverse estimate}
		\left|\left|z_{h}\right|\right|_{\mathbb{W}^{\ell, p}} \leq Ch^{m-\ell + d\min(\frac{1}{p} - \frac{1}{q},0)}\left|\left|z_{h}\right|\right|_{\mathbb{W}^{m,q}}, \ \ \ \forall z_{h} \in \mathbb{S}_{h}.
	\end{equation}
	Provided the triangulation of the domain $D$ is quasi-uniform, one can easily check that the space $\mathbb{H}_{h}$ satisfies both estimates~\eqref{eqprojection} and \eqref{eq inverse estimate}. The reader may refer to \cite{brenner2007mathematical} for adequate proofs. Subsequently, we take $\mathbb{S}_{h} = \mathbb{H}_{h}$.
	
	The discrete differential filter is somewhat defined as its continuous counterpart, but this time by involving the weak formulation of problem~\eqref{eq Stokes}.
	\begin{defi}[Discrete differential filter]\label{definition discrete diff filter}\ \\
		Let $v$ be the vector field of Definition~\ref{def cont diff filter}. Its discrete differential filter, denoted by $\bar{u}_{h} \in \mathbb{V}_{h}$, is given by $$\alpha^{2}\left(\nabla \bar{u}_{h}, \nabla \varphi_{h}\right) + \left(\bar{u}_{h}, \varphi_{h}\right) = \left(v, \varphi_{h}\right), \ \ \ \forall \varphi_{h} \in \mathbb{V}_{h}.$$
	\end{defi}
	Additional information are stated in article~\cite[Section 4]{manica2006convergence}  . We list some of its properties in the following lemma.
	
	\begin{lem}\label{lemma discrete diff filter Hh}\ \\
		Let $v = v_{h} \in \mathbb{V}^{h}$ and $\bar{u}_{h} \in \mathbb{V}_{h}$ be its discrete differential filter. Then,
		\begin{enumerate}
			\item[(i)] $\displaystyle v_{h} = \bar{u}_{h} - \alpha^{2}\Delta^{h}\bar{u}_{h}$ and $\nabla v_{h} = \nabla \bar{u}_{h} - \alpha^{2}\nabla \Delta^{h}\bar{u}_{h}$ a.e. in $D$.
			\item[(ii)] $\displaystyle \left(\nabla v_{h}, \nabla \bar{u}_{h}\right) = \left|\left|\nabla \bar{u}_{h}\right|\right|^{2}_{\mathbb{L}^{2}} + \alpha^{2}\left|\left|\Delta^{h}\bar{u}_{h}\right|\right|^{2}_{\mathbb{L}^{2}}$.
		\end{enumerate}
	\end{lem}
	\begin{proof}
		Assertions $(i)$ and $(ii)$ are covered by \cite[Lemma 2.1]{Connors2010}.
	\end{proof}
	
	Before exhibiting the algorithm, we will define new notations for the approximate functions. The subscript $h$ of the utilized test functions will be dropped throughout the rest of this paper for the sake of clarity. For $t \in [0,T]$, we set $V(t) \coloneqq v_{h}(t)$ for $v_{h} \in \mathbb{V}_{h}$, and denote by $U(t)$ its discrete differential filter, i.e. $U(t) \coloneqq \bar{u}_{h}(t)$. Besides, let $\Pi(t) \coloneqq p_{h}(t)$ and $\tilde{\Pi}(t) \coloneqq \tilde{p}_{h}(t)$ be the (space) approximate pressures. We point out that Algorithm~\ref{Algorithm} is derived from equation~\eqref{eq definiton solution modified}, which contains both variables $\bar{u}$ and $v$.
	\begin{algo}\label{Algorithm}\ \\
		Given a starting point $U^{0} \in \mathbb{H}_{h}$, find for every $m \in \{1, \dotsc, M\}$, a $4$-tuple stochastic process $\left(U^{m}, V^{m}, \Pi^{m}, \tilde{\Pi}^{m}\right)~\in~\mathbb{H}_{h}\times\mathbb{H}_{h} \times L_{h}\times L_{h}$ such that for all $(\varphi, \psi, \Lambda_{1}, \Lambda_{2}) \in \mathbb{H}_{h}\times \mathbb{H}_{h} \times L_{h} \times L_{h}$, there holds $\mathbb{P}$-a.s.
		\begin{equation*}
			\begin{cases}
				\begin{split}
					\bullet &\Big(V^{m} - V^{m-1}, \varphi\Big) + k \nu\Big(\nabla V^{m}, \nabla \varphi\Big) + k\tilde{b}\Big(U^{m}, V^{m-1}, \varphi\Big) - k\Big(\Pi^{m}, div \ \varphi\Big) \\&= k \Big\langle f(t_{m-1}, U^{m-1}), \varphi \Big\rangle + \Big(g(t_{m-1}, U^{m-1})\Delta_{m}W, \varphi\Big),
				\end{split} \\
				\bullet \left(V^{m}, \psi\right) = \left(U^{m}, \psi\right) + \alpha^{2}\left(\nabla U^{m}, \nabla \psi\right) - \left(\tilde{\Pi}^{m}, div \ \psi\right),\\
				\bullet \left(div \ U^{m}, \Lambda_{1}\right) = \left(div \ V^{m}, \Lambda_{2}\right) = 0,
			\end{cases}
		\end{equation*}
		where $\Delta_{m}W = W(t_{m}) - W(t_{m-1})$ for all $m \in \{1, \dotsc, M\}$.
	\end{algo}
	For each $m \in \{0, \dotsc, M\}$, we may conclude from the second and third equations of Algorithm~\ref{Algorithm} along with Definition~\ref{definition discrete diff filter} two facts:
	\begin{enumerate}
		\item[(i)] $U^{m}$ is the discrete differential filter of $V^{m}$ and thereby, all the associated properties are valid.
		\item[(ii)] The Algorithm's starting point $U^{0}$ could be exchanged with $V^{0}$.
	\end{enumerate}
	
	We still need to state two mainly important lemmas that will contribute in the convergence of Algorithm~\ref{Algorithm}. Besides, given a function $u$, the shift operator $\tau_{\ell}$ is defined by $\tau_{\ell}u(t, x) \coloneqq u(t + \ell, x)$, for all $(t, x) \in [0,T-\ell] \times D$.
	
	The following lemma is provided in~\cite[Lemma 6]{CHEN20122252}, and will be employed when $\alpha$ is assumed to be controlled by $h$.
	\begin{lem}\label{lemma convergence}\ \\
		Let $Y$ be a Banach space and $M_{+}$ be a normed vector space in $Y$. Assume that the embedding $M_{+} \hookrightarrow Y$ is compact and that $U$ is a bounded subset of $L^{2}(0,T; M_{+})$. We suppose in addition that $\displaystyle||\tau_{\ell}u - u||_{L^{2}(0,T - \ell; Y)} \to 0$ as $\ell \to 0$, uniformly in $u \in U$. Then, $U$ is relatively compact in $L^{2}(0,T; Y)$.
	\end{lem}
	
	Conversely, when $\alpha$ is considered independently of $h$ and $k$, the below lemma will play an alternative role, and it consists of a discrete version of Lemma~\ref{lemma convergence}.
	\begin{lem}\label{lemma convergence discrete}\ \\
		Let $n \coloneqq (k,h) \in (\mathbb{R}_{+}^{*})^{2}$, $(B, ||\cdot||_{B})$ be a Banach space and $(M_{h}, ||\cdot||_{M_{h}})$ be a normed space in $B$. Let $(u_{n})_{n}$ be a sequence in $L^{2}(0,T; B)$. Assume that 
		\begin{enumerate}
			\item[(i)] if $(\varphi_{h})_{h}$ is a sequence of $B$ such that $||\varphi_{h}||_{M_{h}} \leq C$ for all $h > 0$, for some $C>0$ then, $(\varphi_{h})_{h}$ is relatively compact in $B$,
			\item[(ii)] $\left|\left|u_{n}\right|\right|_{L^{2}(0,T; M_{h})} \leq C_{1}$ and $\left|\left|u_{n}\right|\right|_{L^{1}_{loc}(0,T; B)} \leq C_{2}$ for all $n$, for some $C_{1}, C_{2}>0$,
			\item[(iii)] $\left|\left|\tau_{\ell}u_{n} - u_{n}\right|\right|_{L^{2}(0,T - \ell; B)} \to 0$ as $\ell \to 0$, uniformly in $n \in \mathbb{R}^{*}_{+}\times \mathbb{R}^{*}_{+}$.
		\end{enumerate}
		Then, $(u_{n})_{n}$ is relatively compact in $L^{2}(0,T; B)$.
	\end{lem}
	\begin{proof}
		For $h > 0$, define $b_{h}(v) \coloneqq 
		\begin{cases}
			\frac{||v||_{B}}{||v||_{M_{h}}}v &\mbox{ if } v \in M_{h}\backslash \{0\}, \\
			0 &\mbox{ if } v \in (B \backslash M_{h})\cup \{0\}
		\end{cases}.$
		Let us show that $b_{h} \colon B \to B$ is a (nonlinear) compact operator. Indeed, assume that $(v_{h})_{h>0}$ is a bounded sequence of $B$ i.e. there is $M \geq 0$ such that $||v_{h}||_{B} \leq M$ for all $h > 0$.  We have $b_{h}(v_{h}) \in M_{h}$ and $||b_{h}(v_{h})||_{M_{h}} = ||v_{h}||_{B} \leq M$ for all $h > 0$. Therefore, by assumption~$(i)$, $(b_{h}(v_{h}))_{h}$ is relatively compact in $B$. Whence the compactness of $b_{h}$. For $n = (k,h) \in (\mathbb{R}_{+}^{*})^{2}$ and $t \in [0,T]$, define the sequence $v_{n}(t) \coloneqq \begin{cases}
			\frac{||u_{n}(t)||_{M_{h}}}{||u_{n}(t)||_{B}}u_{n}(t) &\mbox{ on } \{u_{n}(t) \neq 0\},\\
			0  &\mbox{ on } \{u_{n}(t) = 0\}
		\end{cases}$.
		We have $||v_{n}||_{L^{2}(0,T; B)} \leq ||u_{n}||_{L^{2}(0,T; M_{h})} \leq C_{1}$ for all $n$, thanks to assertion~$(ii)$. Thus, $(v_{n})_{n}$ is bounded in $L^{2}(0,T; B)$, particularly in $L^{1}(0,T; B)$. On the other hand, $b_{h}(v_{n}(t)) = \begin{cases}
			\frac{||v_{n}(t)||_{B}}{||v_{n}(t)||_{M_{h}}}v_{n}(t) &\mbox{ on } \{v_{n}(t) \neq 0\},\\
			0 &\mbox{ on } \{v_{n}(t) = 0\}
		\end{cases} = \begin{cases}
			u_{n}(t) &\mbox{ on } \{u_{n}(t) \neq 0\}, \\
			0 &\mbox{ on } \{u_{n}(t) = 0\}
		\end{cases} = u_{n}(t)$ for all $t \in [0,T]$. Thereby, $u_{n} = b_{h}(v_{n})$, for all $n \in (\mathbb{R}_{+}^{*})^{2}$. Using the above results together with assertions $(ii)$, $(iii)$ and applying Theorem 1 in~\cite{CHEN20122252} yield the relative compactness of $(u_{n})_{n}$ in $L^{2}(0,T; B)$.
	\end{proof}
	
	\section{Main results}\label{section3}
	
	In the light of the preceding preliminaries, we are now able to state the main results of this paper. Theorem~\ref{Thm LANS-alpha} concerns the stochastic LANS-$\alpha$ model and Theorem~\ref{Thm NS} is devoted to the stochastic Navier-Stokes equations.
	
	\begin{thm}\label{Thm LANS-alpha}\ \\
		Let $T > 0$, $D \subset \mathbb{R}^{d}, \ d \in \{2,3\}$ be a bounded convex polygonal or polyhedral domain and $\left(\Omega, \mathcal{F}, (F_{t})_{t \in [0,T]}, \mathbb{P}\right)$ be a filtered probability space. Assume that assumptions $(S_{1})$-$(S_{3})$ are fulfilled. For any finite positive pair $(k,h)$, let $\mathcal{T}_{h}$ be a quasi-uniform triangulation of $D$, $I_{k}$ be an equidistant partition of the time interval $[0,T]$, $(\mathbb{H}_{h}, L_{h})$ be a pair of finite element spaces satisfying the LBB-condition \eqref{eq LBB condition}, and $U^{0}$ be in $\mathbb{H}_{h}$ such that $\left|\left|U^{0}\right|\right|_{\mathbb{H}^{1}}$ is uniformly bounded in $h > 0$. If $\frac{\sqrt{k}}{h} < L \leq \alpha$ for some $L \in (0,1)$ independent of $k$ and $h$ then, there exists a solution $\{(U^{m}, V^{m}, \Pi^{m}, \tilde{\Pi}^{m})\}_{m=1}^{M}$ of Algorithm \ref{Algorithm}, and it satisfies Lemma~\ref{lemma3.3.1}. Moreover, if $U^{0} \to \bar{u}_{0}$ in $L^{4}(\Omega; \mathbb{H}^{1})$ as $h \to 0$, Algorithm~\ref{Algorithm} converges toward the unique solution of equations~\eqref{main equation} in the sense of Definition~\ref{definition variational solution}.
	\end{thm}
	
	\begin{thm}\label{Thm NS}\ \\
		Let $T > 0$, $D \subset \mathbb{R}^{2}$ be a bounded convex polygonal domain and $\left(\Omega, \mathcal{F}, (F_{t})_{t \in [0,T]}, \mathbb{P}\right)$ be a filtered probability space. Assume assumptions $(S_{1})$ and $(S_{2})$ and let $v_{0} \in L^{4}(\Omega, \mathcal{F}_{0}, \mathbb{P}; \mathbb{H})$ be the initial datum of equations~\eqref{Navier-Stokes}. For any finite positive pair $(k,h)$, let $\mathcal{T}_{h}$ be a quasi-uniform triangulation of $D$, $I_{k}$ be an equidistant partition of the time interval $[0,T]$, $(\mathbb{H}_{h}, L_{h})$ be a pair of finite element spaces satisfying the LBB-condition \eqref{eq LBB condition}, and $V^{0}$ be in $\mathbb{H}_{h}$ such that $\left|\left|V^{0}\right|\right|_{\mathbb{L}^{2}}$ is uniformly bounded in $h > 0$. If $\alpha \leq Ch$ for some $C > 0$ independent of $k$ and $h$ then, there exists a solution $\{(U^{m}, V^{m}, \Pi^{m}, \tilde{\Pi}^{m})\}_{m=1}^{M}$ of Algorithm \ref{Algorithm}, and it satisfies Lemmas~\ref{lemma3.3.1} and \ref{lemma V a priori estimates}. Further, if $V^{0} \to v_{0}$ in $L^{4}(\Omega; \mathbb{L}^{2})$ as $h \to 0$ then, Algorithm~\ref{Algorithm} converges toward the unique solution of equations~\eqref{Navier-Stokes} in the sense of Definition~\ref{def NS solution}.
	\end{thm}
	
	As stated in the hypothesis of both Theorems~\ref{Thm LANS-alpha} and \ref{Thm NS}, one needs to bound the initial datum $U^{0}$ (or $V^{0}$) of Algorithm~\ref{Algorithm} independently of $h > 0$. To do so, we evoke the Ritz operator $\mathcal{R}_{h}$ which is stable in $\mathbb{H}^{1}$ i.e. there is a positive non-decreasing function $\zeta$, uniform in $h$ such that $||\mathcal{R}_{h}v||_{\mathbb{H}^{1}} \leq \zeta||v||_{\mathbb{H}^{1}}$ for all $v \in \mathbb{H}^{1}$. Given $v \in \mathbb{H}^{1}$, the Ritz operator $\mathcal{R}_{h} \colon \mathbb{H}^{1} \longrightarrow \mathbb{V}_{h}$ is defined as the unique solution of
	\begin{align*}
		\ \left(\nabla \mathcal{R}_{h}v, \nabla v_{h}\right) = \left(\nabla v, \nabla v_{h}\right), \ \  \forall v_{h} \in \mathbb{V}_{h}.
	\end{align*}
	Therefore, we define $U^{0}$ by $U^{0} = \mathcal{R}_{h}\bar{u}_{0}$ where $\bar{u}_{0}$ is the initial datum of equations~\eqref{main equation}. The same operator is suitable if $V^{0}$ was chosen to be the starting point of Algorithm~\ref{Algorithm}. In this case, we set $V^{0} = \mathcal{R}_{h}v_{0}$, where $v_{0}$ is given by problem~\eqref{eq Stokes} when $\bar{u} = \bar{u}_{0}$. For further properties, the reader may refer to \cite[Lemma 4.2]{guermond2008stability}.
	
	\section{Solvability, stability and a priori estimates}\label{section4}
	
	Notice that the system of equations proposed in Algorithm \ref{Algorithm} can be reformulated after taking the test functions $\varphi$ and $\psi$ in $\mathbb{V}_{h}$:
	\begin{equation}\label{eq3.3.1}
		\begin{cases}
			\begin{split}
				\bullet &\big(V^{m} - V^{m-1}, \varphi\big) + k\nu \big(\nabla V^{m}, \nabla \varphi\big) + k\tilde{b}\big(U^{m}, V^{m-1}, \varphi\big) \\&\hspace{20pt}= k \big\langle f(t_{m-1}, U^{m-1}), \varphi \big\rangle + \big(g\small(t_{m-1}, U^{m-1}\small)\Delta_{m}W, \varphi\big), \ \ \forall \varphi \in \mathbb{V}_{h}.
			\end{split}\\
			\bullet \left(V^{m}, \psi\right) = \left(U^{m}, \psi\right) + \alpha^{2}\left(\nabla U^{m}, \nabla \psi\right), \ \ \forall \psi \in \mathbb{V}_{h}.
		\end{cases}
	\end{equation}
	
	In the below lemma, we illustrate the solvability of Algorithm~\ref{Algorithm}, the iterates' measurability, and some a priori estimates whose role is to afford the proposed numerical scheme with stability.
	\begin{lem}\label{lemma3.3.1}\ \\
		Assume that assumptions $(S_{1})$-$(S_{3})$ are valid. Then, there exists a $\mathbb{V}_{h}\times\mathbb{V}_{h}\times L_{h}\times L_{h}$-valued sequence of random variables $\{(U^{m}, V^{m}, \Pi^{m}, \tilde{\Pi}^{m})\}_{m=1}^{M}$ that solves $\mathbb{P}$-a.s. Algorithm~\ref{Algorithm}, and fulfills the following assertions:
		\begin{enumerate}
			\item[(i)] for any $m \in \{1, \dotsc, M\}$, the maps $U^{m}, V^{m} \colon \Omega \to \mathbb{H}_{h}$ are $\mathcal{F}_{t_{m}}$-measurable.
			\item [(ii)]
			$\displaystyle \mathbb{E}\bigg[\max\limits_{1 \leq m \leq M}\left|\left|U^{m}\right|\right|_{\alpha}^{2} + \frac{k\nu}{2} \sum_{m=1}^{M}\left(\left|\left|\nabla U^{m}\right|\right|_{\mathbb{L}^{2}}^{2} + \alpha^{2}\left|\left|\Delta^{h}U^{m}\right|\right|^{2}_{\mathbb{L}^{2}}\right) + \frac{1}{4}\sum_{m=1}^{M}\left|\left|U^{m} - U^{m-1}\right|\right|_{\alpha}^{2}\bigg] \leq C_{T},$
			\item[(iii)] $\displaystyle \mathbb{E}\bigg[\max\limits_{1 \leq m \leq M}\left|\left|U^{m}\right|\right|_{\alpha}^{4} + \frac{1}{4}\sum_{m=1}^{M}\left|\left|U^{m}\right|\right|_{\alpha}^{2}\left|\left|U^{m} - U^{m-1}\right|\right|_{\alpha}^{2} \\ \mbox{\hspace{100pt} }+ \frac{k\nu}{4}\sum_{m=1}^{M}\left|\left|U^{m}\right|\right|_{\alpha}^{2}\left(\left|\left|\nabla U^{m}\right|\right|_{\mathbb{L}^{2}}^{2} + \alpha^{2}\left|\left|\Delta^{h}U^{m}\right|\right|^{2}_{\mathbb{L}^{2}}\right)\bigg] \leq C_{T,2},$
		\end{enumerate}
		where $C_{T,q} = C_{T,q}\left(||U^{0}||_{L^{2q}(\Omega; \mathbb{H}^{1})}, T, (K_{i})_{i=1}^{4}, Tr(Q), \nu, D\right)$, $q \in \{1,2\}$ is a positive constant, independent of $\alpha$, $k$ and $h$. Note that $C_{T} \coloneqq C_{T,1}$.
	\end{lem}
	\begin{proof}
		\textit{Solvability}\\
		To prove the Algorithm's solvability, we will follow a technique similar to that in \cite[Lemma 4.1]{Banas2014Prohl} while relying on equations~\eqref{eq3.3.1}. Since $V^{m} \in \mathbb{V}_{h}$ for all $m \in \{1, \dotsc, M\}$ then, by Lemma~\ref{lemma discrete diff filter Hh}-$(i)$, we get $V^{m} = U^{m} - \alpha^{2}\Delta^{h}U^{m}$, $\mathbb{P}$-a.s. and a.e. in $D$. This means that the existence of $U^{m}$ implies that of $V^{m}$. Assume that, for some $2 \leq \ell \leq M$ and for almost every $\omega \in \Omega$, a sequence $\{(U^{m}(\omega), V^{m}(\omega))\}_{m=1}^{\ell - 1}$ has been found by induction.
		For $\omega \in \Omega$, define $\mathbb{P}$-a.s. the mapping $\mathcal{F}_{\ell - 1}^{\omega} \colon \mathbb{V}_{h} \to \mathbb{V}_{h}'$ by
		\begin{align*}
			\mathcal{F}_{\ell - 1}^{\omega}(\varphi) &\coloneqq \varphi - \alpha^{2}\Delta^{h}\varphi - V^{\ell-1}(\omega) - k\nu\left(\Delta \varphi -\alpha^{2}\Delta \Delta^{h}\varphi\right) + k[\varphi\cdot \nabla]V^{\ell - 1}(\omega)  \\&\hspace{8pt}+ k \left(\nabla \varphi\right)^{T}\cdot V^{\ell - 1}(\omega) - kf(t_{\ell-1}, U^{\ell-1}(\omega)) - g(t_{\ell-1}, U^{\ell-1}(\omega))\Delta_{\ell}W(\omega),
		\end{align*}
		for all $\varphi \in \mathbb{V}_{h}$. The continuity of $\mathcal{F}_{\ell - 1}^{\omega}$ can be shown by a straightforward argument. Since, $\mathbb{V}_{h}$ equipped with the inner product $\left(\cdot, \cdot\right)$, is a Hilbert space, then by Riesz representation theorem, functional $\mathcal{F}_{\ell - 1}^{\omega}$ can be defined through the $L^{2}$-inner product, namely for $\varphi \in \mathbb{V}_{h}$, $\left(\mathcal{F}_{\ell - 1}^{\omega}(\varphi)\right)(\psi) = \left(\mathcal{F}_{\ell - 1}^{\omega}(\varphi), \psi\right)$ for all $\psi \in \mathbb{V}_{h}$. Therefore, for $\psi = \varphi \in \mathbb{V}_{h}$ and by Proposition~\ref{prop trilinear form}-$(ii)$, the discrete Laplace operator~\eqref{def discrete Laplace}, assumption~$(S_{2})$, the Cauchy-Schwarz and Young inequalities,
		\begin{align*}
			&\left(\mathcal{F}_{\ell - 1}^{\omega}(\varphi), \varphi \right) \geq ||\varphi||_{\mathbb{L}^{2}}^{2} + (\alpha^{2} + k\nu)||\nabla \varphi||^{2}_{\mathbb{L}^{2}} - ||V^{\ell - 1}(\omega)||_{\mathbb{L}^{2}}||\varphi||_{\mathbb{L}^{2}} + k\nu\alpha^{2}||\Delta^{h}\varphi||^{2}_{\mathbb{L}^{2}} \\&- k\left(K_{3} + K_{4}||U^{\ell - 1}(\omega)||_{\alpha}\right)||\varphi||_{\mathbb{H}^{1}} - \left(K_{1} + K_{2}||U^{\ell - 1}(\omega)||_{\alpha}\right)||\Delta_{\ell}W(\omega)||_{K}||\varphi||_{\mathbb{L}^{2}}
			\\&\geq \frac{1}{2}||\varphi||_{\mathbb{L}^{2}}^{2} + (\alpha^{2} + \frac{k\nu}{2})||\nabla \varphi||_{\mathbb{L}^{2}} - ||V^{\ell - 1}(\omega)||^{2}_{\mathbb{L}^{2}} - \frac{kC_{D}^{2}}{2\nu}\left(K_{3} + K_{4}||U^{\ell - 1}(\omega)||_{\alpha}\right)^{2} \\&- \left(K_{1} + K_{2}||U^{\ell - 1}(\omega)||_{\alpha}\right)^{2}||\Delta_{\ell}W(\omega)||^{2}_{K} \geq \frac{1}{2}||\varphi||^{2}_{\mathbb{L}^{2}} - L_{\ell - 1}(\omega),
		\end{align*}
		where $L_{\ell - 1} \coloneqq 2K_{1}^{2}||\Delta_{\ell}W||_{K}^{2} + \frac{kC_{D}^{2}K_{3}^{2}}{\nu} + ||V^{\ell - 1}||_{\mathbb{L}^{2}}^{2} + \left(\frac{kC_{D}^{2}K_{4}^{2}}{\nu} + 2K_{2}^{2}||\Delta_{\ell}W||^{2}_{K}\right)||U^{\ell - 1}||^{2}_{\alpha}$.
		By \eqref{eq 2.3.4} and the induction's hypothesis, there holds $\mathbb{P}$-a.s. $L_{\ell - 1}(\omega) < +\infty$. Therefore, taking $\varphi \in \mathbb{V}_{h}$ such that $||\varphi||_{\mathbb{L}^{2}} = \sqrt{2L_{\ell - 1}(\omega)}$ yields $\left(\mathcal{F}_{\ell - 1}^{\omega}(\varphi), \varphi\right) \geq 0$. Subsequently, Brouwer's fixed point theorem (see \cite[Corollary 1.1, p. 279]{girault2012finite}) ensures the existence (but not uniqueness) of a $\phi = \phi(\omega)  \in \mathbb{V}_{h}$ such that $\mathcal{F}^{\omega}_{\ell - 1}(\phi) = 0$. Hence, $(U^{\ell}, V^{\ell}) \in \mathbb{V}_{h}\times \mathbb{V}_{h}$ exists $\mathbb{P}$-a.s. . The discrete LBB-condition \eqref{eq LBB condition} yields the existence of an $L_{h}\times L_{h}$-valued process $\{(\Pi^{m}, \tilde{\Pi}^{m})\}_{m=1}^{M}$ satisfying Algorithm~\ref{Algorithm}.\\
		\textit{Measurabililty} \\
		After proving the algorithm's solvability through the functional $\mathcal{F}^{\omega}_{\ell - 1}$, the measurability of iterates $U^{m}$, $m \in \{1, \dotsc, M\}$ follows by induction (see \cite[Lemma 4.1]{Banas2014Prohl}). Moreover, by Lemma~\ref{lemma discrete diff filter Hh}-$(i)$, one infers the measurability of $\{V^{m}\}_{m=1}^{M}$.\\
		\textit{A priori energy estimate} \\
		Let us denote by $\left|\left|\cdot\right|\right|_{h,\alpha}^{2}$ the quantity $\left|\left|\nabla\cdot\right|\right|^{2}_{\mathbb{L}^{2}} + \alpha^{2}\left|\left|\Delta^{h}\cdot\right|\right|^{2}_{\mathbb{L}^{2}}$. In equation \eqref{eq3.3.1}, we take $\varphi = \psi = U^{m}$ and employ identity~\eqref{eq identity (a-b,a)} and Lemma~\ref{lemma discrete diff filter Hh}-$(ii)$:
		\begin{equation}\label{eq3.3.4}
			\begin{aligned}
				&\frac{1}{2}\left(||U^{m}||_{\alpha}^{2} - ||U^{m-1}||_{\alpha}^{2} + ||U^{m} - U^{m-1}||_{\alpha}^{2}\right) + k\nu ||U^{m}||^{2}_{h,\alpha} = k \langle f(t_{m-1}, U^{m-1}), U^{m} \rangle \\&+ \left(g(t_{m-1}, U^{m-1})\Delta_{m}W, U^{m} - U^{m-1}\right) + \left(g(t_{m-1}, U^{m-1})\Delta_{m}W, U^{m-1}\right).
			\end{aligned}
		\end{equation}
		After employing the Cauchy-Schwarz and Young inequalities along with assumption $(S_{2})$, we take the sum over $m$ from $1$ to $M$:
		\begin{equation}\label{eq 3.3.4'}
			\begin{aligned}
				&\frac{1}{2}||U^{M}||_{\alpha}^{2} - \frac{1}{2}||U^{0}||_{\alpha}^{2} + \frac{1}{4}\sum_{m=1}^{M}||U^{m} - U^{m-1}||_{\alpha}^{2} + \frac{k\nu}{2} \sum_{m=1}^{M}||U^{m}||^{2}_{h,\alpha} \\& \leq \frac{C_{D}^{2}TK_{3}^{2}}{\nu} + \frac{C_{D}^{2}K_{4}^{2}}{\nu}k\sum_{m=1}^{M}||U^{m-1}||_{\alpha}^{2} + \sum_{m=1}^{M}||g(t_{m-1}, U^{m-1})\Delta_{m}W||_{\mathbb{L}^{2}}^{2} \\&\hspace{10pt}+ \sum_{m=1}^{M}(g(t_{m-1}, U^{m-1})\Delta_{m}W,U^{m-1}).
			\end{aligned}
		\end{equation}
		Due to the measurability of $U^{m}$, the last term on the right-hand side vanishes when taking its expectation. The penultimate term is controlled as follows:
		\begin{equation}\label{eq 3.3.6}
			\begin{aligned}
				&\mathbb{E}\left[||g(t_{m-1}, U^{m-1})||_{\mathscr{L}_{2}(K,\mathbb{L}^{2})}^{2}||\Delta_{m}W||_{K}^{2}\right] \\&= \mathbb{E}\left[||g(t_{m-1},U^{m-1})||_{\mathscr{L}_{2}(K,\mathbb{L}^{2})}^{2}\right]\mathbb{E}\left[||\Delta_{m}W||_{K}^{2} \Big| \mathcal{F}_{t_{m-1}}\right] \\&\leq 2Tr(Q)K_{1}^{2}k + 2K_{2}^{2}kTr(Q)\mathbb{E}\left[||U^{m-1}||_{\alpha}^{2}\right],
			\end{aligned}
		\end{equation}
		thanks to the tower property of the conditional expectation, the increments independence of the Wiener process, property \eqref{eq 2.3.4}, and assumption $(S_{2})$. Plugging estimate~\eqref{eq 3.3.6} in equation~\eqref{eq 3.3.4'} returns
		\begin{equation}\label{eq3.3.5}
			\begin{aligned}
				&\frac{1}{2}\mathbb{E}\left[||U^{M}||_{\alpha}^{2}\right] + \frac{1}{4}\sum_{m=1}^{M}\mathbb{E}\left[||U^{m} - U^{m-1}||_{\alpha}^{2}\right] + \frac{k\nu}{2} \sum_{m=1}^{M}\mathbb{E}\left[||U^{m}||^{2}_{h,\alpha}\right]
				\leq \frac{1}{2}\mathbb{E}\left[||U^{0}||_{\alpha}^{2}\right] \\&+ \left(\frac{C_{D}^{2}K_{3}^{2}}{\nu} + 2Tr(Q)K_{1}^{2}\right)T + \left(\frac{C_{D}^{2}K_{4}^{2}}{\nu} + 2K_{2}^{2}Tr(Q)\right)k\sum_{m=0}^{M-1}\mathbb{E}\left[||U^{m}||_{\alpha}^{2}\right].
			\end{aligned}
		\end{equation}
		Now, we employ the discrete Gronwall inequality (see for instance \cite[Lemma 10.5]{thomee2007galerkin}) in order to prove the sought estimate. We replace $M$ in equation \eqref{eq3.3.5} by any other index $\ell \geq 1$. We get
		\begin{equation*}
			\begin{aligned}
				\mathbb{E}\left[||U^{\ell}||_{\alpha}^{2}\right] \leq \left[\mathbb{E}\left[||U^{0}||_{\mathbb{H}^{1}}^{2}\right] + 2\left(\frac{C_{D}^{2}K_{3}^{2}}{\nu} + 2Tr(Q)K_{1}^{2}\right)T\right] e^{T\left(\frac{C_{D}^{2}K_{4}^{2}}{\nu} + 2K_{2}^{2}Tr(Q)\right)} \eqqcolon K_{T}
			\end{aligned}
		\end{equation*}
		for all $\ell \in \{1, \dotsc , M\}$, where $||U^{0}||_{\alpha} \leq ||U^{0}||_{\mathbb{H}^{1}}$ thanks to \eqref{eq norm equivalence}. Consequently,
		\begin{equation}\label{eq3.3.7}
			\max\limits_{1 \leq m \leq M}\mathbb{E}\left[||U^{m}||_{\alpha}^{2}\right] \leq K_{T}.
		\end{equation}
		By virtue of estimate \eqref{eq3.3.5} and the discrete Gronwall lemma, one also obtains the following two estimates: $\frac{k\nu}{2} \sum_{m=1}^{M}\mathbb{E}\left[||U^{m}||_{h,\alpha}^{2}\right] \leq K_{T}$ and $\frac{1}{4}\sum_{m=1}^{M}\mathbb{E}\left[||U^{m} - U^{m-1}||_{\alpha}^{2}\right] \leq K_{T}$.
		We still need to prove $\mathbb{E}\left[\max\limits_{1 \leq m \leq M}||U^{m}||_{\alpha}^{2}\right] \leq C_{T}$, for a certain positive constant $C_{T}$ independent of $\alpha$, $k$ and $h$. To this end, we make use of estimate~\eqref{eq 3.3.4'}, but this time by summing from $m=1$ to $m=\ell$ where $\ell \geq 1$ is an integer. Then, we take the maximum over $\ell$ and apply the mathematical expectation on both sides to get
		\begin{equation}\label{eq3.3.4''}
			\begin{aligned}
				&\frac{1}{2}\mathbb{E}\left[\max\limits_{1 \leq \ell \leq M}||U^{\ell}||_{\alpha}^{2}\right] \leq \frac{1}{2}\mathbb{E}\left[||U^{0}||_{\alpha}^{2}\right] + \frac{C_{D}^{2}TK_{3}^{2}}{\nu} + \frac{C_{D}^{2}K_{4}^{2}}{\nu}k\sum_{m=1}^{M}\mathbb{E}\left[||U^{m-1}||_{\alpha}^{2}\right] \\&+ \sum_{m=1}^{M}\mathbb{E}\left[||g(t_{m-1},U^{m-1})\Delta_{m}W||_{\mathbb{L}^{2}}^{2}\right] + \mathbb{E}\left[\max\limits_{1 \leq \ell \leq M}\sum_{m=1}^{\ell}(g(t_{m-1},U^{m-1})\Delta_{m}W,U^{m-1})\right].
			\end{aligned}
		\end{equation}
		To bound the last term on the right-hand side, we use assumption $(S_{2})$, the Burkholder-Davis-Gundy and Young inequalities, after considering the sum as the stochastic integral of a piecewise constant integrand:
		\begin{equation}\label{eq3.3.9}
			\begin{aligned}
				&\mathbb{E}\left[\max\limits_{1 \leq \ell \leq M}\sum_{m=1}^{\ell}\left(g(t_{m-1}, U^{m-1})\Delta_{m}W,U^{m-1}\right)\right] \\&\lesssim \mathbb{E}\left[\left(k\sum_{m=1}^{M}\left|\left|g(t_{m-1},U^{m-1})\right|\right|_{\mathscr{L}_{2}(K,\mathbb{L}^{2})}^{2}||U^{m-1}||_{\mathbb{L}^{2}}^{2}\right)^{1/2}\right]
				\\&\leq \frac{1}{4}\mathbb{E}\left[||U^{0}||_{\mathbb{L}^{2}}^{2}\right] + 2K_{1}^{2}T + \mathbb{E}\left[\frac{1}{4}\max\limits_{1 \leq \ell \leq M}||U^{\ell}||_{\mathbb{L}^{2}}^{2} + 2K_{2}^{2}k\sum_{m=1}^{M}||U^{m-1}||_{\alpha}^{2}\right].
			\end{aligned}
		\end{equation}
		Returning to estimate~\eqref{eq3.3.4''}, we avail ourselves of \eqref{eq 3.3.6}, \eqref{eq3.3.7} and \eqref{eq3.3.9} to conclude
		\begin{align*}
			&\mathbb{E}\left[\max\limits_{1 \leq m \leq M}||U^{m}||_{\alpha}^{2}\right] \leq C_{T},
		\end{align*}
		where $C_{T}>0$ depends only on the parameters of $K_{T}$.\\
		\textit{Bounds for higher velocity moments}\\
		We start by multiplying equation \eqref{eq3.3.4} by the norm $||U^{m}||_{\alpha}^{2}$.
		\begin{equation}\label{eq higher moment velocity}
			\begin{aligned}
				&\frac{1}{2}||U^{m}||_{\alpha}^{4} - \frac{1}{2}||U^{m-1}||_{\alpha}^{2}||U^{m}||_{\alpha}^{2} + \frac{1}{2}||U^{m} - U^{m-1}||_{\alpha}^{2}||U^{m}||_{\alpha}^{2} + k\nu ||U^{m}||_{h,\alpha}^{2}||U^{m}||_{\alpha}^{2} \\&= k \langle f(t_{m-1}, U^{m-1}), U^{m} \rangle ||U^{m}||_{\alpha}^{2} + \left(g(t_{m-1},U^{m-1})\Delta_{m}W, U^{m} - U^{m-1}\right)||U^{m}||_{\alpha}^{2} \\&\hspace{10pt}+ \left(g(t_{m-1},U^{m-1})\Delta_{m}W, U^{m-1}\right)||U^{m}||_{\alpha}^{2} =  I + II + III.
			\end{aligned}
		\end{equation}
		For $I$, we apply the norm equivalence~\eqref{eq norm equivalence}, the Young inequality and estimate $|a+b|^{p} \leq 2^{p-1}(|a|^{p} + |b|^{p})$ for $p=4$:
		\begin{align*}
			&I \leq kC_{D}\left(K_{3} + K_{4}||U^{m-1}||_{\alpha}\right)||\nabla U^{m}||_{\mathbb{L}^{2}}^{\frac{3}{2}}||U^{m}||^{\frac{3}{2}}_{\alpha} \leq \frac{kC_{D}^{4}}{4\nu^{3}}\left(K_{3} + K_{4}||U^{m-1}||_{\alpha}\right)^{4} \\&+ \frac{3k\nu}{4}||U^{m}||_{h,\alpha}^{2}||U^{m}||^{2}_{\alpha} \leq \frac{2kC_{D}^{4}K_{3}^{4}}{\nu^{3}} + \frac{2kC_{D}^{4}K_{4}^{4}}{\nu^{3}}||U^{m-1}||^{4}_{\alpha} + \frac{3k\nu}{4}||U^{m}||_{h,\alpha}^{2}||U^{m}||^{2}_{\alpha}.
		\end{align*}
		For $II$,
		\begin{align*}
			II & \leq ||g(t_{m-1},U^{m-1})||_{\mathscr{L}_{2}(K,\mathbb{L}^{2})}^{2}||\Delta_{m}W||_{K}^{2}\left(||U^{m}||_{\alpha}^{2} - ||U^{m-1}||_{\alpha}^{2} + ||U^{m-1}||_{\alpha}^{2}\right) \\&\hspace{10pt}+ \frac{1}{4}||U^{m} - U^{m-1}||_{\mathbb{L}^{2}}^{2}||U^{m}||_{\alpha}^{2} \\& \leq ||g(t_{m-1}, U^{m-1})||_{\mathscr{L}_{2}(K,\mathbb{L}^{2})}^{2}||\Delta_{m}W||_{K}^{2}||U^{m-1}||_{\alpha}^{2} + \frac{1}{16}\left| ||U^{m}||_{\alpha}^{2} - ||U^{m-1}||_{\alpha}^{2}\right|^{2} \\&\hspace{10pt}+ 4||g(t_{m-1}, U^{m-1})||_{\mathscr{L}_{2}(K,\mathbb{L}^{2})}^{4}||\Delta_{m}W||_{K}^{4} + \frac{1}{4}||U^{m} - U^{m-1}||_{\mathbb{L}^{2}}^{2}||U^{m}||_{\alpha}^{2}.
		\end{align*}
		For $III$,
		\begin{align*}
			III &\coloneqq \left(g(t_{m-1}, U^{m-1})\Delta_{m}W, U^{m-1}\right)\left(||U^{m}||_{\alpha}^{2} - ||U^{m-1}||_{\alpha}^{2} + ||U^{m-1}||_{\alpha}^{2}\right) \\& \leq \left(g(t_{m-1},U^{m-1})\Delta_{m}W, U^{m-1}\right)||U^{m-1}||_{\alpha}^{2} + \frac{1}{16}\left| ||U^{m}||_{\alpha}^{2} - ||U^{m-1}||_{\alpha}^{2}\right|^{2} \\&\hspace{10pt} + 4||g(t_{m-1},U^{m-1})||_{\mathscr{L}_{2}(K,\mathbb{L}^{2})}^{2}||\Delta_{m}W||_{K}^{2}||U^{m-1}||_{\alpha}^{2}.
		\end{align*}
		Equation \eqref{eq higher moment velocity} becomes
		\begin{align*}
			&\frac{1}{2}||U^{m}||_{\alpha}^{4} - \frac{1}{2}||U^{m-1}||_{\alpha}^{2}||U^{m}||_{\alpha}^{2} + \frac{1}{4}||U^{m} - U^{m-1}||_{\alpha}^{2}||U^{m}||_{\alpha}^{2} + \frac{k\nu}{4}||U^{m}||_{h,\alpha}^{2}||U^{m}||_{\alpha}^{2} \\&\leq \frac{2kC_{D}^{4}K_{3}^{4}}{\nu^{3}} + \frac{2kC_{D}^{4}K_{4}^{4}}{\nu^{3}}||U^{m-1}||^{4}_{\alpha} + \frac{1}{8}\left| ||U^{m}||_{\alpha}^{2} - ||U^{m-1}||_{\alpha}^{2}\right|^{2} \\&\hspace{10pt}+ \left(g(t_{m-1},U^{m-1})\Delta_{m}W, U^{m-1}\right)||U^{m-1}||_{\alpha}^{2} +4||g(t_{m-1},U^{m-1})||_{\mathscr{L}_{2}(K,\mathbb{L}^{2})}^{4}||\Delta_{m}W||_{K}^{4} \\&\hspace{10pt}+ 5||g(t_{m-1},U^{m-1})||^{2}_{\mathscr{L}_{2}(K,\mathbb{L}^{2})}||\Delta_{m}W||^{2}_{K}||U^{m-1}||_{\alpha}^{2}.
		\end{align*}
		Note that $||U^{m}||_{\alpha}^{4} - ||U^{m-1}||_{\alpha}^{2}||U^{m}||_{\alpha}^{2} = \frac{1}{2}(||U^{m}||_{\alpha}^{4} - ||U^{m-1}||_{\alpha}^{4} + \left| ||U^{m}||_{\alpha}^{2} - ||U^{m-1}||_{\alpha}^{2}\right|^{2})$,
		therefore
		\begin{equation}\label{eq3.3.12}
			\begin{aligned}
				&\frac{1}{4}\Big(||U^{m}||_{\alpha}^{4} - ||U^{m-1}||_{\alpha}^{4} + \frac{1}{2}\left| ||U^{m}||_{\alpha}^{2} - ||U^{m-1}||_{\alpha}^{2}\right|^{2} + ||U^{m} - U^{m-1}||_{\alpha}^{2}||U^{m}||_{\alpha}^{2} \\&+ k\nu||U^{m}||_{h,\alpha}^{2}||U^{m}||_{\alpha}^{2}\Big) \leq \frac{2kC_{D}^{4}K_{3}^{4}}{\nu^{3}} + \frac{2kC_{D}^{4}K_{4}^{4}}{\nu^{3}}||U^{m-1}||^{4}_{\alpha} \\&+ \left(g(t_{m-1},U^{m-1})\Delta_{m}W, U^{m-1}\right)||U^{m-1}||_{\alpha}^{2}+ 4||g(t_{m-1},U^{m-1})||_{\mathscr{L}_{2}(K,L^{2})}^{4}||\Delta_{m}W||_{K}^{4} \\& +  5||g(t_{m-1},U^{m-1})||^{2}_{\mathscr{L}_{2}(K,\mathbb{L}^{2})}||\Delta_{m}W||^{2}_{K}||U^{m-1}||_{\alpha}^{2},
			\end{aligned}
		\end{equation}
		Proceeding as \eqref{eq 3.3.6}, the penultimate term can be estimated as follows
		\begin{equation}\label{eq3.3.12'}
			\mathbb{E}\left[||g(t_{m-1},U^{m-1})||_{\mathscr{L}_{2}(K,\mathbb{L}^{2})}^{4}||\Delta_{m}W||^{4}_{K}\right] \lesssim K_{1}^{4}Tr(Q)^{2}k^{2} + K_{2}^{4}Tr(Q)^{2}k^{2}\mathbb{E}\left[||U^{m-1}||_{\alpha}^{4}\right].
		\end{equation}
		Next, we bound the last term  on the right-hand side of \eqref{eq3.3.12}
		\begin{equation}\label{eq bound in g}
			\begin{aligned}
				&\mathbb{E}\left[||g(t_{m-1},U^{m-1})||_{\mathscr{L}_{2}(K,\mathbb{L}^{2})}^{2}||\Delta_{m}W||_{K}^{2}||U^{m-1}||_{\alpha}^{2}\right] \\&\lesssim K_{1}^{2}kTr(Q)\mathbb{E}\left[||U^{m-1}||^{2}_{\alpha}\right] +  K_{2}^{2}Tr(Q)k\mathbb{E}\left[||U^{m-1}||_{\alpha}^{4}\right].
			\end{aligned}
		\end{equation}
		The third term on the right-hand side of \eqref{eq3.3.12} vanishes after taking its expectation, thanks to the measurability of the iterates $U^{m}, \ m \in \{1, \dotsc, M\}$. We collect and plug the above estimates back in \eqref{eq3.3.12}, and we sum it up over $m$ from $m=1$ to $m=M$. Then, we apply the mathematical expectation, and employ the discrete Gronwall lemma to get
		\begin{equation}\label{eq3.3.13}
			\max\limits_{1 \leq m \leq M} \mathbb{E}\left[||U^{m}||_{\alpha}^{4}\right] \leq C_{T,2},
		\end{equation}
		where $C_{T,2} > 0$ does not depend on $\alpha$, $k$ and $h$. We also get by Gronwall lemma the following two estimates: $$\frac{1}{4}\mathbb{E}\left[\sum_{m=1}^{M}||U^{m} - U^{m-1}||_{\alpha}^{2}||U^{m}||_{\alpha}^{2}\right] \leq C_{T,2} \mbox{ and } \frac{k\nu}{4}\mathbb{E}\left[\sum_{m=1}^{M}||U^{m}||_{h,\alpha}^{2}||U^{m}||_{\alpha}^{2}\right] \leq C_{T,2}.$$
		It remains to show that $\mathbb{E}\left[\max\limits_{1 \leq m \leq M}||U^{m}||_{\alpha}^{4}\right] \leq C_{T,2}$. To do so, we follow the technique which was employed in the previous step (A priori energy estimate) by summing up inequality~\eqref{eq3.3.12} over $m$ from $1$ to $\ell \geq 1$. We will only need to control the following stochastic term:
		\begin{align*}
			&\mathbb{E}\left[\max\limits_{1 \leq \ell\leq M}\sum_{m=1}^{\ell}\left(g(t_{m-1},U^{m-1})\Delta_{m}W,U^{m-1}\right)||U^{m-1}||_{\alpha}^{2}\right] \\&\lesssim \mathbb{E}\left[\left(k\sum_{m=1}^{M}||g(t_{m-1},U^{m-1})||^{2}_{\mathscr{L}_{2}(K,\mathbb{L}^{2})}||U^{m-1}||_{\alpha}^{6}\right)^{\frac{1}{2}}\right] \\&\leq \mathbb{E}\left[\frac{1}{8}||U^{0}||^{4}_{\mathbb{H}^{1}} + \frac{1}{8}\max\limits_{1 \leq m \leq M}||U^{m}||_{\alpha}^{4} + 4K_{1}^{2}k\sum_{m=1}^{M}||U^{m-1}||_{\alpha}^{2} + 4K_{2}^{2}k\sum_{m=1}^{M}||U^{m-1}||_{\alpha}^{4}\right].
		\end{align*}
		Collecting all estimates together and using \eqref{eq3.3.13} complete the proof.
	\end{proof}
	
	While proving the solvability of Algorithm~\ref{Algorithm}, we found out that the iterates $\{(U^{m}, V^{m})\}_{m=1}^{M}$ might not be unique. We will discuss this property in the upcoming lemma.
	\begin{lem}\label{lemma3.3.2}\ \\
		Iterates $\{(U^{m}, V^{m})\}_{m=1}^{M}$ of Algorithm~\ref{Algorithm}, are unique $\mathbb{P}$-almost surely and almost everywhere in $D$.
	\end{lem}
	\begin{proof}
		Assume that the $\mathbb{V}_{h}\times\mathbb{V}_{h}$-valued processes $\{(U_{1}^{m}, V_{1}^{m})\}_{m=1}^{M}$ and $\{(U^{m}_{2}, V_{2}^{m})\}_{m=1}^{M}$ solve the projected version~\eqref{eq3.3.1} of Algorithm~\ref{Algorithm}, starting from the same initial datum $U^{0}$. For all $m \in \{0, 1, \dotsc, M\}$, denote by $\mathcal{U}^{m} \coloneqq U_{1}^{m} - U_{2}^{m}$ and by $\mathcal{V}^{m} \coloneqq V_{1}^{m} - V^{m}_{2}$. Clearly, $\mathcal{U}^{0} = \mathcal{V}^{0} = 0$ $\mathbb{P}$-almost surely. Replace both solutions in equation \eqref{eq3.3.1} and subtract them to get for all $(\varphi, \psi) \in \mathbb{V}_{h}\times \mathbb{V}_{h}$:
		\begin{equation}\label{eq algo subtracted}
			\begin{cases}
				\begin{aligned}
					\bullet &\left(\mathcal{V}^{m} - \mathcal{V}^{m-1}, \varphi\right) + k\nu\left(\nabla \mathcal{V}^{m}, \nabla\varphi\right) + k[\tilde{b}(U^{m}_{1}, V^{m-1}_{1}, \varphi) - \tilde{b}(U^{m}_{2}, V^{m-1}_{2}, \varphi)] \\&= k \langle f(t_{m-1}, U_{1}^{m-1})- f(t_{m-1}, U_{2}^{m-1}), \varphi \rangle + \left(g(t_{m-1}, U_{1}^{m-1}) - g(t_{m-1}, U_{2}^{m-1}), \varphi\right),
				\end{aligned}\\
				\bullet\left(\mathcal{V}^{m}, \psi\right) = \left(\mathcal{U}^{m}, \psi\right) + \alpha^{2}\left(\nabla \mathcal{U}^{m}, \nabla \psi\right).
			\end{cases}
		\end{equation}
		We proceed by induction. For $m=1$, the trilinear term becomes $\tilde{b}(\mathcal{U}^{1}, 0, \varphi) = 0$ and equation~\eqref{eq algo subtracted}$_{1}$ leads to $\left(\mathcal{V}^{1}, \varphi\right) + k\nu\left(\nabla \mathcal{V}^{1}, \nabla\varphi\right) = 0$. Taking $\varphi = \psi = \mathcal{U}^{1}$ in both \eqref{eq algo subtracted}$_{2}$ and the latter equation and subtracting them yield $||\mathcal{U}^{1}||_{\alpha}^{2} = -k\nu\left(\nabla \mathcal{V}^{1}, \nabla\mathcal{U}^{1}\right)$. From Lemma~\ref{lemma discrete diff filter Hh}-$(i)$, one gets $\nabla \mathcal{V}^{1} = \nabla \mathcal{U}^{1} - \alpha^{2}\nabla\Delta^{h}\mathcal{U}^{1}$ $\mathbb{P}$-a.s. and a.e. in $D$. Thereby, $||\mathcal{U}^{1}||_{\alpha}^{2} = -k\nu||\nabla \mathcal{U}^{1}||_{\mathbb{L}^{2}}^{2} + k\nu\alpha^{2}\left(\nabla \Delta^{h}\mathcal{U}^{1}, \nabla \mathcal{U}^{1}\right)$. Using identity~\eqref{def discrete Laplace}, we infer that $$||\mathcal{U}^{1}||^{2}_{\alpha} + k\nu||\nabla \mathcal{U}^{1}||^{2}_{\mathbb{L}^{2}} + k\nu\alpha^{2}||\Delta^{h}\mathcal{U}^{1}||^{2}_{\mathbb{L}^{2}} = 0, \mbox{ } \mathbb{P}\mbox{-a.s.}$$ Hence, $\mathcal{U}^{1} = \mathcal{V}^{1} = 0$, $\mathbb{P}$-a.s. and a.e. in $D$. By following the same technique while assuming $(U_{1}^{m-1}, V^{m-1}_{1}) = (U^{m-1}_{2}, V_{2}^{m-1})$, one gets eventually the uniqueness $\mathbb{P}$-a.s. and a.e. in $D$ for all $m \in \mathbb{N}$.
	\end{proof}
	
	In order to obtain a priori estimates for $\{V^{m}\}_{m=1}^{M}$ in Sobolev spaces, independently of the discretization's parameters, we shall assume that there is a positive constant $\mathcal{C}$, independent of $\alpha$, $h$ and $k$ such that $0 < \alpha \leq \mathcal{C}h$. We will present in Lemma~\ref{lemma V < U} some preliminary estimates.
	
	\begin{lem}\label{lemma V < U}\ \\
		Let $\{(U^{m}, V^{m})\}_{m=1}^{M}$ be the iterates of Algorithm~\ref{Algorithm} and $0 < \alpha \leq \mathcal{C}h$, where $\mathcal{C}>0$ independent of $\alpha$, $h$ and $k$. Then, for all $m \in \{1, \dotsc, M\}$ and $\mathbb{P}$-a.s.
		\begin{enumerate}
			\item[(i)] $\displaystyle \left|\left|V^{m}\right|\right|_{\mathbb{L}^{2}} \leq \mathcal{C}_{1}\left|\left|U^{m}\right|\right|_{\alpha}$,
			\item[(ii)] $\displaystyle \left|\left|\nabla V^{m}\right|\right|_{\mathbb{L}^{2}}^{2} \leq \mathcal{C}_{1}\left(\left|\left|\nabla U^{m}\right|\right|^{2}_{\mathbb{L}^{2}} + \alpha^{2}\left|\left|\Delta^{h}U^{m}\right|\right|^{2}_{\mathbb{L}^{2}}\right)$,
			\item[(iii)] $\displaystyle \left|\left|V^{m+\ell} - V^{m}\right|\right|_{\mathbb{L}^{2}} \leq \mathcal{C}_{1}\left|\left|U^{m+\ell} - U^{m}\right|\right|_{\alpha}$, for all $\ell \in \{1, \dotsc, M-m\}$,
		\end{enumerate}
		where $\mathcal{C}_{1}>0$ depends only on $\mathcal{C}$ and the constant $C$ of the inverse inequality~\eqref{eq inverse estimate}.
	\end{lem}
	\begin{proof}
		Let $m \in \{1, \dotsc, M\}$. From equation~\eqref{eq3.3.1}$_{2}$, taking $\psi = V^{m}$ and applying the Cauchy-Schwarz and Young inequalities yield $||V^{m}||_{\mathbb{L}^{2}}^{2} \leq ||U^{m}||_{\mathbb{L}^{2}}^{2} + \frac{1}{4}||V^{m}||^{2}_{\mathbb{L}^{2}} + \frac{\alpha^{2}}{\epsilon}||\nabla U^{m}||^{2}_{\mathbb{L}^{2}} + \frac{\epsilon\alpha^{2}}{4}||\nabla V^{m}||^{2}_{\mathbb{L}^{2}}$, where $\epsilon > 0$. Taking $\epsilon = \frac{1}{\mathcal{C}^{2}C^{2}}$ and applying the inverse inequality~\eqref{eq inverse estimate} complete the proof of assertion $(i)$. On the other hand, by Lemma~\ref{lemma discrete diff filter Hh}-$(i)$, $\nabla V^{m} = \nabla U^{m} - \alpha^{2}\nabla \Delta^{h}U^{m}$, $\mathbb{P}$-a.s. and a.e. in $D$. Thus, $||\nabla V^{m}||_{\mathbb{L}^{2}}^{2} \leq 2||\nabla U^{m}||^{2}_{\mathbb{L}^{2}} + 2\mathcal{C}^{2}C^{2}\alpha^{2}||\Delta^{h}U^{m}||^{2}_{\mathbb{L}^{2}}$, thanks to the inverse inequality~\eqref{eq inverse estimate}. Estimate~$(iii)$ has similar proof to that of assertion~$(i)$.
	\end{proof}
	
	Clearly, one must incorporate Lemmas~\ref{lemma3.3.1} and \ref{lemma V < U} to obtain:
	\begin{lem}\label{lemma V a priori estimates}\ \\
		Let $\{V^{m}\}_{m=1}^{M}$ be the iterates of Algorithm~\ref{Algorithm}. Assume that assumptions $(S_{1})$-$(S_{3})$ are fulfilled and that $0 < \alpha \leq \mathcal{C}h$, for a certain $\mathcal{C}>0$ independent of $\alpha, k$ and $h$. Then, 
		$$\displaystyle \mathbb{E}\left[\max\limits_{1 \leq m \leq M}\left|\left|V^{m}\right|\right|^{2}_{\mathbb{L}^{2}} + \frac{k\nu}{2}\sum_{m=1}^{M}\left|\left|\nabla V^{m}\right|\right|^{2}_{\mathbb{L}^{2}} + \frac{1}{4}\sum_{m = 1}^{M}\left|\left|V^{m} - V^{m-1}\right|\right|^{2}_{\mathbb{L}^{2}}\right] \leq C'_{T},$$ where $C'_{T} > 0$ does not depend on $\alpha, k$ and $h$.
	\end{lem}
	
	We end up this section with the following a priori estimate for $\{V^{m}\}_{m=1}^{M}$, where the scale $\alpha$ is not necessarily assumed to be controlled by $h$.
	\begin{lem}\label{lemma a priori estimate V alpha indep.}\ \\
		Assume $(S_{1})$-$(S_{3})$ and let $\{V^{m}\}_{m=1}^{M}$ be the iterates of Algorithm~\ref{Algorithm}. Then, $$\displaystyle \mathbb{E}\left[k\sum_{m=1}^{M}\left|\left|V^{m}\right|\right|^{2}_{\mathbb{L}^{2}}\right] \leq C_{T}.$$
	\end{lem}
	\begin{proof}
		From equation~\eqref{eq3.3.1}, replacing $\psi$ by $V^{m}$ and employing Lemma~\ref{lemma discrete diff filter Hh}-$(ii)$ lead to $||V^{m}||^{2}_{\mathbb{L}^{2}} = ||U^{m}||_{\alpha}^{2} + \alpha^{2}\left(||\nabla U^{m}||^{2}_{\mathbb{L}^{2}} + \alpha^{2}||\Delta^{h}U^{m}||_{\mathbb{L}^{2}}^{2}\right)$. Multiplying by $k$, summing both sides over $m$ from $1$ to $M$, using the condition $\alpha \leq 1$, then applying the mathematical expectation and Lemma~\ref{lemma3.3.1}-$(ii)$ complete the proof.
	\end{proof}
	
	\section{Convergence}\label{section5}
	All the previous interpretation relied on $\{(U^{m}, V^{m})\}_{m=1}^{M}$, which does not depend explicitly on the time variable. To investigate the convergence in continuous-time spaces, e.g. $L^{2}(\Omega; L^{2}(0,T; \mathbb{H}^{1}))$, we need to define the following processes
	\begin{align}
		&\mathcal{V}_{k,h}(t,x) \coloneqq \frac{t - t_{m-1}}{k}V^{m}(x) + \frac{t_{m} - t}{k}V^{m-1}(x), \ \ \forall \ (t,x) \in [t_{m-1}, t_{m}) \times D, \label{eq new process}
		\\&\left(\mathcal{U}_{k,h}^{-}(t,x), \mathcal{V}_{k,h}^{-}(t,x)\right) \coloneqq \left(U^{m-1}(x), V^{m-1}(x)\right), \  \forall (t,x) \in [t_{m-1}, t_{m})\times D, \\& \left(\mathcal{U}^{+}_{k,h}(t,x), \mathcal{V}_{k,h}^{+}(t,x)\right) \coloneqq \left(U^{m}(x), V^{m}(x)\right), \ \   \forall (t,x) \in (t_{m-1}, t_{m}] \times D.
	\end{align}
	Note that $\mathcal{V}_{k,h}$ is continuous on $[0,T]$,
	\begin{align}
		&\mathcal{V}_{k,h}(t,\cdot) - \mathcal{V}_{k,h}^{-}(t,\cdot) = \frac{t - t_{m-1}}{k}\left(V^{m} - V^{m-1}\right) \mbox{ for all } t \in [t_{m-1}, t_{m}), \label{eq new process prop1}\\&\mbox{ and }\notag\\&
		\mathcal{V}_{k,h}(t,\cdot) - \mathcal{V}_{k,h}^{+}(t,\cdot) = \frac{t - t_{m}}{k}\left(V^{m} - V^{m-1}\right) \mbox{ for all } t \in (t_{m-1}, t_{m}].\label{eq new process prop2}
	\end{align}
	Plugging both new processes in equation \eqref{eq3.3.1}, we get for every $t \in [t_{m-1}, t_{m})$, $\left(\varphi, \psi\right) \in \mathbb{V}_{h}\times\mathbb{V}_{h}$ and $\mathbb{P}$-a.s. the following:
	\begin{equation*}
		\begin{cases}
			\begin{aligned}
				\bullet &\left(\mathcal{V}_{k,h} - \mathcal{V}_{k,h}^{-}, \varphi\right) + \nu\int_{t_{m-1}}^{t}\left(\nabla \mathcal{V}_{k,h}^{+}, \nabla \varphi\right)ds + \int_{t_{m-1}}^{t}\tilde{b}(\mathcal{U}_{k,h}^{+}, \mathcal{V}_{k,h}^{-}, \varphi)ds \\&= \int_{t_{m-1}}^{t}\Big\langle f^{-}(t, \mathcal{U}^{-}_{k,h}), \varphi \Big\rangle ds + \frac{t - t_{m-1}}{k}\Big(\int_{t_{m-1}}^{t_{m}}g^{-}(s,\mathcal{U}_{k,h}^{-})dW(s), \varphi\Big),
			\end{aligned}\\
			\bullet \left(\mathcal{V}_{k,h}^{+}, \psi\right) = \left(\mathcal{U}_{k,h}^{+}, \psi\right) + \alpha^{2}\left(\nabla\mathcal{U}_{k,h}^{+}, \nabla \psi \right).
		\end{cases}
	\end{equation*}
	where $\displaystyle f^{-}(t, \cdot) = f(t_{m-1}, \cdot)$ and $\displaystyle g^{-}(t, \cdot) = g(t_{m-1}, \cdot)$ for all $t \in [t_{m-1}, t_{m})$ and $m \in \{1, \dotsc, M\}$.
	
	Since Lemmas~\ref{lemma convergence} and \ref{lemma convergence discrete} are essential to retrieve a strong convergence in spaces that do not depend on randomness, we must consider applying them in a probability subset. We shall give first a dedicated configuration for the case $0 < \alpha \leq \mathcal{C}h$. It will be adjusted afterward to fit the converse case. 
	
	Denote by $\displaystyle\mathscr{U}^{k,h} \coloneqq k\sum_{m = 1}^{M}\left|\left|\nabla U^{m}\right|\right|_{\mathbb{L}^{2}}^{2}$ and by $\displaystyle\mathscr{V}^{k,h} \coloneqq k\sum_{m = 1}^{M}\left|\left|\nabla V^{m}\right|\right|_{\mathbb{L}^{2}}^{2}$. Both law families $\{\mathscr{L}(\mathscr{U}^{k,h})\}_{k,h}$ and $\{\mathscr{L}(\mathscr{V}^{k,h})\}_{k,h}$ are tight. Indeed, for $\varepsilon > 0$, we consider the closed $\mathbb{R}$-ball $K_{\varepsilon} \coloneqq \bar{B}_{\mathbb{R}}(0, 1/\varepsilon)$. Clearly, $K_{\varepsilon}$ is a compact set of $\mathbb{R}$, and $$\mathscr{L}(\mathscr{U}^{k,h})(K_{\varepsilon}) = \mathbb{P}(\mathscr{U}^{k,h} \in K_{\varepsilon}) = \mathbb{P}(\{\omega \in \Omega \ | \ k\sum_{m=1}^{M}||\nabla U^{m}||^{2}_{\mathbb{L}^{2}} \leq 1/\varepsilon\}) \geq 1-C_{T}\varepsilon,$$ for all $k, h > 0$, thanks to the Markov inequality and Lemma~\ref{lemma3.3.1}-$(ii)$. Similarly, by Lemma~\ref{lemma V a priori estimates} and the Markov inequality, the tightness of $\{\mathscr{L}(\mathscr{V}^{k,h})\}_{k,h}$ follows. Such an argument is summoned just to emphasize the non-dependence of $\varepsilon$ on $k$ and $h$. Concerning the Wiener increments' tightness, we define $W_{k,h}\eqqcolon \frac{1}{k}\max\limits_{1 \leq m \leq M}\left|\left|\Delta_{m}W\right|\right|^{2}_{K}$, and consider the same compact set $K_{\varepsilon}$. Therefore, by virtue of estimate~\eqref{eq 2.3.4}, $$\mathscr{L}(W_{k,h})(K_{\varepsilon}^{W}) = \mathbb{P}(W_{k,h} \leq 1/\varepsilon) \geq 1 - Tr(Q)\varepsilon,$$ for all $k,h > 0$, thanks again to the Markov inequality.
	With that being said, we fix $\varepsilon > 0$, and define the sample subset $\Omega_{k,h}^{\varepsilon} \coloneqq \Omega_{k,h}^{1}\cap \Omega_{k,h}^{2} \subset \Omega$, where $\Omega_{k,h}^{1} \coloneqq \Big\{\omega \in \Omega \ \big| \ W_{k,h} \leq 1/\varepsilon\Big\}$, and $\displaystyle\Omega_{k,h}^{2} \coloneqq \Big\{\omega \in \Omega \ \big| \ \int_{0}^{T}\left(||\nabla \mathcal{U}^{+}_{k,h}||^{2}_{\mathbb{L}^{2}} + ||\nabla \mathcal{V}^{+}_{k,h}||^{2}_{\mathbb{L}^{2}}\right)dt \leq 1/\varepsilon\Big\}$. In the light of the preceding analysis, there is a $c>0$ depending only on $Tr(Q), C_{T}$ and $C_{T}'$ such that
	\begin{equation}\label{eq sample subset measure}
		\mathbb{P}(\Omega_{k,h}^{\varepsilon}) \geq 1 - c\varepsilon.
	\end{equation}
	
	The next lemma fulfills one of the hypotheses of Lemmas~\ref{lemma convergence} and \ref{lemma convergence discrete} in the sample subset $\Omega_{k,h}^{\varepsilon}$. Note that assumption $\alpha \leq \mathcal{C}h$ can be replaced by an alternative hypothesis $\sqrt{k}/h < L \leq \alpha$, for some $L \in (0,1)$ along with a minor modification in $\Omega_{k,h}^{\varepsilon}$, as explained in Remark\ref{rmk k-h relation}. Such conditions are mandatory if we are aiming at converging Algorithm~\ref{Algorithm} toward the LANS-$\alpha$ equations instead of the NSEs.
	\begin{lem}\label{lemma translation property}\ \\
		Let $0 < \alpha \leq \mathcal{C}h$, for a certain $\mathcal{C}>0$ independent of $\alpha, h$ and $k$. For almost every $\omega \in \Omega_{k,h}^{\varepsilon}$, and $\ell \in \{1, \dotsc, M-1\}$,
		$$k\sum_{m=1}^{M-\ell}\left|\left|U^{m+\ell}(\omega) - U^{m}(\omega)\right|\right|_{\alpha}^{2} \leq Ct_{\ell}^{\frac{1}{2}},$$ where $C= C(T, \nu, D, K_{1}, K_{2}, K_{3}, K_{4}, \varepsilon)>0$ independent of $\alpha$, $k$ and $h$.
	\end{lem}
	\begin{proof}
		Let $\ell \in \{1, \dotsc, M-1\}$, and $\omega \in \Omega_{k,h}^{\varepsilon}$. We replace the index $m$ by $i$ in \eqref{eq3.3.1}, and sum it over $i$ from $m+1$ to $m+\ell$. We get for all $\varphi \in \mathbb{V}_{h}$ and $\mathbb{P}$-a.s.
		\begin{equation}\label{eq modified}
			\begin{aligned}
				&\left(V^{m+\ell} - V^{m}, \varphi\right) + k\nu \sum_{i=1}^{\ell}\left(\nabla V^{m+i}, \nabla \varphi\right) + k\sum_{i=1}^{\ell}\tilde{b}(U^{m+i}, V^{m+i-1}, \varphi) \\&= k\sum_{i=1}^{\ell}\langle f(t_{m+i-1}, U^{m+i-1}), \varphi \rangle + \sum_{i=1}^{\ell}\left(g(t_{m+i-1},U^{m+i-1})\Delta_{m+i}W, \varphi\right).
			\end{aligned}
		\end{equation}
		Set $\varphi = U^{m+\ell} - U^{m}$, sum \eqref{eq modified} up over $m$ from $1$ to $M-\ell$, then multiply by $k$:
		\begin{align*}
			&k\sum_{m=1}^{M-\ell}\left|\left|U^{m+\ell} - U^{m}\right|\right|_{\alpha}^{2} = -  k^{2}\nu \sum_{m=1}^{M-\ell}\sum_{i=1}^{\ell}\left(\nabla V^{m+i}, \nabla (U^{m+\ell} - U^{m})\right) \\&- k^{2}\sum_{m=1}^{M-\ell}\sum_{i=1}^{\ell}\tilde{b}(U^{m+i}, V^{m+i-1}, U^{m+\ell} - U^{m}) \\&+ k^{2}\sum_{m=1}^{M-\ell}\sum_{i=1}^{\ell}\big\langle f(t_{m+i-1}, U^{m+i-1}), U^{m+\ell} - U^{m} \big\rangle \\&+ k\sum_{m=1}^{M-\ell}\sum_{i=1}^{\ell}\left(g(t_{m+i-1}, U^{m+i-1})\Delta_{m+i}W, U^{m+\ell} - U^{m}\right) = I + II + III + IV. 
		\end{align*}
		Our aim is to bound each term by $t_{\ell}^{1/2}$.
		For the term $I$, we use the Cauchy-Schwarz inequality, $|I| \leq 2\nu t_{\ell}^{1/2} k^{3/2}\sum_{m=1}^{M}||\nabla U^{m}||_{\mathbb{L}^{2}}\left(\sum_{i=1}^{M}||\nabla V^{i}||^{2}_{\mathbb{L}^{2}}\right)^{1/2} \leq \frac{2\nu T^{1/2}}{\varepsilon}t_{\ell}^{1/2}.$
		On the other hand,
		\begin{align*}
			\left|II\right| &\leq C_{D}k^{\frac{1}{2}}t_{\ell}^{\frac{1}{2}}k\sum_{m=1}^{M-\ell}||\nabla (U^{m+\ell} - U^{m})||_{\mathbb{L}^{2}}\left(\sum_{i=1}^{\ell}||\nabla U^{m+i}||^{2}_{\mathbb{L}^{2}}||\nabla V^{m+i-1}||^{2}_{\mathbb{L}^{2}}\right)^{1/2} \\&\leq 2C_{D}k^{\frac{1}{2}}t_{\ell}^{\frac{1}{2}}k\sum_{m=1}^{M}||\nabla U^{m}||_{\mathbb{L}^{2}}\left(\sum_{i=1}^{M}||\nabla U^{i}||^{2}_{\mathbb{L}^{2}}\right)^{1/2}\left(\sum_{i=1}^{M}||\nabla V^{i}||^{2}_{\mathbb{L}^{2}}\right)^{1/2} \\&\leq 6C_{D}t_{\ell}^{\frac{1}{2}}k\sum_{m=1}^{M}||\nabla U^{m}||^{2}_{\mathbb{L}^{2}}\left(k\sum_{i=1}^{M}||\nabla V^{i}||^{2}_{\mathbb{L}^{2}}\right)^{1/2} \leq \frac{6C_{D}}{\varepsilon\sqrt{\varepsilon}}t_{\ell}^{\frac{1}{2}},
		\end{align*}
		where Proposition~\ref{prop trilinear form}-$(iii)$, the Cauchy-Schwarz inequality for sums, $\ell^{1}(\mathbb{N}) \hookrightarrow \ell^{2}(\mathbb{N})$ and estimate~\eqref{eq sum estimate} were employed. For the term $III$, we need assumption $(S_{2})$, the norm equivalence~\eqref{eq norm equivalence}, the Poincar{\'e} and Cauchy-Schwarz inequalities, estimate~\eqref{eq sum estimate} and $\ell^{1}(\mathbb{N}) \hookrightarrow \ell^{2}(\mathbb{N})$:
		\begin{align*}
			&\left|III\right| \leq C_{D}k\sum_{m=1}^{M-\ell}||\nabla (U^{m + \ell} - U^{m})||_{\mathbb{L}^{2}}k\sum_{i=1}^{\ell}\left(K_{3} + C_{D}K_{4}||\nabla U^{m+i-1}||_{\mathbb{L}^{2}}\right) \\&\leq 2C_{D}t_{\ell}^{\frac{1}{2}}K_{3}T\left(k\sum_{m=1}^{M}||\nabla U^{m}||^{2}_{\mathbb{L}^{2}}\right)^{1/2} + 6C_{D}t_{\ell}^{\frac{1}{2}}K_{4}k^{\frac{1}{2}}k\sum_{m=1}^{M}||\nabla U^{m}||^{2}_{\mathbb{L}^{2}} \\&\leq 2C_{D}(K_{3}T/\sqrt{\varepsilon} + 3K_{4}T^{\frac{1}{2}}/\varepsilon)t_{\ell}^{\frac{1}{2}}.
		\end{align*}
		For the term $IV$, we need the Cauchy-Schwarz inequality, assumption $(S_{2})$, $\ell^{1}(\mathbb{N}) \hookrightarrow \ell^{2}(\mathbb{N})$ and inequality~\eqref{eq sum estimate}:
		\begin{align*}
			\left|IV\right| &\leq C_{D}K_{2}k^{\frac{1}{2}}\sum_{m=1}^{M-\ell}||U^{m-\ell} - U^{m}||_{\mathbb{L}^{2}}\left(\sum_{i=1}^{\ell}||\nabla U^{m+i-1}||^{2}_{\mathbb{L}^{2}}||\Delta_{m+i}W||^{2}_{K}\right)^{1/2}t_{\ell}^{\frac{1}{2}} \\&\hspace{10pt}+ 2K_{1}t_{\ell}\sum_{m=1}^{M}||U^{m}||_{\mathbb{L}^{2}}\max\limits_{1 \leq m \leq M}||\Delta_{m}W||_{K}
			\leq \frac{2K_{1}C_{D}}{\varepsilon}t_{\ell} + \frac{2C_{D}K_{2}}{\varepsilon\sqrt{\varepsilon}}t_{\ell}^{\frac{1}{2}}.
		\end{align*}
		Collecting all the above estimates together yields the final result.
	\end{proof}
	
	\begin{rmk}\label{rmk k-h relation}
		Observe that only terms $I$ and $II$ in the above proof required the use of the condition $\alpha \leq \mathcal{C}h$. We avert this assumption as follows: firstly, we assume the relation $\frac{\sqrt{k}}{h} < L \leq \alpha$, for a certain $L \in (0,1)$ independent of $k$ and $h$. Secondly, since Lemma~\ref{lemma V a priori estimates} is no longer valid, we adjust the sample subset $\Omega_{k,h}^{2}$ by replacing $||\nabla \mathcal{V}^{+}_{k,h}||_{\mathbb{L}^{2}}^{2}$ with $||\mathcal{V}^{+}_{k,h}||^{2}_{\mathbb{L}^{2}}$. We also need to intersect $\Omega_{k,h}^{\varepsilon}$ with a supplementary sample subset $\Omega_{k,h}^{3}\coloneqq \{\omega \in \Omega \ | \ \max\limits_{1 \leq m \leq M}||U^{m}||^{2}_{\alpha} \leq 1/\varepsilon\}$.\newline
		Indeed, by virtue of Lemmas~\ref{lemma3.3.1}-$(ii)$, \ref{lemma a priori estimate V alpha indep.} and the Markov inequality, there holds $\mathbb{P}(\Omega_{k,h}^{\varepsilon}) \geq 1 - c\varepsilon$, for some $c>0$ depending only on $Tr(Q)$ and $C_{T}$. We recall that $\varepsilon$ does not depend on $k$ and $h$. For the term $I$, we use inequalities~\eqref{eq inverse estimate},\eqref{eq sum estimate} and some elementary calculations: $|I| \leq \frac{2C\nu k}{h}t_{\ell}^{1/2}\sum_{m=1}^{M}||\nabla U^{m}||_{\mathbb{L}^{2}}\left(k\sum_{i=1}^{M}||V^{i}||^{2}_{\mathbb{L}^{2}}\right)^{1/2} < \frac{2\sqrt{3}CL\nu}{\varepsilon}t_{\ell}^{1/2}.$ On the other hand,	from Proposition~\ref{prop trilinear form}-$(iii)$, the Poincar{\'e} and Cauchy-Schwarz inequalities, estimate~\eqref{eq sum estimate} and the inverse inequality~\eqref{eq inverse estimate}, 
		\begin{align*}
			|II| &\leq \frac{2CC_{D}k}{h}t_{\ell}^{1/2}\sum_{m=1}^{M}||\nabla U^{m}||_{\mathbb{L}^{2}}\left(k\sum_{i=2}^{M}||V^{i-1}||_{\mathbb{L}^{2}}^{2}||\nabla U^{i}||_{\mathbb{L}^{2}}^{2}\right)^{1/2} \\& \leq 2CC_{D}\frac{1}{\sqrt{\varepsilon}}t_{\ell}^{1/2}\max\limits_{1 \leq i \leq M}\alpha||\nabla U^{i}||_{\mathbb{L}^{2}} \sqrt{k}\sum_{m=1}^{M}||\nabla U^{m}||_{\mathbb{L}^{2}} \leq \frac{2\sqrt{3}CC_{D}}{\varepsilon\sqrt{\varepsilon}}t_{\ell}^{1/2}.
		\end{align*}
	\end{rmk}
	
	We summarize, in the upcoming proposition, all the convergence results emerging from the condition $\alpha \leq \mathcal{C}h$. We give afterwards, in Proposition~\ref{prop convergence U k-h}, convergence outcomes concerning the case $\sqrt{k}/h < L \leq \alpha$.
	\begin{prop}\label{prop limits of U and V}\ \\
		Assume $0 < \alpha \leq \mathcal{C}h$ for some $\mathcal{C}>0$ independent of $\alpha, k$ and $h$. There exist $u ,v \in L^{2}(\Omega; L^{2}(0,T; \mathbb{H}^{1}_{0}))$ such that the convergences below occur, up to extractions, as $k,h \to 0$:
		\begin{equation}\label{eq U convergence}
			\begin{aligned}
				&\mathcal{U}^{+}_{k,h} \rightharpoonup u \ \ \ \& \ \ \ \mathcal{V}^{+}_{k,h} \rightharpoonup v \ \ \mbox{ in } \ \ L^{2}(\Omega; L^{2}(0,T; \mathbb{H}^{1}_{0})),
				\\& \mathcal{U}_{k,h}^{+} \to u \ \ \ \& \ \ \ \mathcal{V}^{+}_{k,h} \to v \ \ \mbox{ in } \ \ L^{2}(\Omega; L^{2}(0,T;\mathbb{L}^{2})),
				\\& \mathcal{V}_{k,h} \to v \ \ \mbox{ in } L^{2}(\Omega; L^{2}(0,T;\mathbb{L}^{2})),
				\\& \mathcal{V}^{-}_{k,h} \to v \ \ \mbox{ in } \ \ L^{2}(\Omega; L^{2}(0,T;\mathbb{L}^{2})).
			\end{aligned}
		\end{equation}
	\end{prop}
	\begin{proof}
		We point out that all subsequences in this proof will be denoted as their original sequences for the sake of clarity.
		By virtue of Lemmas~\ref{lemma3.3.1}-(ii) and \ref{lemma V a priori estimates}, $\{\mathcal{U}^{+}_{k,h}\}_{k,h}$ and $\{\mathcal{V}_{k,h}^{+}\}_{k,h}$ are bounded in the Hilbert space $L^{2}(\Omega; L^{2}(0,T; \mathbb{H}^{1}_{0}))$. This implies the existence of two subsequences of $\{\mathcal{U}_{k,h}^{+}\}_{k,h}$ and $\{\mathcal{V}^{+}_{k,h}\}_{k,h}$ such that $\mathcal{V}^{+}_{k,h} \rightharpoonup v$ and
		$\mathcal{U}^{+}_{k,h} \rightharpoonup u$  in $L^{2}(\Omega; L^{2}(0,T; \mathbb{H}^{1}_{0}))$ for some functions $v$ and $u$ belonging to the same space of convergence. To justify \eqref{eq U convergence}$_{2}$, we need to apply Lemma~\ref{lemma convergence}. Indeed, we have $\mathbb{H}^{1}_{0} \hookrightarrow \mathbb{L}^{2}$ compactly, $\{\mathbb{1}_{\Omega_{k,h}^{\varepsilon}}\mathcal{U}^{+}_{k,h}\}_{k,h}$ is $\mathbb{P}$-a.s. bounded in $L^{2}(0,T; \mathbb{H}^{1}_{0})$, and for all $\ell \in \{1, \dotsc, M-1\}$, there holds $\mathbb{P}$-a.s.
		\begin{align*}
			\left|\left|\tau_{t_{\ell}}\mathbb{1}_{\Omega_{k,h}^{\varepsilon}}\mathcal{U}^{+}_{k,h} - \mathbb{1}_{\Omega_{k,h}^{\varepsilon}}\mathcal{U}_{k,h}^{+}\right|\right|_{L^{2}(0,T - t_{\ell}; \mathbb{L}^{2})}^{2} = \mathbb{1}_{\Omega_{k,h}^{\varepsilon}}k\sum_{m=1}^{M-\ell}\left|\left|U^{m+\ell} - U^{m}\right|\right|_{\mathbb{L}^{2}}^{2} \leq Ct_{\ell}^{\frac{1}{2}} \xrightarrow[t_{\ell} \to 0]{}0, 
		\end{align*}
		uniformly in $\mathcal{U}_{k,h}^{+}$, thanks to Lemma~\ref{lemma translation property}. As a result, $\{\mathbb{1}_{\Omega_{k,h}^{\varepsilon}}\mathcal{U}_{k,h}^{+}\}_{k,h}$ is $\mathbb{P}$-a.s. relatively compact in $L^{2}(0,T; \mathbb{L}^{2})$. In other words, $\{||\mathbb{1}_{\Omega_{k,h}^{\varepsilon}}\mathcal{U}_{k,h}^{+}||_{L^{2}(0,T; \mathbb{L}^{2})}^{2}\}_{k,h}$ is $\mathbb{P}$-a.s. convergent. Moreover, $||\mathbb{1}_{\Omega_{k,h}^{\varepsilon}}\mathcal{U}_{k,h}^{+}||^{2}_{L^{2}(0,T: \mathbb{L}^{2})} \leq 1/\varepsilon$, $\mathbb{P}$-a.s. where the function $\omega \mapsto \frac{1}{\varepsilon} \in L^{1}(\Omega)$. Therefore, by the dominated convergence theorem, $\{||\mathbb{1}_{\Omega_{k,h}^{\varepsilon}}\mathcal{U}^{+}_{k,h}||_{L^{2}(\Omega; L^{2}(0,T; \mathbb{L}^{2}))}\}_{k,h}$ is convergent. We need to prove in addition to the latter that $\{\mathbb{1}_{\Omega_{k,h}^{\varepsilon}}\mathcal{U}^{+}_{k,h}\}_{k,h}$ is weakly convergent in $L^{2}(\Omega; L^{2}(0,T; \mathbb{L}^{2}))$ to achieve strong convergence. We have $\mathbb{E}\left[\int_{0}^{T}||\mathbb{1}_{(\Omega_{k,h}^{\varepsilon})^{c}}||^{2}_{\mathbb{L}^{2}}dt\right] = |D|T\mathbb{P}((\Omega_{k,h}^{\varepsilon})^{c}) \lesssim |D|T\varepsilon$, for all $\varepsilon>0$, thanks to estimate~\eqref{eq sample subset measure}. By \eqref{eq U convergence}$_{1}$, one gets $\mathcal{U}^{+}_{k,h} \rightharpoonup u$ in $L^{2}(\Omega; L^{2}(0,T; \mathbb{L}^{2}))$. Thus, $\mathbb{E}\left[\int_{0}^{T}\left(\mathbb{1}_{\Omega_{k,h}^{\varepsilon}}\mathcal{U}^{+}_{k,h}, \varphi\right)dt\right] = -\mathbb{E}\left[\int_{0}^{T}\left(\mathcal{U}_{k,h}^{+}, \mathbb{1}_{(\Omega_{k,h}^{\varepsilon})^{c}}\varphi\right)dt\right] + \mathbb{E}\left[\int_{0}^{T}\left(\mathcal{U}_{k,h}^{+}, \varphi\right)dt\right]$, for all $\varphi \in L^{2}(\Omega; L^{2}(0,T; \mathbb{L}^{2}))$, which converges toward the term $\mathbb{E}\left[\int_{0}^{T}\left(u(t), \varphi\right)dt\right]$. Note that the test function $\varphi$ can be considered in a dense set of $L^{2}(\Omega; L^{2}(0,T; \mathbb{L}^{2})$ since $\{\mathbb{1}_{\Omega_{k,h}^{\varepsilon}}\mathcal{U}^{+}_{k,h}\}_{k,h}$ is bounded in the latter space. Subsequently, $\mathbb{1}_{\Omega_{k,h}^{\varepsilon}}\mathcal{U}^{+}_{k,h} \to u$ in $L^{2}(\Omega; L^{2}(0,T; \mathbb{L}^{2}))$. As a result,
		\begin{align*}
			\mathbb{E}\left[\int_{0}^{T}||\mathcal{U}^{+}_{k,h} - u||^{2}_{\mathbb{L}^{2}}dt\right] &\lesssim \mathbb{E}\left[\int_{0}^{T}\left(||\mathbb{1}_{(\Omega_{k,h}^{\varepsilon})^{c}}\mathcal{U}^{+}_{k,h}||^{2}_{\mathbb{L}^{2}} + ||\mathbb{1}_{\Omega_{k,h}^{\varepsilon}}\mathcal{U}^{+}_{k,h} - u||^{2}_{\mathbb{L}^{2}}\right)dt\right] \xrightarrow[k,h \to 0]{}0.
		\end{align*}
		This convergence takes place due to what we have shown so far in this proof and the boundedness of $\{\mathcal{U}^{+}_{k,h}\}_{k,h}$ in $L^{4}(\Omega; L^{\infty}(0,T; \mathbb{L}^{2}))$, thanks to Lemma~\ref{lemma3.3.1}-$(iii)$. The convergence of $\{\mathcal{V}^{+}_{k,h}\}_{k,h}$ is done similarly through Lemmas~\ref{lemma V < U}-$(iii)$ and \ref{lemma V a priori estimates}. Moving on to \eqref{eq U convergence}$_{3}$, we have
		\begin{align*}
			\mathbb{E}\left[\int_{0}^{T}\left|\left|\mathcal{V}_{k,h} - v\right|\right|^{2}_{\mathbb{L}^{2}}dt\right] &\lesssim \mathbb{E}\left[\int_{0}^{T}\left(||\mathcal{V}_{k,h} - \mathcal{V}_{k,h}^{+}||^{2}_{\mathbb{L}^{2}} + ||\mathcal{V}_{k,h}^{+} - v||^{2}_{\mathbb{L}^{2}}\right)dt\right] = I_{1} + I_{2}.
		\end{align*}
		From \eqref{eq U convergence}$_{2}$, we get $I_{2} \xrightarrow[k,h \to 0]{} 0$. On the other hand, identity~\eqref{eq new process prop2} yields
		$$I_{1} = \sum_{m=1}^{M}\mathbb{E}\left[||V^{m} - V^{m-1}||^{2}_{\mathbb{L}^{2}}\right]\int_{t_{m-1}}^{t_m}\frac{(t -  t_{m})^{2}}{k^{2}}dt \leq \frac{k}{3}C_{T} \xrightarrow[k,h \to 0]{} 0,$$
		thanks to Lemma~\ref{lemma V a priori estimates}. The proof of \eqref{eq U convergence}$_{4}$ follows from \eqref{eq new process prop1} and the proving technique of \eqref{eq U convergence}$_{3}$.
	\end{proof}
	
	\begin{rmk}
		In the three-dimensional stochastic NSEs framework, the obtained convergence results in proposition~\ref{prop limits of U and V} remain up to extractions due to the nonuniqueness of the corresponding solution, conversely to the the two-dimensional case whose solution is unique.
	\end{rmk}
	
	The limiting function $u_{\alpha}$ in the next proposition does not coincide with $u$ that was found in Proposition~\ref{prop limits of U and V}. It is worth mentioning that one can demonstrate the convergence of whole sequences $\{\mathcal{U}^{+}_{k,h}\}_{k,h}$, $\{\mathcal{V}^{+}_{k,h}\}_{k,h}$ and $\{\mathcal{V}_{k,h}\}_{k,h}$ once we verify that the limiting functions satisfy equation~\eqref{eq definiton solution modified}, $\mathbb{P}$-a.s. and for $t \in [0,T]$. Such an idea is true due to the solution's uniqueness of equations~\eqref{main equation} (see for instance \cite[Theorem 4.4]{Caraballo2006Takeshi}).
	
	\begin{prop}\label{prop convergence U k-h}\ \\
		Assume $0 < \sqrt{k}/h < L \leq \alpha$, for some $L \in (0,1)$ independent of $k$ and $h$. Then, there exist two functions $u_{\alpha} \in L^{2}(\Omega; L^{2}(0,T; \mathbb{H}^{2}\cap\mathbb{H}^{1}_{0}))$ and $v_{\alpha} \in L^{2}(\Omega; L^{2}(0,T; \mathbb{L}^{2}))$ such that, up to extractions and as $k, h \to 0$, one gets
		\begin{equation}\label{eq U convergence k-h}
			\begin{aligned}
				&\mathcal{U}^{+}_{k,h} \to u_{\alpha} \ \ \mbox{ in } \ \ L^{2}(\Omega; L^{2}(0,T; \mathbb{H}^{1}_{0})), \\& \mathcal{V}_{k,h}^{+} \rightharpoonup v_{\alpha} \ \ \mbox{ in } \ \ L^{2}(\Omega; L^{2}(0,T; \mathbb{L}^{2})),
				\\& \mathcal{V}^{-}_{k,h} \rightharpoonup v_{\alpha} \ \ \mbox{ in } \ \ L^{2}(\Omega; L^{2}(0,T; \mathbb{L}^{2})).
			\end{aligned}
		\end{equation}
		and $v_{\alpha} = u_{\alpha} - \alpha^{2}\Delta u_{\alpha}$, $\mathbb{P}$-almost surely and almost everywhere in $(0,T)\times D$.
	\end{prop}
	\begin{proof}
		Once again, all subsequences in this proof will be denoted as their original sequences. We define $M_{h}^{\alpha} \coloneqq \big\{u \in \mathbb{V}_{h} \ | \ \left|\left|u\right|\right|_{h,\alpha} < +\infty\big\}$, where $||\cdot||^{2}_{h,\alpha} \coloneqq ||\nabla\cdot||_{\mathbb{L}^{2}}^{2} + \alpha^{2}||\Delta^{h}\cdot||_{\mathbb{L}^{2}}^{2}$. Obviously, $M_{h}^{\alpha}$ is a subspace of $\mathbb{H}_{0}^{1}$, and $(M_{h}^{\alpha}, ||\cdot||_{h,\alpha})$ forms a normed space. Note that Lemma~\ref{lemma convergence} is no longer applicable since $M_{h}^{\alpha}$ depends on $h$. However, to come out with a strong convergence in $\mathbb{H}_{0}^{1}$, we shall apply Lemma~\ref{lemma convergence discrete} within the sample subset $\Omega_{k,h}^{\varepsilon}$ that was exclusively modified for the case $\sqrt{k}/h < L$ (see Remark~\ref{rmk k-h relation}) . We begin by proving the relative compactness of $M_{h}^{\alpha}$ in $\mathbb{H}_{0}^{1}$: let $(\varphi_{h})_{h}$ be a bounded sequence in $(M_{h}^{\alpha}, ||\cdot||_{h,\alpha})$. Therefore, $(\varphi_{h})_{h}$ (resp. $(\Delta^{h}\varphi_{h})_{h}$) converges weakly as $h \to 0$, in $\mathbb{H}_{0}^{1}$ (resp. $\mathbb{L}^{2}$) toward a function $\varphi$ (resp. $z$), up to an extraction. Let $\psi \in [C_{c}^{\infty}(D)]^{d}$. By identities~\ref{def projection} and \ref{def discrete Laplace}, $\left(\Delta^{h}\varphi_{h}, \psi\right) = \left(\Delta^{h}\varphi_{h}, \Pi_{h}\psi\right) = - \left(\nabla \varphi_{h}, \nabla \Pi_{h}\psi\right) \to -\left(\nabla \varphi, \nabla \psi\right) = \left(\Delta \varphi, \psi\right)$, thanks to estimate~\eqref{eqprojection}. Therefore, $z = \Delta\varphi$ a.e. in $D$. Note that $\varphi \in \mathbb{H}^{2}\cap \mathbb{H}^{1}_{0}$. The strong convergence of $(\varphi_{h})_{h}$ in $\mathbb{H}_{0}^{1}$ follows from its weak convergence along with the property: $||\nabla \varphi_{h}||^{2}_{\mathbb{L}^{2}} = -\left(\Delta^{h}\varphi_{h}, \varphi_{h}\right) \to -\left(\Delta \varphi, \varphi\right) = ||\nabla\varphi||^{2}_{\mathbb{L}^{2}}$, where we used identity~\ref{def discrete Laplace} and the weak and strong (which arises from the compact embedding $\mathbb{H}_{0}^{1} \hookrightarrow \mathbb{L}^{2}$) convergences of $(\Delta^{h}\varphi_{h})_{h}$ and $(\varphi_{h})_{h}$ in $\mathbb{L}^{2}$, respectively. On the other hand, $\{\mathbb{1}_{\Omega_{k,h}^{\varepsilon}}\mathcal{U}^{+}_{k,h}\}_{k,h}$ is $\mathbb{P}$-a.s. bounded in $L^{2}(0,T; M^{\alpha}_{h})$ and $L^{2}(0,T; \mathbb{H}_{0}^{1})$, thanks to the definition of $\Omega_{k,h}^{\varepsilon}$. Moreover, by Lemma~\ref{lemma translation property}, $\mathbb{1}_{\Omega_{k,h}^{\varepsilon}}||\tau_{t_{\ell}}\mathcal{U}^{+}_{k,h} - \mathcal{U}^{+}_{k,h}||_{L^{2}(0,T-t_{\ell}; \mathbb{H}^{1}_{0})} \to 0$, as $t_{\ell} \to 0$, uniformly in $k$ and $h$. The latter convergence holds in $\mathbb{H}^{1}_{0}$ due to the norm equivalence~\eqref{eq norm equivalence}. Subsequently, all conditions of Lemma~\ref{lemma convergence discrete} are met, we infer that $\{\mathbb{1}_{\Omega_{k,h}^{\varepsilon}}\mathcal{U}^{+}_{k,h}\}_{k,h}$ is $\mathbb{P}$-a.s. relatively compact in $L^{2}(0,T; \mathbb{H}_{0}^{1})$. Hence, convergence~\eqref{eq U convergence k-h}$_{1}$ can be justified as in the proof of Proposition~\ref{prop limits of U and V}, where we recall that $||\cdot||_{\alpha}$ is equivalent to $||\cdot||_{\mathbb{H}_{0}^{1}}$ because $\alpha$ does not depend on $k$ and $h$. Let $u_{\alpha}$ denote the limit of $\{\mathcal{U}^{+}_{k,h}\}_{k,h}$ in $L^{2}(\Omega; L^{2}(0,T; \mathbb{H}^{1}_{0}))$. By Lemma~\ref{lemma3.3.1}-$(ii)$, $\{\Delta^{h}\mathcal{U}^{+}_{k,h}\}_{k,h}$ converges weakly, up to a subsequence, to a function $z$ in $L^{2}(\Omega; L^{2}(0,T; \mathbb{L}^{2})).$ As done earlier in this proof, one can prove that $z = \Delta u_{\alpha}$. Thus, $\mathbb{E}\left[\int_{0}^{T}||\Delta u_{\alpha}(t)||^{2}_{\mathbb{L}^{2}}dt\right] \leq \liminf\mathbb{E}\left[\int_{0}^{T}||\Delta^{h}\mathcal{U}^{+}_{k,h}||^{2}_{\mathbb{L}^{2}}dt\right] \leq \alpha^{-2}C_{T}$, which implies $u_{\alpha} \in L^{2}(\Omega; L^{2}(0,T; \mathbb{H}^{2}\cap\mathbb{H}_{0}^{1}))$, thanks to the domain's properties. Moving on to convergence~\eqref{eq U convergence k-h}$_{2}$, by Lemma~\ref{lemma a priori estimate V alpha indep.}, $\{\mathcal{V}^{+}_{k,h}\}_{k,h}$ converges weakly, up to a subsequence, to a function $v_{\alpha}$, in $L^{2}(\Omega; L^{2}(0,T; \mathbb{L}^{2}))$. Additionally, by Lemma~\ref{lemma discrete diff filter Hh}-$(i)$, we have $\mathbb{P}$-a.s. and a.e. in $(0,T)\times D$ that $\mathcal{V}^{+}_{k,h} = \mathcal{U}^{+}_{k,h} - \alpha^{2}\Delta^{h}\mathcal{U}^{+}_{k,h}$. By relying on what we have proven so far in this proof, one gets the identification $v_{\alpha} = u_{\alpha} - \alpha^{2}\Delta u_{\alpha}$, $\mathbb{P}$-a.s. and a.e. in $(0,T)\times D$. For the last convergence, we need the following estimate
		\begin{equation}\label{eq discrete laplace and gradiant}
			||\Delta^{h}\varphi_{h}||_{\mathbb{L}^{2}} \lesssim h^{-1}||\nabla \varphi_{h}||_{\mathbb{L}^{2}},\ \  \forall \varphi_{h} \in \mathbb{V}_{h}.
		\end{equation}
		This can be illustrated through identity~\eqref{def discrete Laplace}, the inverse estimate~\eqref{eq inverse estimate} and the Cauchy-Schwarz inequality. Since $\{\mathcal{V}^{-}_{k,h}\}_{k,h}$ is bounded in the Hilbert space $L^{2}(\Omega; L^{2}(0,T; \mathbb{L}^{2}))$ , we let $\varphi \in [C_{c}^{\infty}(D)]^{d}$. Subsequently, $$\mathbb{E}\left[\int_{0}^{T}\left(\mathcal{V}^{-}_{k,h}, \varphi\right)dt\right] = \mathbb{E}\left[\int_{0}^{T}\left(\mathcal{V}^{-}_{k,h} - \mathcal{V}^{+}_{k,h}, \varphi\right)dt\right] + \mathbb{E}\left[\int_{0}^{T}\left(\mathcal{V}^{+}_{k,h}, \varphi\right)dt\right] \coloneqq \zeta_{1} + \zeta_{2}.$$ Using \eqref{eq U convergence k-h}$_{2}$, we get the convergence of $\zeta_{2}$. It remains to show that $\zeta_{1}$ vanishes on the limits. Indeed, by Lemmas~\ref{lemma discrete diff filter Hh}-$(i)$, \ref{lemma3.3.1}-$(ii)$, the Cauchy-Schwarz inequality and estimate~\eqref{eq discrete laplace and gradiant}, we get $\displaystyle \zeta_{1} \leq ||\varphi||_{\mathbb{L}^{2}}\mathbb{E}\left[k\sum_{m=1}^{M}(||U^{m} - U^{m-1}||_{\mathbb{L}^{2}} + \alpha^{2}||\Delta^{h}(U^{m} - U^{m-1})||_{\mathbb{L}^{2}})\right] \lesssim ||\varphi||_{\mathbb{L}^{2}}C_{T}k + ||\varphi||_{\mathbb{L}^{2}}\alpha h^{-1}kC_{T} \leq ||\varphi||_{\mathbb{L}^{2}}C_{T}(k + \alpha^{2}k^{1/2}) \xrightarrow[k,h \to 0]{} 0.$
	\end{proof}
	
	Propositions~\ref{prop limits of U and V} and \ref{prop convergence U k-h} give an insight into the limiting functions spaces. Yet, they do not provide the divergence-free property and the relationship between $u$ and $v$. The following proposition treats this issue.
	\begin{prop}\label{prop divergence-free}\ \\
		Let $u, v$ and $u_{\alpha}$ be the limiting functions provided in Propositions~\ref{prop limits of U and V} and \ref{prop convergence U k-h}. Then, $u$ and $u_{\alpha}$ are divergence-free, and $v = u$ almost everywhere in $(0,T)\times D$ and $\mathbb{P}$-almost surely.
	\end{prop}
	\begin{proof}
		To prove that $u$ and $u_{\alpha}$ are divergence-free, we show that $\{div \ \mathcal{U}^{+}_{k,h}\}_{k,h}$ converges weakly in $L^{2}(\Omega; L^{2}(0,T; \mathbb{L}^{2}))$ toward $0$, thanks to \eqref{eq U convergence}$_{1}$ and \eqref{eq U convergence k-h}$_{1}$. To this end, we evoke the Lagrange interpolation $\mathcal{I} \colon C^{\ell}(D) \to L_{h}$, $\ell \in \mathbb{N}$ (c.f. \cite[Theorem 4.4.4]{brenner2007mathematical}). For $\ell\in\mathbb{N}$, let $z \in C^{\ell}(D)$, then 
		\begin{align*}
			\mathbb{E}\left[\int_{0}^{T}\left(div\mathcal{U}^{+}_{k,h}, z\right)dt\right] &= \mathbb{E}\left[\int_{0}^{T}\left(div\mathcal{U}^{+}_{k,h}, z - \mathcal{I}z\right)dt\right] + \mathbb{E}\left[ \int_{0}^{T}\left(div\mathcal{U}^{+}_{k,h}, \mathcal{I}z\right)dt\right] \\&\lesssim \mathbb{E}\left[\int_{0}^{T}||\nabla \mathcal{U}^{+}_{k,h}||_{L^{2}}dt\right]||z - \mathcal{I}z||_{L^{2}} \xrightarrow[k,h \to 0]{}0,
		\end{align*}
		where the second term in the first equality vanishes because $\{\mathcal{U}^{+}_{k,h}\}_{k,h}$ is weakly divergence-free. Consider $\varphi \in \mathbb{L}^{2}$. By Lemma~\ref{lemma3.3.1}-$(ii)$, the embedding $L^{2} \hookrightarrow L^{1}$, and $\alpha \leq \mathcal{C}h$, one gets $\left|\mathbb{E}\left[\int_{0}^{T}\left(\alpha^{2}\Delta^{h}\mathcal{U}_{k,h}^{+}, \varphi\right)dt\right]\right| \lesssim C_{T}\mathcal{C}h||\varphi||_{\mathbb{L}^{2}} \to 0$. Therefore, $\alpha^{2}\Delta^{h}\mathcal{U}_{k,h}^{+} \rightharpoonup 0$ in $L^{2}(\Omega; L^{2}(0,T; \mathbb{L}^{2}))$ if $\alpha \leq \mathcal{C}h$. By Lemma~\ref{lemma discrete diff filter Hh}-$(i)$ and Proposition~\ref{prop limits of U and V}, we infer that $u = v, \ \mathbb{P}$-almost surely and almost everywhere in $(0,T)\times D$.
	\end{proof}
	
	In the below proposition, we identify the limits of deterministic and stochastic integrals of Algorithm~\ref{Algorithm}.
	\begin{prop}\label{prop limits identification}\ \\
		Let $t \in [0,T]$. For all $\varphi \in \mathbb{H}^{2}\cap\mathbb{V}$,
		\begin{enumerate}
			\item if $0 < \alpha \leq \mathcal{C}h$, for some $\mathcal{C}>0$ independent of $\alpha, h$ and $k$, then
			\begin{enumerate}
				\item[(i)] $\displaystyle \lim\limits_{k,h \to 0}\mathbb{E}\left[\int_{0}^{t}\left(\nabla \mathcal{V}^{+}_{k,h}, \nabla \Pi_{h}\varphi\right)ds\right] = \mathbb{E}\left[\int_{0}^{t}\left(\nabla v(s), \nabla \varphi\right)ds\right]$,
				\item[(ii)] $\displaystyle \lim\limits_{k,h \to 0} \mathbb{E}\left[\int_{0}^{t}\tilde{b}\left(\mathcal{U}_{k,h}^{+}, \mathcal{V}_{k,h}^{-}, \Pi_{h}\varphi\right)ds\right] = \mathbb{E}\left[\int_{0}^{t}\left([v(s)\cdot\nabla]v(s),\varphi\right)ds\right]$,
				\item[(iii)] $\displaystyle \lim\limits_{k,h \to 0}\mathbb{E}\left[\int_{0}^{t}\big\langle f^{-}(s, \mathcal{U}^{-}_{k,h}), \Pi_{h}\varphi \big\rangle ds\right] = \mathbb{E}\left[\int_{0}^{t}\big\langle f(s, v(s)), \varphi \big\rangle ds\right]$,
				\item[(iv)] $\displaystyle \lim\limits_{k,h \to 0} \Big(\int_{0}^{t}g^{-}\big(s, \mathcal{U}_{k,h}^{-}\big)dW(s), \Pi_{h}\varphi\Big) = \Big(\int_{0}^{t}g(s,v(s))dW(s), \varphi\Big), \mathbb{P}\mbox{-}a.s.$,
			\end{enumerate}
			\item if $\frac{\sqrt{k}}{h} < L \leq \alpha$, for some $L>0$  independent of $\alpha, h$ and $k$, then
			\begin{enumerate}
				\item[(i)] $\displaystyle \displaystyle \lim\limits_{k,h \to 0} \mathbb{E}\left[\int_{0}^{t}\tilde{b}\left(\mathcal{U}_{k,h}^{+}, \mathcal{V}_{k,h}^{-}, \Pi_{h}\varphi\right)ds\right] = \mathbb{E}\left[\int_{0}^{t}\tilde{b}(u_{\alpha}(s), v_{\alpha}(s), \varphi)ds\right]$,
				\item[(ii)] Assertions $1.(iii)$-$(iv)$ remain valid, provided that $v$ is replaced by $u_{\alpha}$.
			\end{enumerate}
		\end{enumerate}
	\end{prop}
	\begin{proof}
		Fix $\varphi$ in $\mathbb{H}^{2}\cap\mathbb{V}$. Starting out with assertion $1$, we have $\{\mathcal{V}^{+}_{k,h}\}_{k,h}$ (resp. $\{\Pi_{h}\varphi\}_{h}$) is weakly (resp. strongly) convergent in $L^{2}(\Omega; L^{2}(0,T; \mathbb{H}_{0}^{1}))$ (resp. $\mathbb{H}_{0}^{1}$), thanks to Proposition~\ref{prop limits of U and V} and estimate~\eqref{eqprojection}. Therefore, $1.(i)$ follows. Moreover, $$\tilde{b}(\mathcal{U}^{+}_{k,h}, \mathcal{V}^{-}_{k,h},\Pi_{h}\varphi) = \left([\mathcal{U}^{+}_{k,h}\cdot\nabla]\mathcal{V}^{-}_{k,h},\Pi_{h}\varphi\right) + \left((\nabla \mathcal{U}^{+}_{k,h})^{T}\cdot\mathcal{V}_{k,h}^{-},\Pi_{h}\varphi\right) \eqqcolon J_{1} + J_{2}.$$ We have $J_{1} = -\left([\mathcal{U}_{k,h}^{+}\cdot\nabla]\Pi_{h}\varphi, \mathcal{V}^{-}_{k,h}\right) - \left([div \ \mathcal{U}^{+}_{k,h}]\mathcal{V}^{-}_{k,h}, \Pi_{h}\varphi\right) \eqqcolon -J^{1}_{1} - J_{1}^{2}$ by integrating by parts. Therefore, by Proposition~\ref{prop trilinear form}-$(iii)$ and estimate~\eqref{eqprojection}, 
		\begin{align*}
			&\left|\mathbb{E}\left[\int_{0}^{t}\left(-J_{1}^{1} + \left([u\cdot\nabla]\varphi, v\right)\right)ds\right]\right| \leq C_{D}Ch||\varphi||_{\mathbb{H}^{2}}\mathbb{E}\left[\int_{0}^{t}||\nabla\mathcal{U}^{+}_{k,h}||_{\mathbb{L}^{2}}||\nabla\mathcal{V}^{-}_{k,h}||_{\mathbb{L}^{2}}ds\right] \\& + \left|\mathbb{E}\left[\int_{0}^{t}\left\{-\left([\mathcal{U}^{+}_{k,h}\cdot\nabla]\varphi, \mathcal{V}^{-}_{k,h}\right) + \left([u\cdot\nabla]\varphi, v\right)\right\}ds\right]\right|,
		\end{align*} 
		The first term on the right hand side tends to $0$ after applying the Cauchy-Shwarz inequality, Lemmas~\ref{lemma3.3.1} and \ref{lemma V a priori estimates}. Similarly, the second term goes to $0$ by virtue of Proposition~\ref{prop limits of U and V}. Moreover, $J_{1}^{2} = \left([div \ \mathcal{U}^{+}_{k,h}]\mathcal{V}^{-}_{k,h}, \Pi_{h}\varphi - \varphi\right) + \left(\mathcal{V}^{-}_{k,h}, [div \ \mathcal{U}^{+}_{k,h}]\varphi\right)$. Thus, $\mathbb{E}\left[\int_{0}^{t}|J_{1}^{2}|ds\right]$ goes to $0$ after applying Hölder's inequality, embedding $\mathbb{H}^{1} \hookrightarrow \mathbb{L}^{4}$, Poincar{\'e}'s inequality, Lemmas~\ref{lemma3.3.1}, \ref{lemma V a priori estimates}, estimate~\eqref{eqprojection},convergence~\eqref{eq U convergence}$_{4}$ and Proposition~\ref{prop divergence-free}. On the other hand, by virtue of Proposition~\ref{prop trilinear form}-$(iv)$, $J_{2}=-\int_{D}[\Pi_{h}\varphi\cdot\nabla]\mathcal{V}^{-}_{k,h}\mathcal{U}^{+}_{k,h}dx$. Thus,
		\begin{align*}
			\left|\mathbb{E}\left[\int_{0}^{t}J_{2}ds\right]\right| &\leq C_{D}Ch||\varphi||_{\mathbb{H}^{2}}\mathbb{E}\left[\int_{0}^{t}||\nabla\mathcal{V}_{k,h}^{-}||_{\mathbb{L}^{2}}||\nabla\mathcal{U}^{+}_{k,h}||_{\mathbb{L}^{2}}ds\right] \\&+ \left|\mathbb{E}\left[\int_{0}^{t}\left([\varphi\cdot\nabla]\mathcal{V}^{-}_{k,h}, \mathcal{U}^{+}_{k,h}\right)ds\right]\right|.
		\end{align*}
		The first term on the right converges toward $0$, thanks to Lemmas~\ref{lemma3.3.1} and \ref{lemma V a priori estimates}. By Proposition~\ref{prop limits of U and V}, the second term goes to $\left|\mathbb{E}\left[\int_{0}^{t}\left([\varphi\cdot\nabla]v,u\right)ds\right]\right|$ which vanishes because $u = v$, $\mathbb{P}$-a.s. and a.e. in $(0,T)\times D$ (see Proposition~\ref{prop divergence-free}) and $\varphi$ is divergence-free. Subsequently, the trilinear term $\tilde{b}$ tends to $-\mathbb{E}\left[\int_{0}^{t}\left([u\cdot\nabla]\varphi,v\right)ds\right]$, which by virtue of Proposition~\ref{prop divergence-free} and an integration by parts, lead to the sought term. For assertion $1.(iii)$, we make use of assumption $(S_{2})$:
		\begin{align*}
			&\left|\langle f^{-}(s, \mathcal{U}^{-}_{k,h}), \Pi_{h}\varphi \rangle - \langle f(s, u(s)), \varphi\rangle\right| \leq (K_{3} + K_{4}||\mathcal{U}^{-}_{k,h}||_{\alpha})||\Pi_{h}\varphi - \varphi||_{\mathbb{H}^{1}} \\&+ ||f(s,u(s)) - f^{-}(s, u(s))||_{\mathbb{H}^{-1}}||\varphi||_{\mathbb{H}^{1}} + L_{f}||u(s) - \mathcal{U}^{-}_{k,h}||_{\alpha}||\varphi||_{\mathbb{H}^{1}} \eqqcolon \mathcal{L}_{1} + \mathcal{L}_{2} + \mathcal{L}_{3}.
		\end{align*}
		We have $\mathbb{E}\left[\int_{0}^{t}\mathcal{L}_{1}ds\right] \leq Ch||\varphi||_{\mathbb{H}^{2}}(TK_{3} + K_{4}C_{T}) \to 0$, thanks to estimate~\eqref{eqprojection} and Lemma~\ref{lemma3.3.1}. Besides, $\mathbb{E}\left[\int_{0}^{t}\mathcal{L}_{2}ds\right] \to 0$, thanks to the continuity of $f$ with respect to $t$. Lastly, we decompose the norm $||\cdot||_{\alpha}$ after using the embedding $\mathbb{L}^{2}\hookrightarrow \mathbb{L}^{1}$: $$\mathbb{E}\left[\int_{0}^{t}\mathcal{L}_{3}dt\right] \lesssim L_{f}||\varphi||_{\mathbb{H}^{1}}\mathbb{E}\left[\int_{0}^{t}(||u(s) - \mathcal{U}^{-}_{k,h}||^{2}_{\mathbb{L}^{2}} + \alpha^{2}||\nabla(u(s) - \mathcal{U}^{-}_{k,h})||^{2}_{\mathbb{L}^{2}})ds\right]^{1/2}.$$ Convergence of $\mathbb{E}\left[\int_{0}^{t}||u(s) - \mathcal{U}^{-}_{k,h}||^{2}_{\mathbb{L}^{2}}ds\right]$ is handled by adding and subtracting $\mathcal{U}^{+}_{k,h}$, then by employing the triangular inequality along with convergence~\eqref{eq U convergence}$_{2}$ and Lemma~\ref{lemma3.3.1}-$(ii)$. Furthermore, $\mathbb{E}\left[\int_{0}^{t}\alpha^{2}||\nabla(u(s) - \mathcal{U}^{-}_{k,h})||^{2}_{\mathbb{L}^{2}}ds\right]^{1/2} \lesssim C_{T}\mathcal{C}h$, due to Lemma~\ref{lemma3.3.1}-$(ii)$, convergence~\eqref{eq U convergence}$_{1}$ and $\alpha \leq \mathcal{C}h$. We now justify assertion~$1.(iv)$. Let us denote by
		\begin{align*}
			&\mathcal{J}\coloneqq\left(\int_{0}^{t}g^{-}\big(s,\mathcal{U}_{k,h}^{-}\big)dW(s), \Pi_{h}\varphi\right) - \left(\int_{0}^{t}g(s,u(s))dW(s), \varphi\right) \\&= \Big(\int_{0}^{t}g^{-}\big(s,\mathcal{U}_{k,h}^{-}\big)dW(s), \Pi_{h}\varphi - \varphi\Big) + \Big(\int_{0}^{t}\left(g^{-}\big(s,\mathcal{U}_{k,h}^{-}\big) - g^{-}(s,u(s))\right)dW(s), \varphi\Big) \\&\hspace{10pt}+ \Big(\int_{0}^{t}\left(g^{-}(s,u(s)) - g(s, u(s))\right)dW(s), \varphi\Big) \eqqcolon \mathcal{J}_{1} + \mathcal{J}_{2} + \mathcal{J}_{3}.
		\end{align*}
		After squaring both sides and applying the mathematical expectation, we shall bound each term separately. To this end, assumption $(S_{2})$, Itô's isometry,the Cauchy-Schwarz inequality, estimate~\eqref{eqprojection}, and Lemma~\ref{lemma3.3.1} are all needed in the calculations below.
		\begin{align*}
			&\mathbb{E}\left[\int_{0}^{t}\mathcal{J}_{1}^{2}ds\right] \lesssim Tr(Q)||\Pi_{h}\varphi - \varphi||_{\mathbb{L}^{2}}^{2}\mathbb{E}\left[\int_{0}^{t}(K_{1}+K_{2}||\mathcal{U}^{-}_{k,h}||_{\alpha})^{2}ds\right] \lesssim h^{4}C_{T} \to 0, \\&
			\mathbb{E}\left[\int_{0}^{t}\mathcal{J}_{2}^{2}ds\right] \lesssim Tr(Q)||\varphi||^{2}_{\mathbb{L}^{2}}L_{g}^{2}\mathbb{E}\left[\int_{0}^{t}||\mathcal{U}^{-}_{k,h} - u(s)||^{2}_{\alpha}ds\right] \to 0 \ \ (\mbox{ similar to } \mathcal{L}_{3}) \mbox{ and, } \\&
			\mathbb{E}\left[\int_{0}^{t}\mathcal{J}_{3}^{2}ds\right] \lesssim Tr(Q)||\varphi||^{2}_{\mathbb{L}^{2}}\mathbb{E}\left[\sum_{m=1}^{M}\int_{t_{m-1}}^{t_{m}}||g(t_{m-1},u) - g(s, u)||^{2}_{\mathscr{L}_{2}(K,\mathbb{L}^{2})}ds\right] \to 0,
		\end{align*}
		by the continuity of $g$ with respect to its first variable. Replacing $u$ by $v$ (Proposition~\ref{prop divergence-free}) completes the proof of $(iv)$. Moving on to assertion $2$. We have
		\begin{align*}
			\tilde{b}(\mathcal{U}^{+}_{k,h}, \mathcal{V}^{-}_{k,h}, \Pi_{h}\varphi) - \tilde{b}(u_{\alpha}, v_{\alpha}, \varphi) &= \tilde{b}(\mathcal{U}^{+}_{k,h}, \mathcal{V}^{-}_{k,h}, \Pi_{h}\varphi - \varphi) + \Big(\tilde{b}(\mathcal{U}^{+}_{k,h}, \mathcal{V}^{-}_{k,h}, \varphi) - \tilde{b}(u_{\alpha}, v_{\alpha}, \varphi)\Big) \\&\eqqcolon I_{1} + I_{2}.
		\end{align*}
		By Proposition~\ref{prop trilinear form}-$(iii)$, estimate~\eqref{eqprojection}, the inverse~\eqref{eq inverse estimate} and sum~\eqref{eq sum estimate} inequalities, Young's inequality, Lemmas~\ref{lemma3.3.1},\ref{lemma a priori estimate V alpha indep.} and $\sqrt{k}/h < \alpha$, one gets
		\begin{align*}
			&\mathbb{E}\left[\int_{0}^{t}|I_{1}|ds\right] \leq \mathbb{E}\left[k\sum_{m=1}^{M}|I_{1}|\right] \lesssim h||\varphi||_{\mathbb{H}^{2}}\mathbb{E}\left[\alpha k^{1/2}\max\limits_{1 \leq m \leq M}||\nabla U^{m}||_{\mathbb{L}^{2}}\sum_{m=1}^{M}||V^{m-1}||_{\mathbb{L}^{2}}\right] \\&\lesssim h||\varphi||_{\mathbb{H}^{2}}\mathbb{E}\left[\max\limits_{1 \leq m \leq M}||U^{m}||_{\alpha}^{2} + 3k\sum_{m=1}^{M}||V^{m-1}||_{\mathbb{L}^{2}}^{2}\right] \leq (C_{T} + C'_{T})h||\varphi||_{\mathbb{H}^{2}} \xrightarrow[k,h \to 0]{} 0.
		\end{align*}
		On the other hand, by employing Proposition~\ref{prop trilinear form}-$(iv)$ and integrating by parts twice, there holds $\tilde{b}(\mathcal{U}^{+}_{k,h}, \mathcal{V}^{-}_{k,h}, \varphi) = \left([\varphi\cdot\nabla]\mathcal{U}^{+}_{k,h}, \mathcal{V}_{k,h}^{-}\right) - \left([\mathcal{U}^{+}_{k,h}\cdot\nabla]\varphi, \mathcal{V}^{-}_{k,h}\right)$. Proposition~\ref{prop convergence U k-h} ensures the convergence $\left|\mathbb{E}\left[\int_{0}^{t}I_{2}ds\right]\right| \xrightarrow[k,h \to 0]{} 0$, while taking into account that $\left([\varphi\cdot\nabla]u_{\alpha}, v_{\alpha}\right) - \left([u_{\alpha}\cdot\nabla]\varphi, v_{\alpha}\right)$ coincides $\mathbb{P}$-a.s. with $\tilde{b}(u_{\alpha}, v_{\alpha}, \varphi)$ after integrating it twice by parts and using the null divergence of $\varphi$ and $u_{\alpha}$ (see Proposition~\ref{prop divergence-free}). We point out that the low regularity of $v_{\alpha}$ (in $\mathbb{L}^{2}$) does not prevent $\tilde{b}(u_{\alpha}, v_{\alpha}. \phi)$ from being well-defined, because $u_{\alpha}$ and $\phi$ have high regularities. The proof of convergence for $f$ and $g$ is similar to that of assertion $1$. One may only need to replace the employed estimate $\alpha \leq \mathcal{C}h$ (when decomposing the norm $||\cdot||_{\alpha}$ in the above steps) by the strong convergence of $\{\mathcal{U}^{+}_{k,h}\}_{k,h}$ in $L^{2}(\Omega; L^{2}(0,T; \mathbb{H}_{0}^{1}))$, thanks to Proposition~\ref{prop convergence U k-h}.
	\end{proof}
	
	\section{Conclusion}\label{section conclusion}
	We devote this section to results and conclusions emerging from what we have carried out so far.
	
	\subsubsection{Convergence of LANS-$\alpha$ to NS in $2$D}
	Assume $d = 2$, and $V^{0} \to v_{0}$ in $L^{2}(\Omega; \mathbb{L}^{2})$ as $h \to 0$. From Propositions~\ref{prop divergence-free} and \ref{prop limits identification}-$(1)$, we infer that $v$ satisfies equation~\eqref{eq def weak formulation NS}, $\mathbb{P}$-a.s., for all $t \in [0,T]$. Moreover, by a standard technique (e.g. \cite{Pardoux1975}), it is easy to check that $v \in L^{2}(\Omega; C([0,T]; \mathbb{H}))$. We also have, by Proposition~\ref{prop limits of U and V}-$(1)$, that $v \in M^{2}_{\mathcal{F}_{t}}(0,T; \mathbb{V})$.
	
	\subsubsection{Convergence to the LANS-$\alpha$}
	Assume $d \in \{2, 3\}$, and $U^{0} \to \bar{u}_{0}$ in $L^{2}(\Omega; \mathbb{H}^{1})$ as $h \to 0$.
	According to Propositions~\ref{prop convergence U k-h} and \ref{prop limits identification}-$(2)$, one may notice that we still need to illustrate the convergence of $\mathbb{E}\left[\int_{0}^{T}\left(\nabla \mathcal{V}^{+}_{k,h}, \nabla\varphi_{h}\right)dt\right]$, $\varphi_{h} \in \mathbb{V}_{h}$ toward its continuous counterpart. To this end, we define, for $z \in \mathbb{H}_{0}^{1}$, the elliptic projection $E_{h} \colon \mathbb{H}_{0}^{1} \to \mathbb{V}_{h}$ as the unique solution of $$\left(\nabla E_{h}z, \nabla \varphi_{h}\right) = \left(\nabla z, \nabla \varphi_{h}\right), \ \ \ \forall \varphi_{h} \in \mathbb{V}_{h}.$$
	Operator $E_{h}$ satisfies $\displaystyle \Delta^{h}E_{h}z = \Pi_{h}\Delta z$ for all $z \in \mathbb{H}^{2}\cap \mathbb{V}$ (e.g. \cite[Page 593]{guermond2008stability}).
	Therefore, for all $\varphi \in \mathbb{H}^{2}\cap \mathbb{V}$, it follows from identities~
	\ref{def projection}, \ref{def discrete Laplace}, Proposition~\ref{prop convergence U k-h} and the above relation that
	\begin{align*}
		\left(\nabla \mathcal{V}^{+}_{k,h}, \nabla E_{h}\varphi\right) = - \left(\mathcal{V}^{+}_{k,h}, \Delta^{h}E_{h}\varphi\right) = -\left(\mathcal{V}^{+}_{k,h}, \Pi_{h}\Delta \varphi\right) = -\left(\mathcal{V}^{+}_{k,h}, \Delta \varphi\right).
	\end{align*}
	Subsequently, $\mathbb{E}\left[\int_{0}^{T}\left(\nabla \mathcal{V}^{+}_{k,h}, \nabla E_{h}\varphi\right)dt\right]$ converges toward $$-\mathbb{E}\left[\int_{0}^{T}\left(v_{\alpha}(t), \Delta\varphi\right)dt\right] = - \mathbb{E}\left[\int_{0}^{T}\left(v_{\alpha}(t), A\varphi\right)dt\right] =  \mathbb{E}\left[\int_{0}^{T}\left( u_{\alpha}(t) + \alpha^{2}Au_{\alpha}(t), A\varphi\right)dt\right].$$ Putting it all together yields the sought result. In other words, $u_{\alpha}$ satisfies equation~\eqref{eq definition solution}, $\mathbb{P}$-a.s., for all $t \in [0,T]$ and $u_{\alpha} \in M^{2}_{\mathcal{F}_{t}}(0,T; D(A)) \cap L^{2}(\Omega; L^{\infty}(0,T; \mathbb{V}))$ together with the fact that $u_{\alpha}$ is weakly continuous with values in $\mathbb{V}$.
	
	\subsection{Numerical experiments}
	This part is devoted to giving computational experiments in $2$D for the stochastic LANS-$\alpha$ model through Algorithm~\ref{Algorithm} when the spatial scale $\alpha$ fulfills either $\alpha \leq Ch$ or $\alpha > L \geq \frac{\sqrt{k}}{h}$. Since our primary objective is to compare solutions' behavior of LANS-$\alpha$ to that of Navier-Stokes, we provide simulation of solutions to the latter equations as well through a linear scheme covered in \cite[Algorithm 3]{brzezniak2013finite}. The implementation hereafter is performed using the open source finite element software FEniCS~\cite{LoggMardalEtAl2012a}. We employ the lower order Taylor-Hood ($P_{2}$-$P_{1}$) element for the spatial discretization within a mixed finite element framework. The chosen domain is a unit square $D = (0,1)^{2}$ along the time interval $[0, T]$ with $T=1$. The initial condition $\bar{u}_{0}$ and viscosity $\nu$ are set to $0$ and $1$, respectively. For the sake of simplicity, the source term $f$ is considered as a deterministic constant and the drift term $g$ plays the identity operator role.
	\paragraph{\underline{Q-Wiener process approximation}} For computational purposes, we must deal with a truncated form of the series \eqref{def Wiener process}. Besides, we consider two independent $H_{0}^{1}(D)$-valued Wiener processes $W_{1}$ and $W_{2}$ such that $W = (W_{1}, W_{2})$. For $M \in \mathbb{N}\backslash \{0\}$, the utilized increments are expressed by $$\Delta_{m}W_{\ell} \approx k^{1/2}\sum_{i,j=1}^{M}(\lambda_{i,j}^{\ell})^{1/2}\xi_{i,j}^{\ell, m}e_{i,j}, \ \ \ \ell \in \{1, 2\},$$ where for all $i,j \in \mathbb{N}$ and $(x, y) \in D$, the basis elements $e_{i,j} \coloneqq 2\sin(i\pi x)\sin(j\pi y)$ represent the Laplace eigenfunctions with Dirichlet boundary conditions on $D$. For $\ell \in \{1,2\}$, $\{\{\xi_{i,j}^{\ell, m}\}_{i, j}\}_{m}$ is a family of independent identically distributed standard normal random variables, $\displaystyle\lambda_{i,j}^{\ell} \coloneqq \frac{1}{(i + j)^{2}}$ for $\ell \in \{1, 2\}$ and $M = 10$.
	
	\paragraph{Case $\alpha \leq Ch$}\ \\
	Consider $\alpha = 10^{-3}h$, $h \approx 0.03$ and $k = 10^{-3}$. 
	\begin{figure}[ht]
		\begin{subfigure}[b]{0.5\linewidth}
			\centering
			\includegraphics[width=\linewidth]{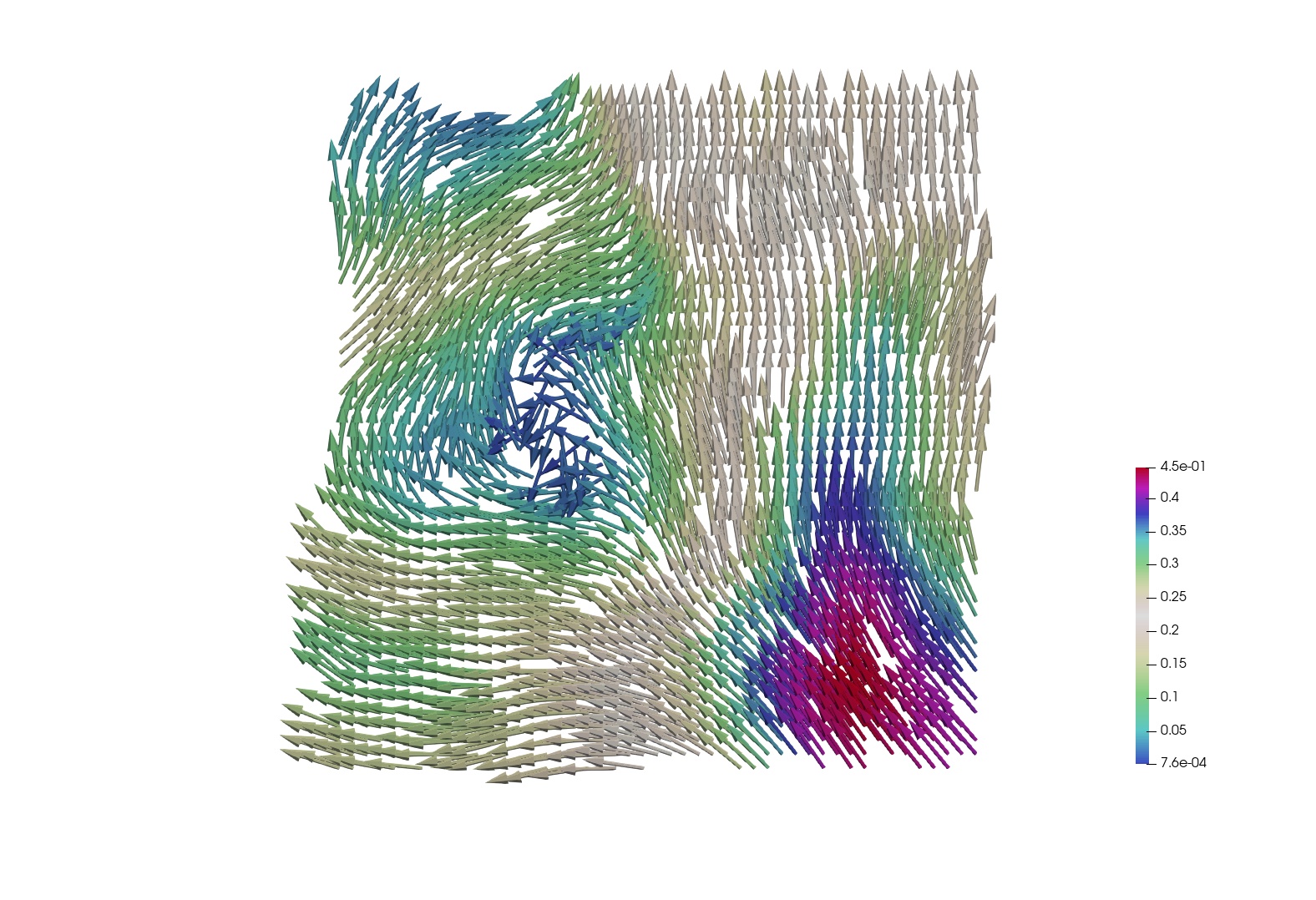}
			\caption{Velocity field of LANS-$\alpha$ at time $t = 0.161$}
		\end{subfigure}
		\begin{subfigure}[b]{0.5\linewidth}
			\centering
			\includegraphics[width=\linewidth]{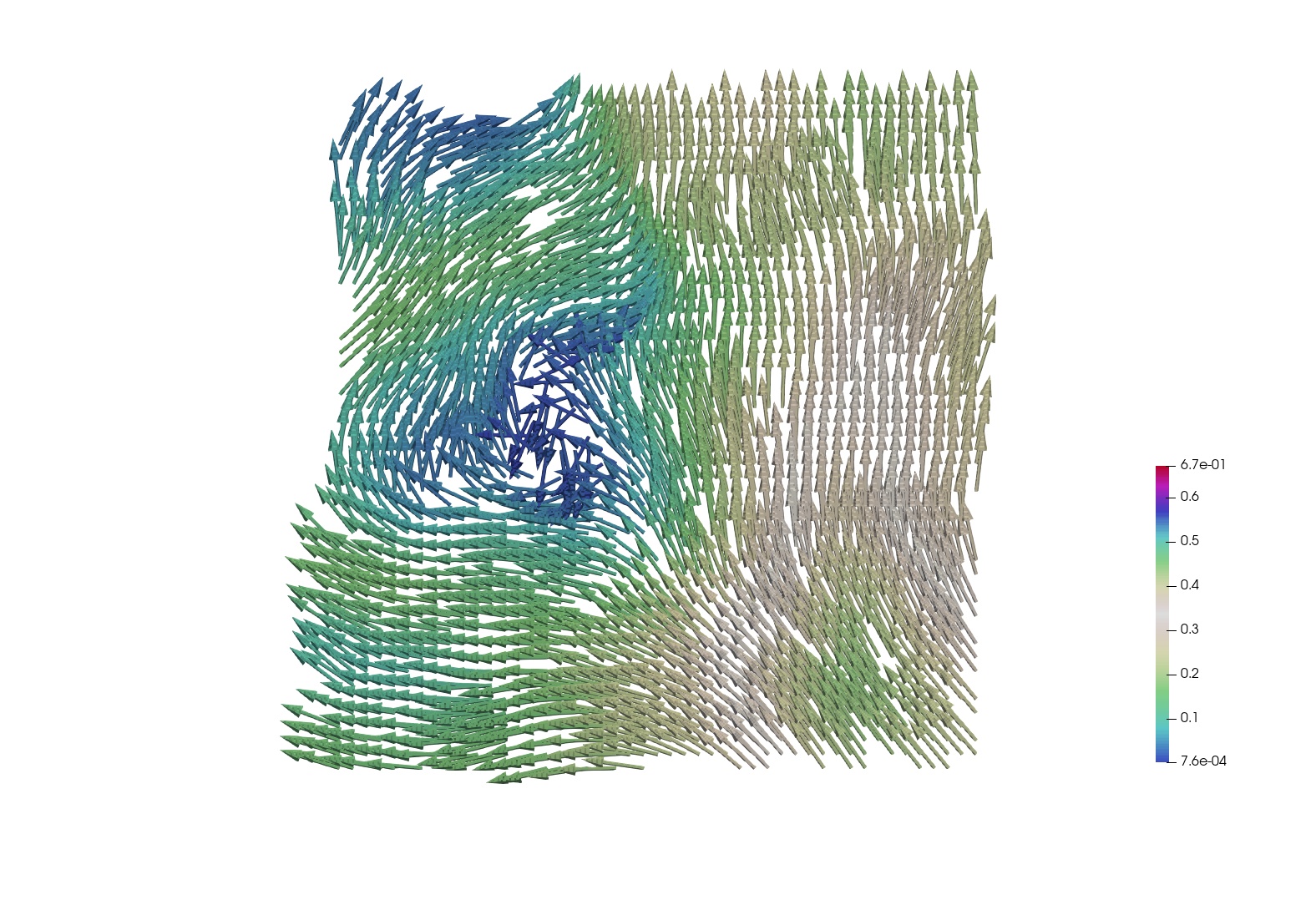}
			\caption{Velocity field of NS at time $t = 0.161$}
		\end{subfigure}
		\begin{subfigure}{0.5\linewidth}
			\centering
			\includegraphics[width=\linewidth]{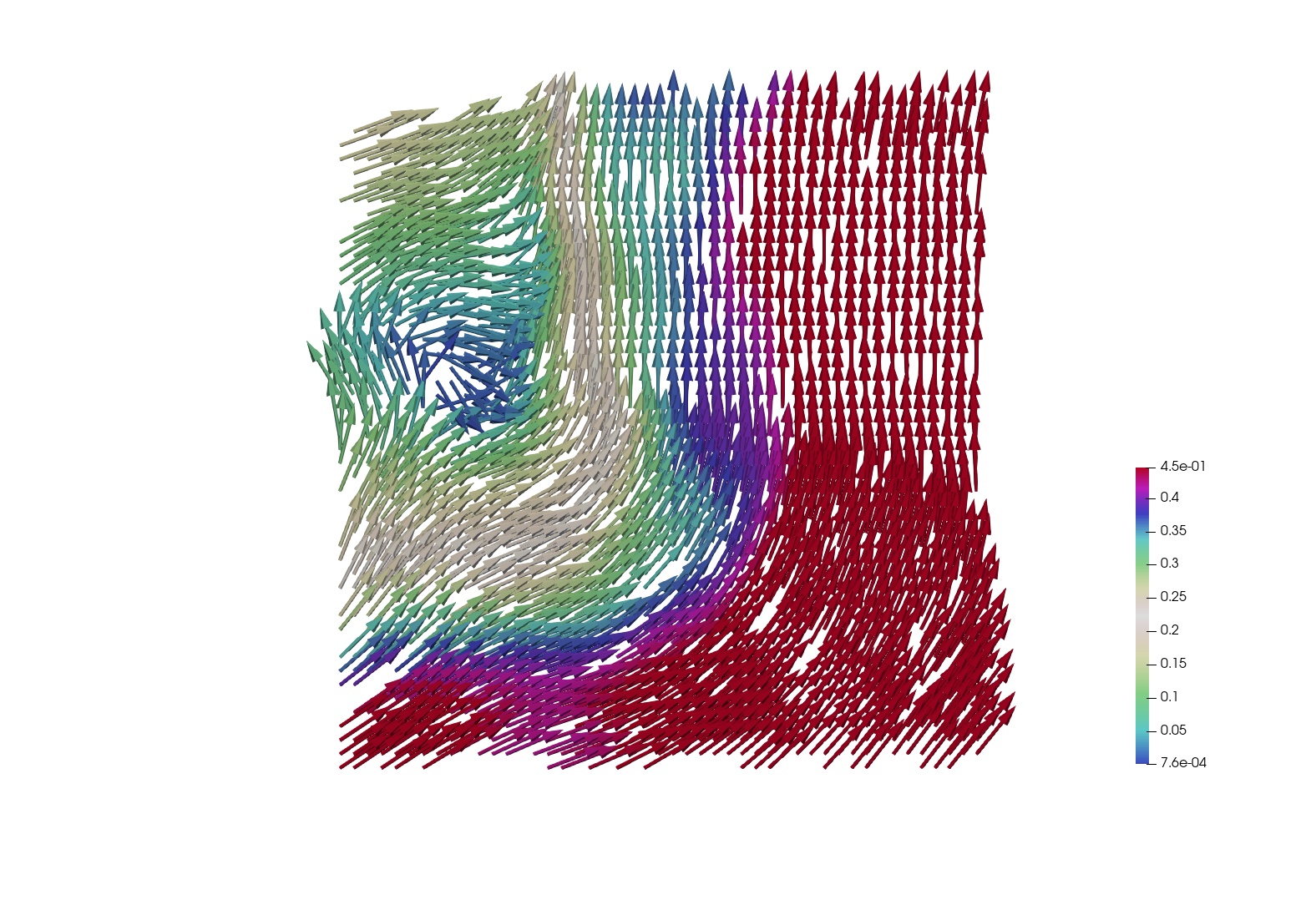}\
			\caption{Velocity field of LANS-$\alpha$ at time $t = 0.586$}
		\end{subfigure}%
		\begin{subfigure}{0.5\linewidth}
			\centering	\includegraphics[width=\linewidth]{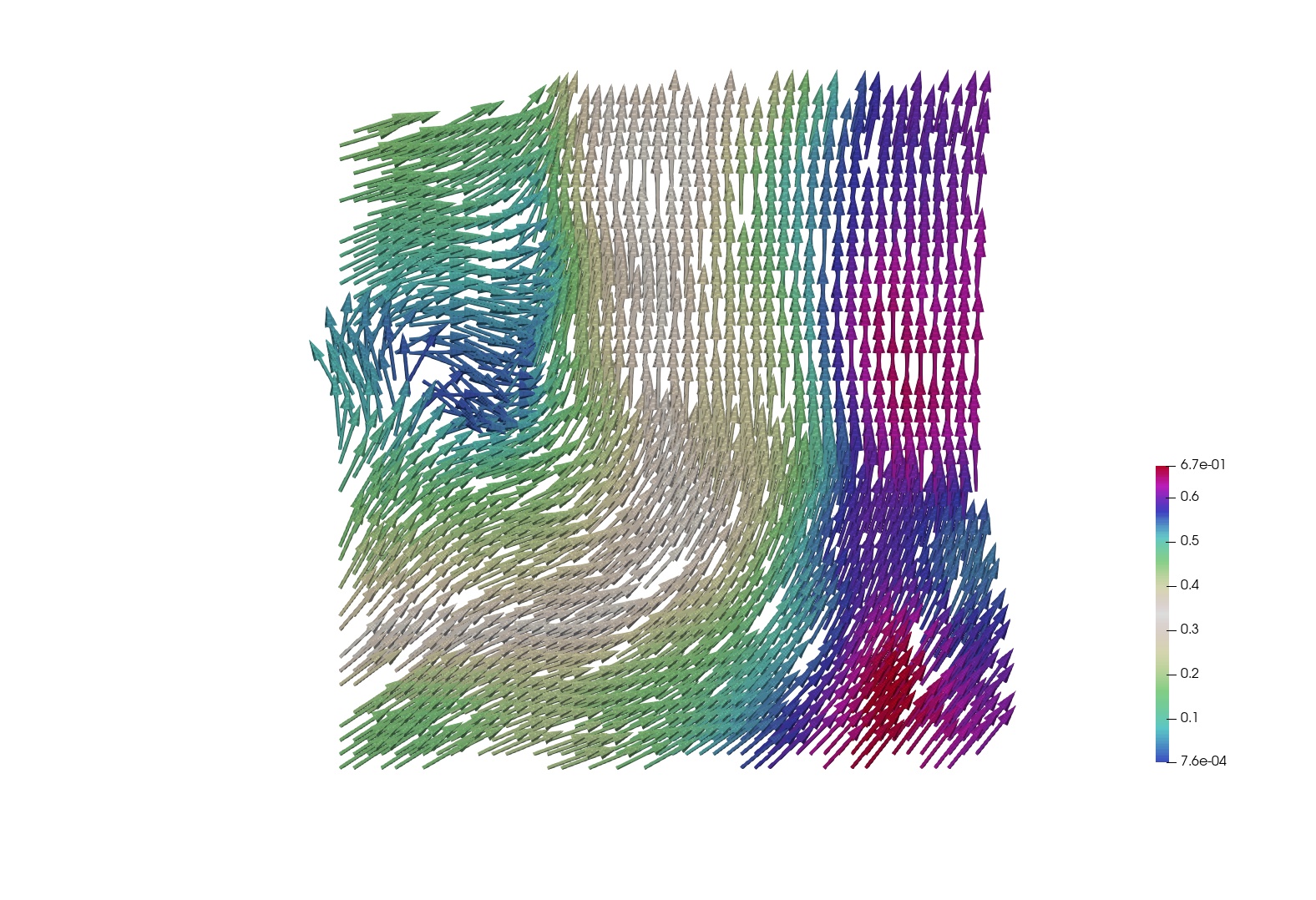}
			\caption{Velocity field of NS at time $t = 0.586$}
		\end{subfigure}
	\end{figure}
	
	Since this case relates both equations~\eqref{main equation} and \eqref{Navier-Stokes}, we choose two different time values in $[0,T]$, and plot the associated figures side by side. This allows us to compare the solutions' behavior together with the occurring differences. Observe that both LANS-$\alpha$ and NS solutions behave similarly with a slight variation in values. Such a difference was expected since we are dealing here with approximate computations, not to mention the considered space discretization's step $h$ which is not too close to $0$, yet its code execution is costly. We also provide the following pressure figures
	\begin{figure}[ht]
		\centering
		\begin{subfigure}{0.4\linewidth}
			\centering
			\includegraphics[width=\linewidth]{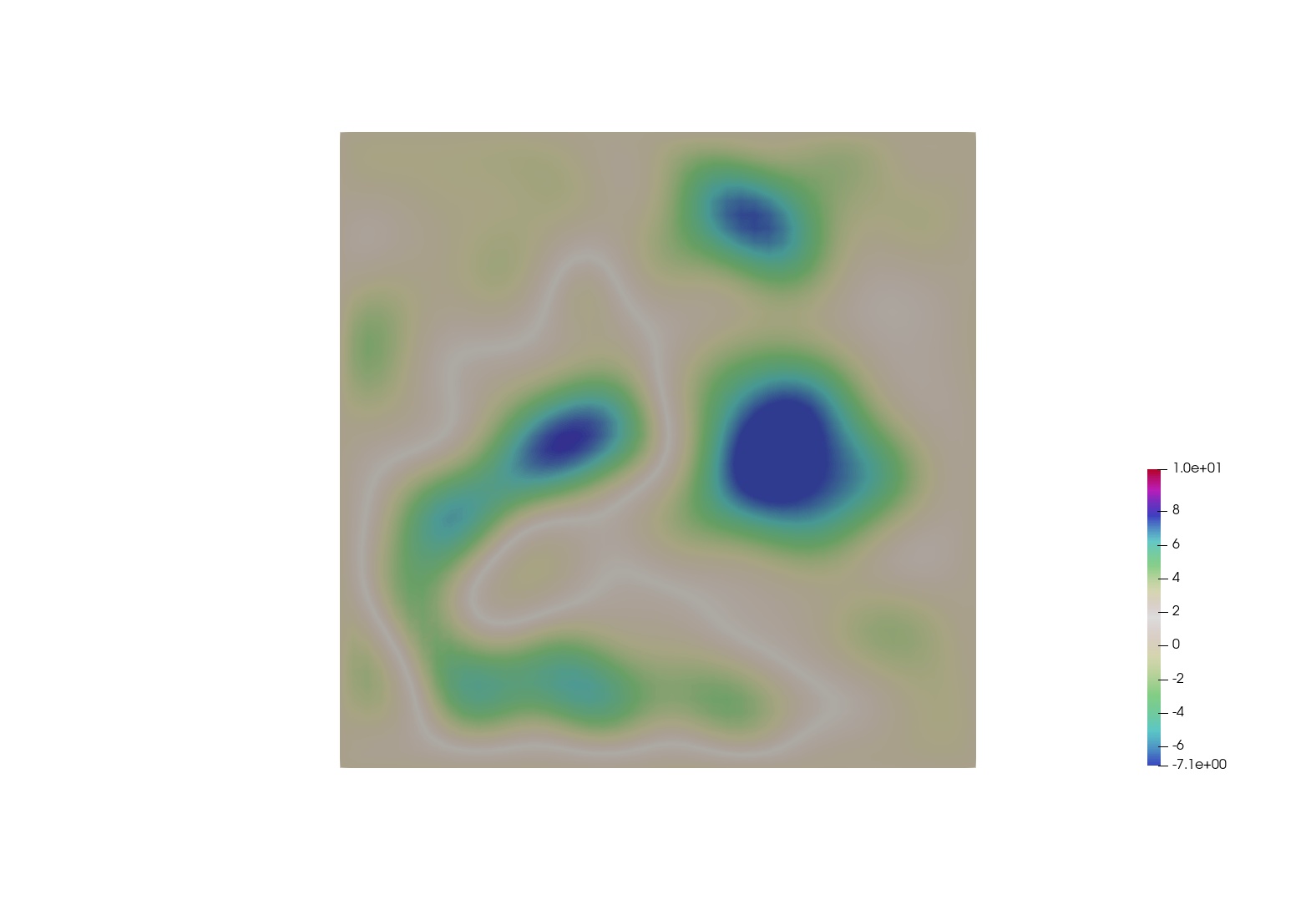}
			\caption{LANS-$\alpha$ pressure at $t=0.161$}
		\end{subfigure}
		\begin{subfigure}{0.4\linewidth}
			\centering
			\includegraphics[width=\linewidth]{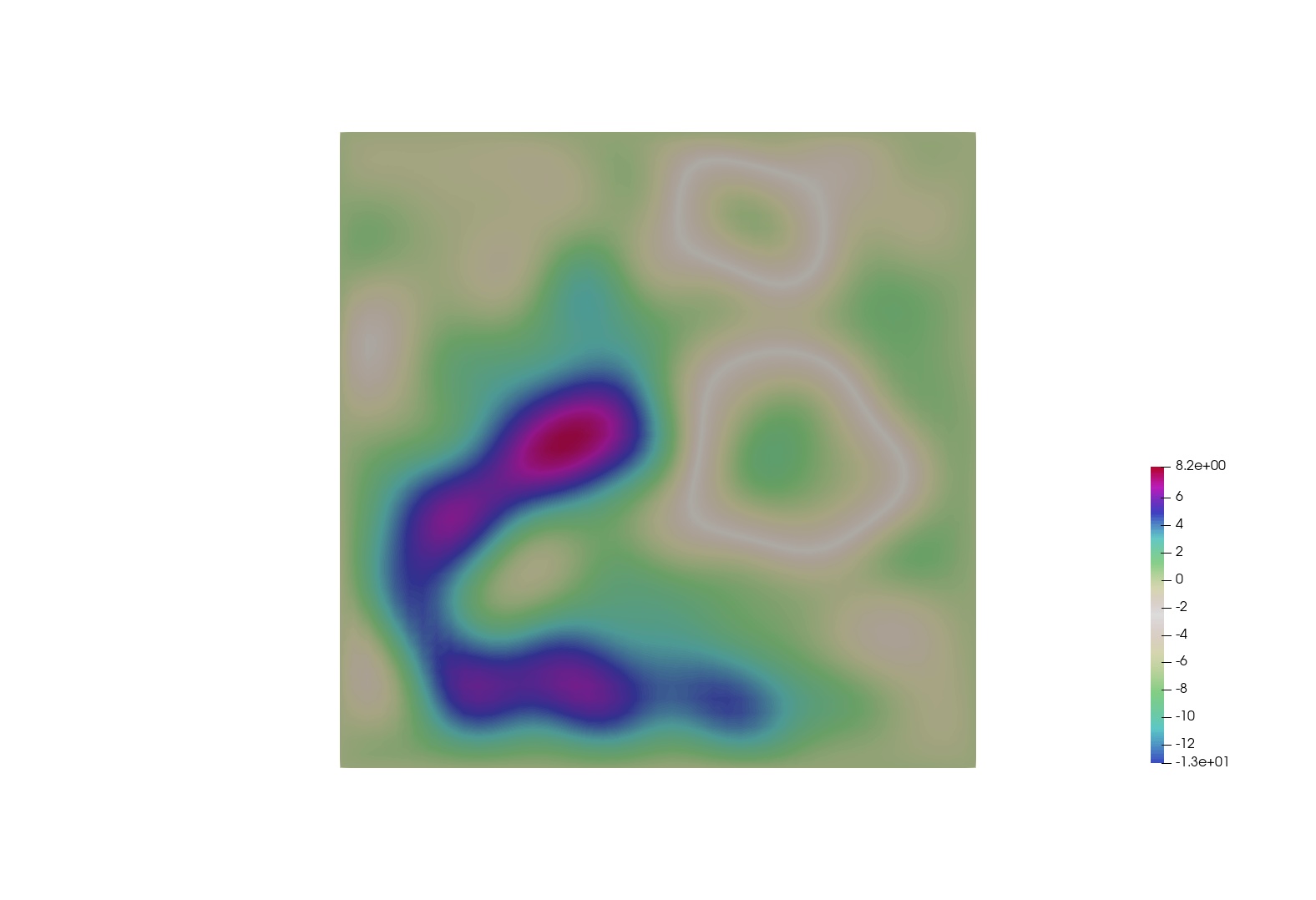}
			\caption{NS pressure at $t = 0.161$}
		\end{subfigure}
		\begin{subfigure}{0.4\linewidth}
			\centering
			\includegraphics[width=\linewidth]{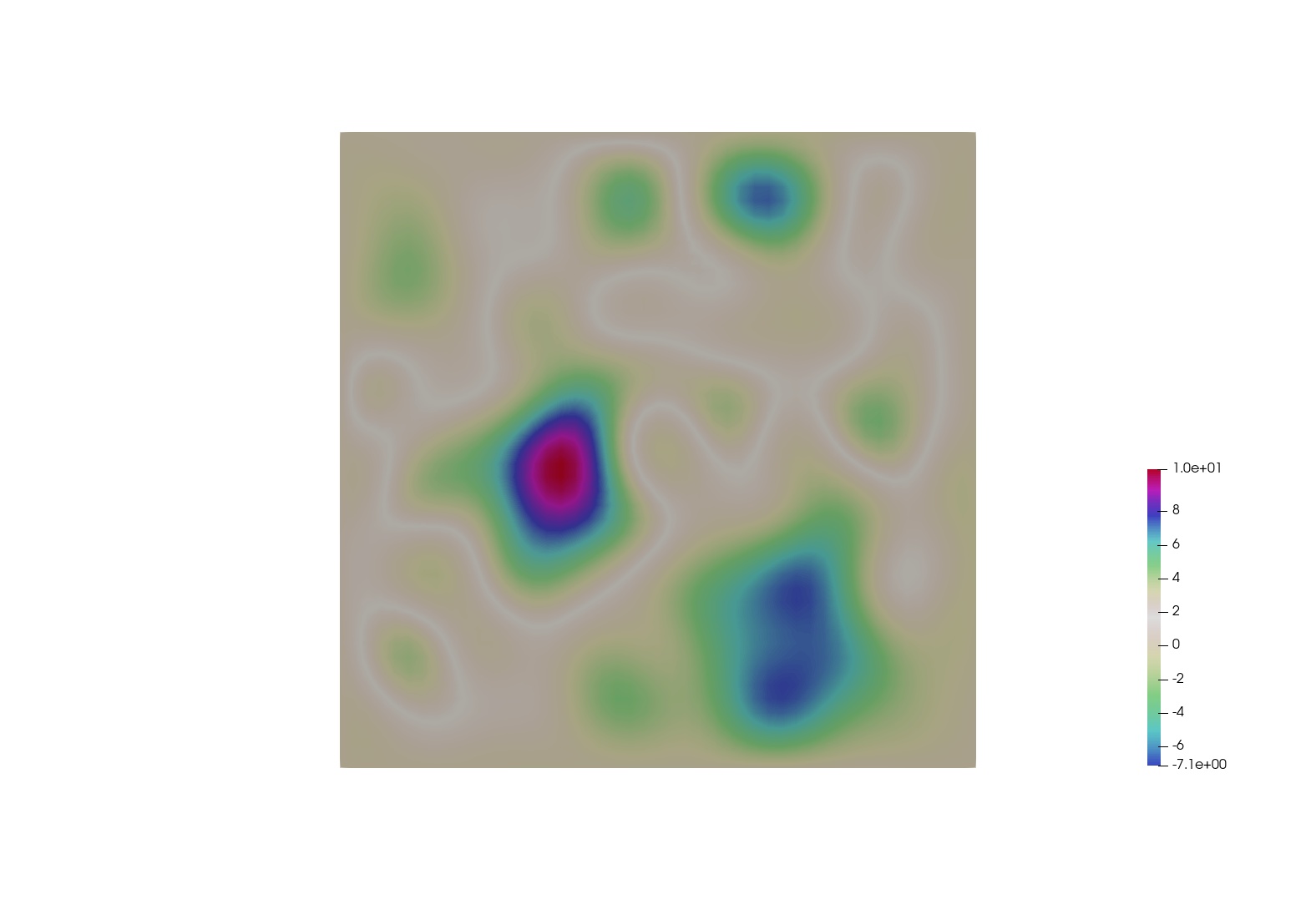}
			\caption{LANS-$\alpha$ pressure at $t = 0.586$}
		\end{subfigure}
		\begin{subfigure}{0.4\linewidth}
			\centering
			\includegraphics[width=\linewidth]{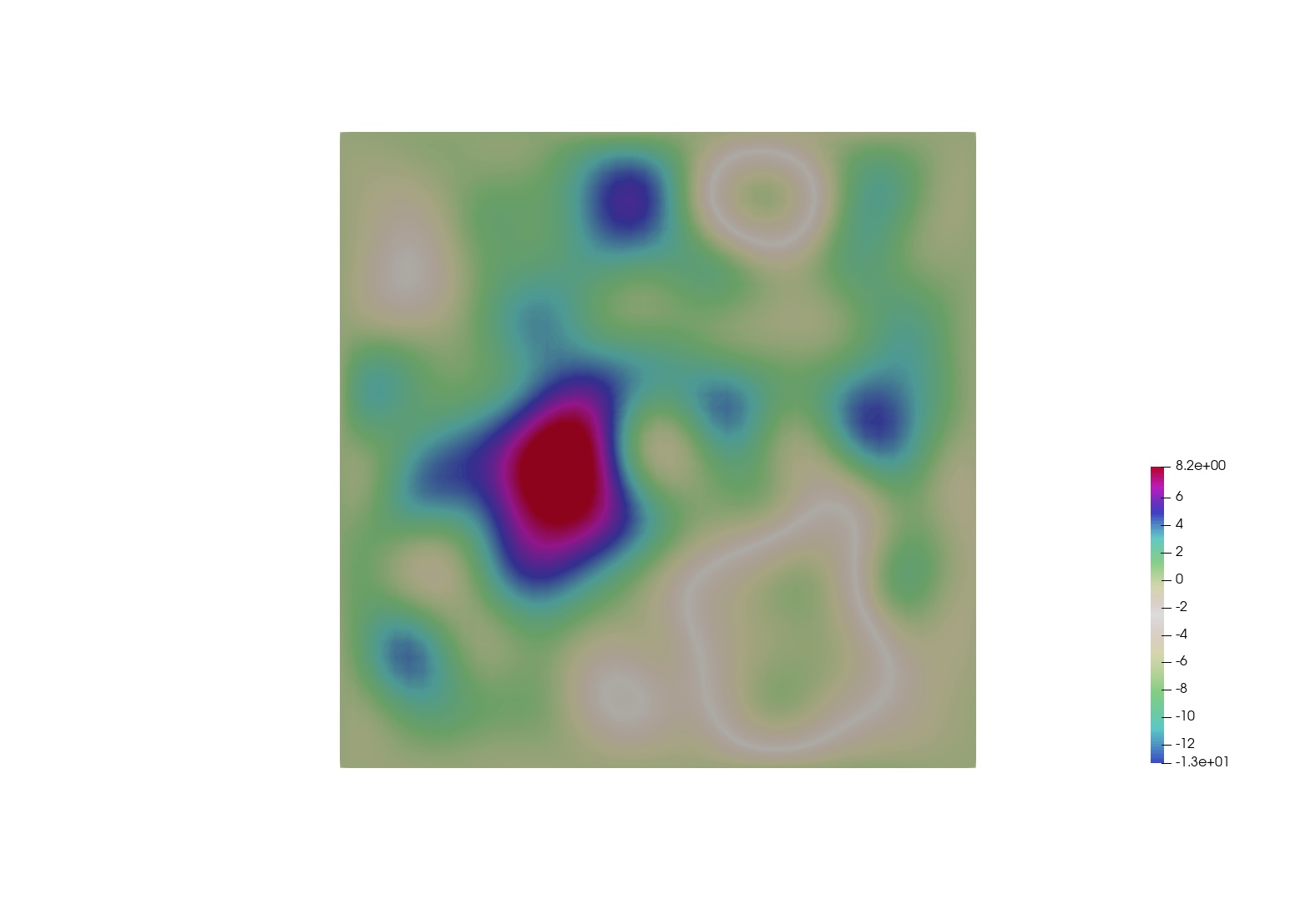}
			\caption{NS pressure at $t = 0.586$}
		\end{subfigure}
	\end{figure}
	
	\paragraph{Case $\alpha \geq L > \sqrt{k}/h$}\ \\
	For this framework, we set $\alpha = 1$, $h \approx 0.03$ and take $k$ in terms of $h$; namely $k = 0.9h^{2}$, so that the condition $\sqrt{k}/h < \alpha$ is met.
	\begin{figure}[ht]
		\centering
		\captionsetup[subfigure]{oneside,margin={-1.5cm,0cm}}
		\begin{subfigure}{0.36\linewidth}
			\hspace{-25pt}\includegraphics[width=\linewidth]{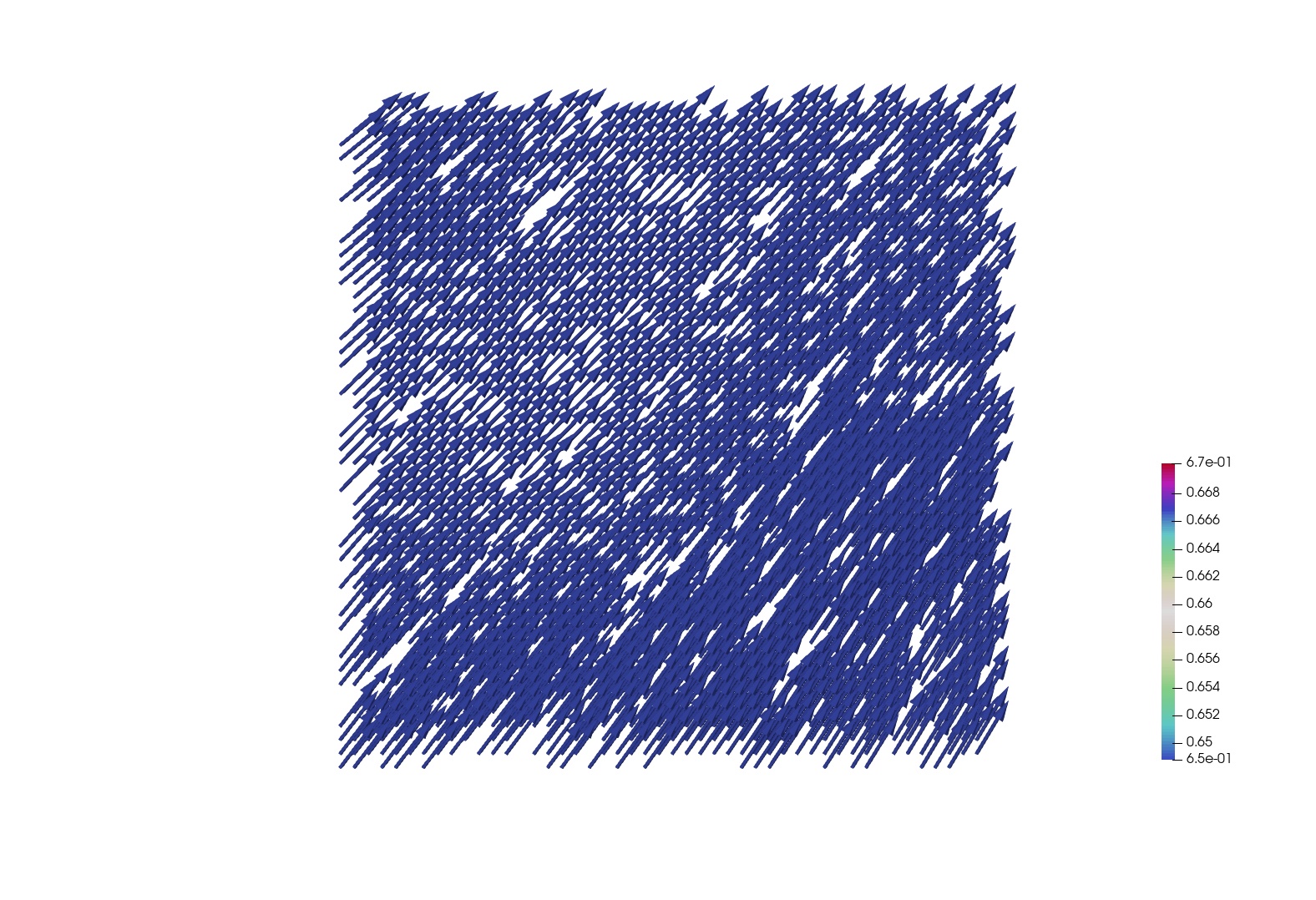}
			\caption{Time $t = 0.161$}
		\end{subfigure}
		\begin{subfigure}{0.36\linewidth}
			\centering
			\hspace{-55pt}\includegraphics[width=\linewidth]{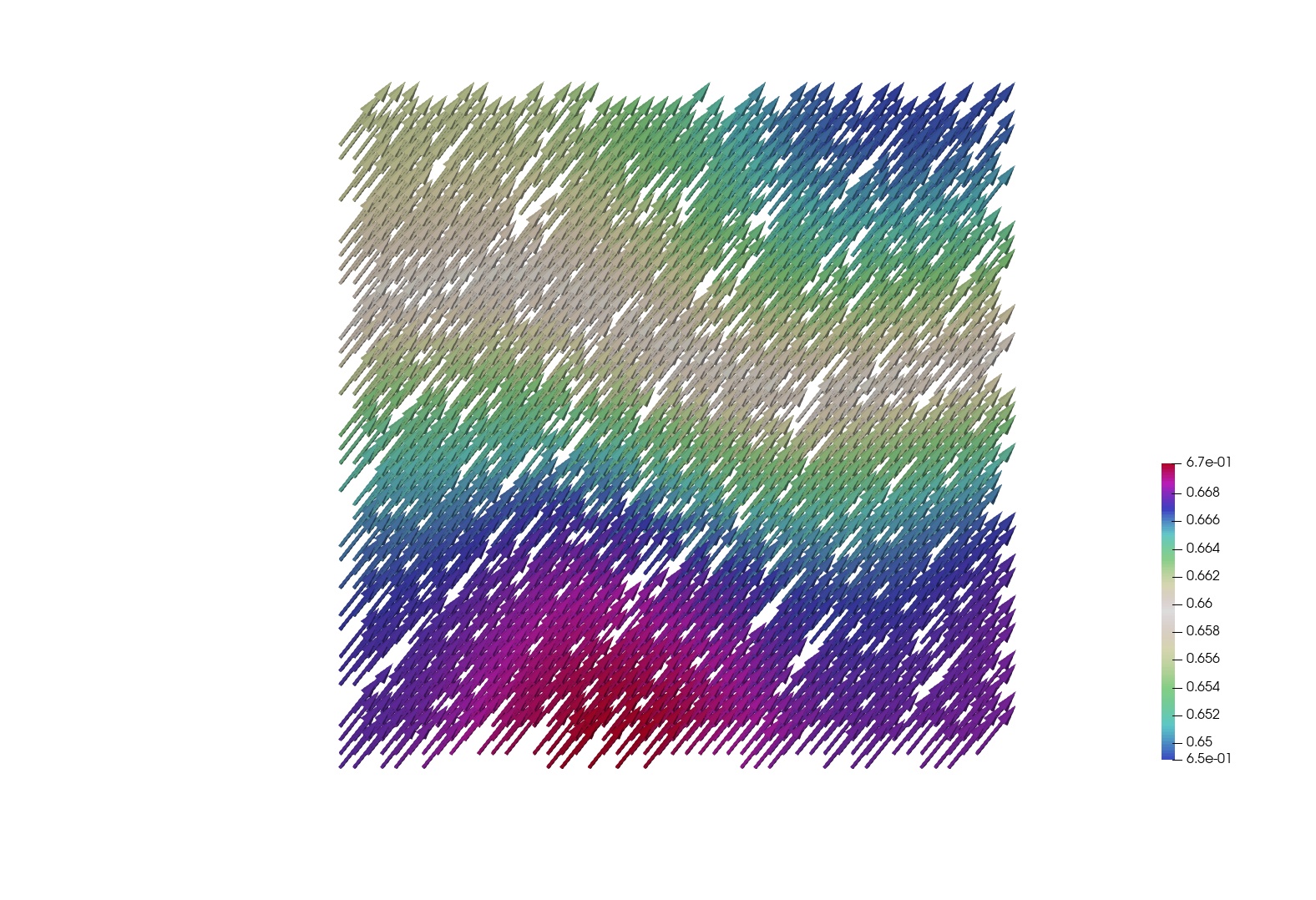}
			\caption{Time $t = 0.586$}
		\end{subfigure}
		\begin{subfigure}{0.36\linewidth}
			\centering
			\hspace{-45pt}\includegraphics[width=\linewidth]{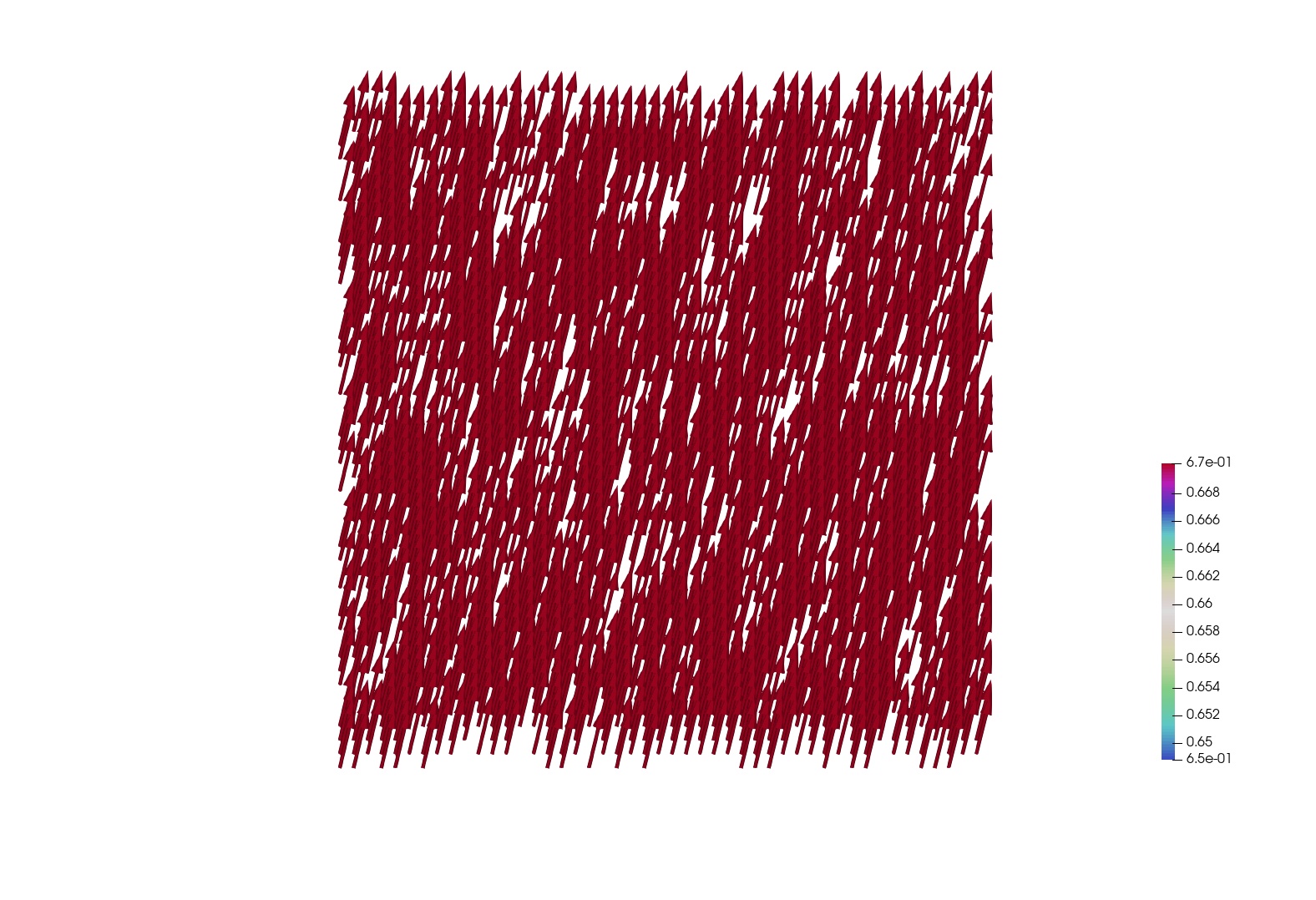}
			\caption{Time $t = 1$}
		\end{subfigure}
		\caption*{Velocity field of LANS-$\alpha$ evaluated at different time values}
	\end{figure}
	
	Within one realization, outcomes of the current case are clearly well-behaved, which means a higher regularity in terms of the velocity fields. It is worth mentioning the speed variation stage within the time interval $[0.45 - \varepsilon, 0.65 + \varepsilon]$, $\varepsilon << 1$ where the velocity value goes from its lowest to its highest rate. Below $t = 0.45$ and above $t = 0.65$, the velocity field maintains almost a constant value. Here are the associated pressure figures:
	\begin{figure}[ht]
		\centering
		\captionsetup[subfigure]{oneside,margin={-1.5cm,0cm}}
		\begin{subfigure}{0.36\linewidth}
			\hspace{-25pt}\includegraphics[width=\linewidth]{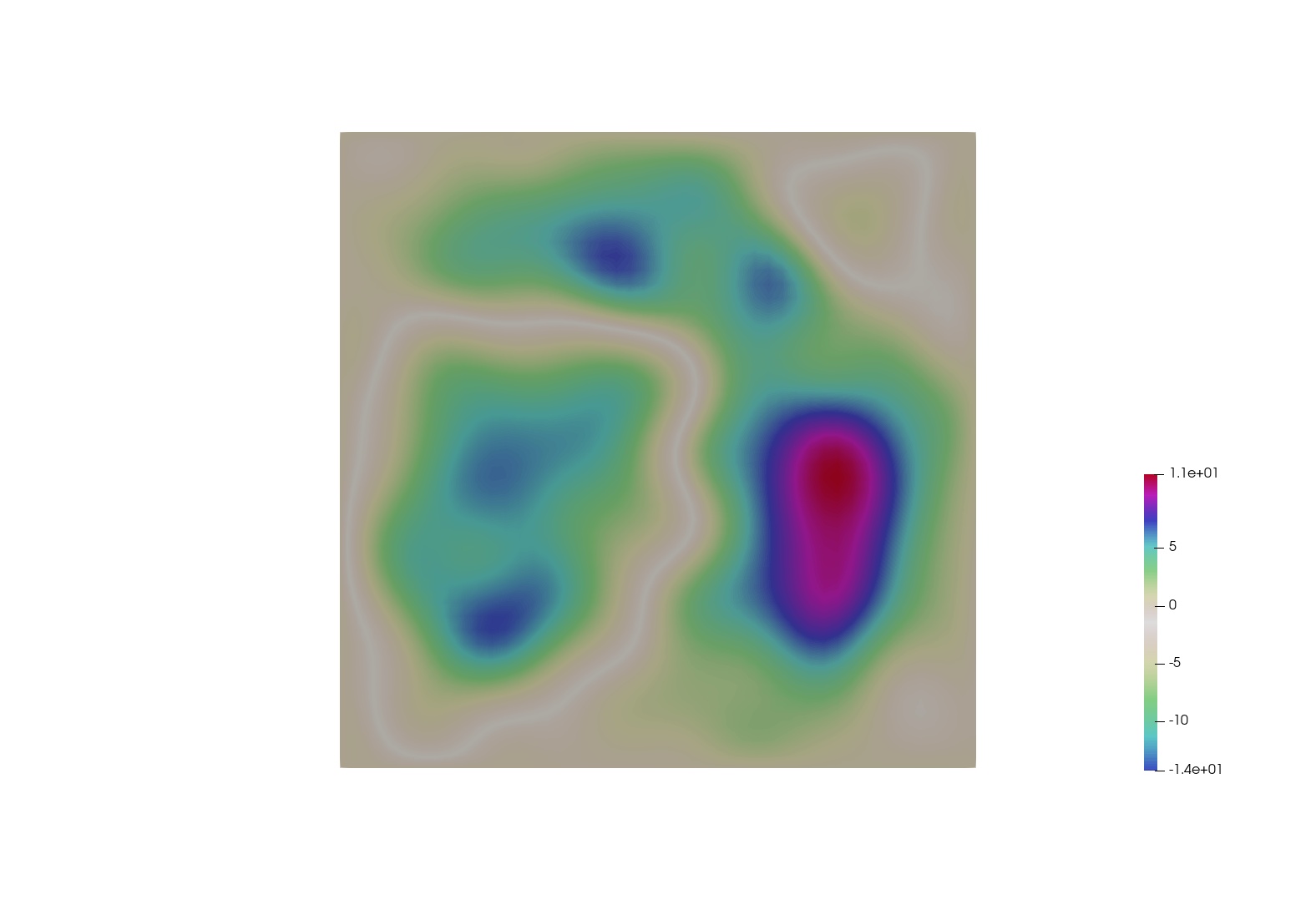}
			\caption{Time $t = 0.161$}
		\end{subfigure}
		\begin{subfigure}{0.36\linewidth}
			\centering
			\hspace{-55pt}\includegraphics[width=\linewidth]{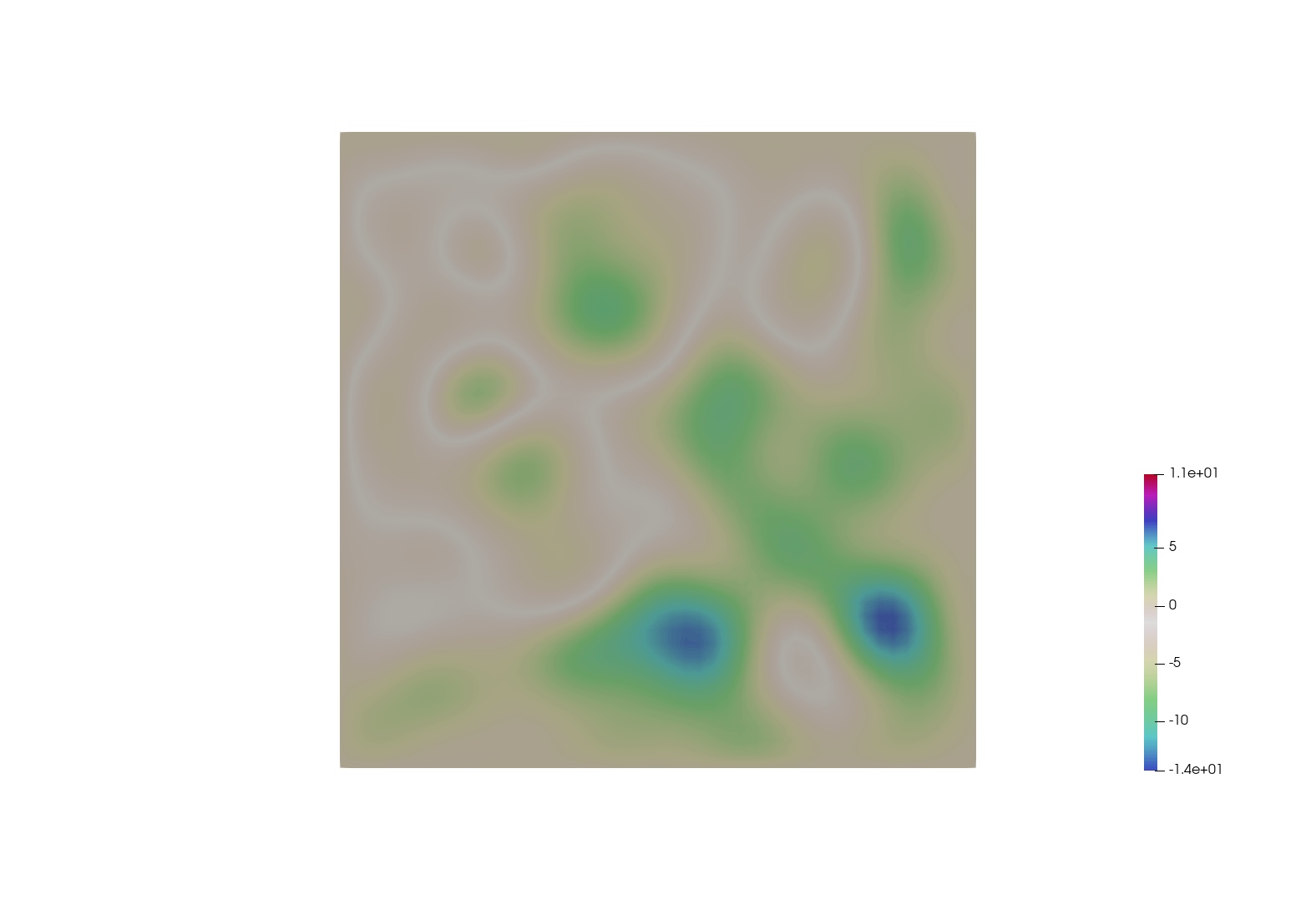}
			\caption{Time $t = 0.586$}
		\end{subfigure}
		\begin{subfigure}{0.36\linewidth}
			\centering
			\hspace{-45pt}\includegraphics[width=\linewidth]{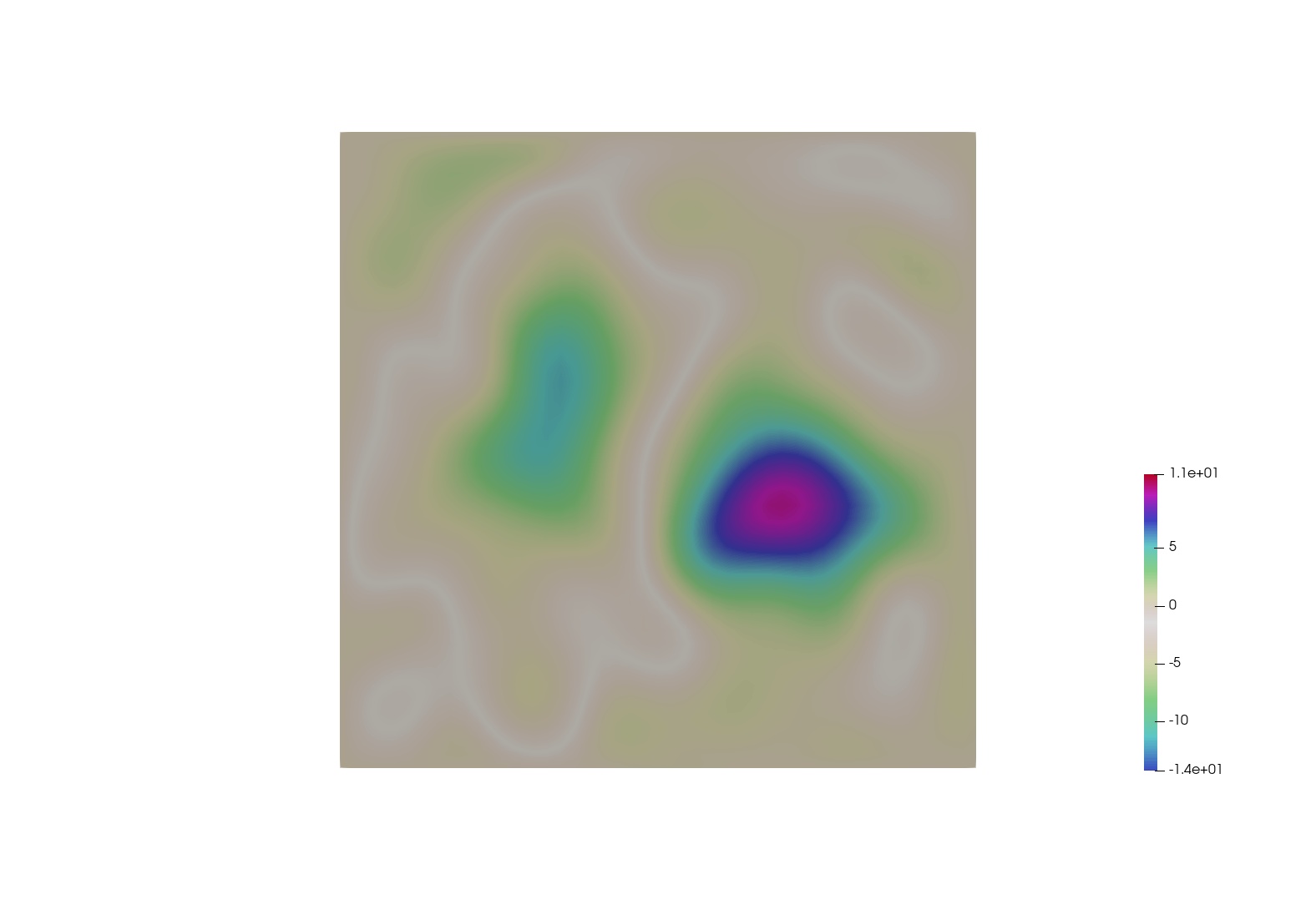}
			\caption{Time $t = 1$}
		\end{subfigure}
		\caption*{Pressure of LANS-$\alpha$ evaluated at different time values}
	\end{figure}
	
	We point out that in both cases, the pressure is heavily impacted by the noise. As its curve progresses in time, we notice a random behavior at each time node. This can be thought of as the stochastic pressure decomposition that was evoked in article~\cite{breit2021convergence} for the two-dimensional stochastic Navier-Stokes equations, which states that $p$ can be split into a few terms, one of which can be written in terms of the Wiener process $W$.
	
	\printbibliography
	
\end{document}